\documentclass[11pt]{amsart}

\usepackage{algorithm}
\usepackage{algpseudocode}
\usepackage{nicefrac}

\usepackage{amssymb,amscd,amsthm,verbatim,amsmath,color,fancyhdr,mathrsfs}
\usepackage{graphicx}
\usepackage{turnstile}
\usepackage{multirow}

\usepackage[top=.95in, bottom = 1 in, left=1in, right = .6in]{geometry}
\usepackage{styfile} 

\newtheorem{theorem}{Theorem}[section]
\newtheorem{proposition}[theorem]{Proposition}
\newtheorem{lemma}[theorem]{Lemma}
\newtheorem{definition}[theorem]{Definition}

\newtheorem{assumption}{Assumption}

\theoremstyle{definition}
\newtheorem{remark}{Remark}[section]
\newtheorem{example}{Example}[section]

\DeclarePairedDelimiterX{\bracket}[3]{#1}{#2}{#3}
\newcommand{\round}[1]{\bracket*{(}{)}{#1}}

\newcommand{\squarebrack}[1]{\bracket*{\lbrack}{\rbrack}{#1}}

\providecommand*{\eu}{\ensuremath{\mathrm{e}}} 

\providecommand{\newoperator}[3]{\newcommand*{#1}{\mathop{#2}#3}}
\providecommand{\renewoperator}[3]{\renewcommand*{#1}{\mathop{#2}#3}}

\renewoperator{\Re}{\mathrm{Re}}{\nolimits}
\renewoperator{\Im}{\mathrm{Im}}{\nolimits}

\providecommand*{\ped}[1]{\ensuremath{_\mathrm{#1}}}

\providecommand*{\slot}[1]{\ifblank{#1}{\,\cdot\,}{#1}}

\DeclarePairedDelimiterXPP{\nrm}[2]{}{\lVert}{\rVert}{\ensuremath{_{#1}}}{\ifblank{#2}{\:\cdot\:}{#2}}
\newcommand{\norm}[2]{\nrm*{#1}{#2}}

\newcommand{\enorm}[1]{\norm{2}{#1}} 
\newcommand{\infnorm}[1]{\norm{\infty}{#1}}

\newcommand{\normF}[1]{\norm{\mathrm{F}}{#1}}

\newcommand{\abs}[1]{\bracket*{\lvert}{\rvert}{#1}}

\newcommand{\inner}[1]{\bracket*{\langle}{\rangle}{#1}}
\newcommand{\ceil}[1]{\bracket*{\lceil}{\rceil}{#1}}
\newcommand{\floor}[1]{\bracket*{\lfloor}{\rfloor}{#1}}

\DeclarePairedDelimiterXPP\prob[1]{\mathbb{P}}{\lbrace}{\rbrace}{}{\renewcommand\given{\nonscript\:\delimsize\vert\nonscript\:\mathopen{}}#1} 
\newcommand{\Prob}[1]{\prob*{#1}} 

\DeclarePairedDelimiterXPP\probability[2]{\mathbb{P}_{#1}}{\lbrace}{\rbrace}{}{\renewcommand\given{\nonscript\:\delimsize\vert\nonscript\:\mathopen{}}#2} 

\DeclarePairedDelimiterXPP\expectation[1]{\mathbb{E}}{\lbrack}{\rbrack}{}{\renewcommand\given{\nonscript\:\delimsize\vert\nonscript\:\mathopen{}}#1} 
\newcommand{\E}[1]{\expectation*{#1}} 

\DeclarePairedDelimiterXPP\expectationdist[2]{\mathbb{E}_{#1}}{\lbrack}{\rbrack}{}{\renewcommand\given{\nonscript\:\delimsize\vert\nonscript\:\mathopen{}}#2} 
\newcommand{\Exp}[2]{\expectationdist*{#1}{#2}} 

\DeclarePairedDelimiterXPP\variance[1]{\mathrm{Var}}{\lbrack}{\rbrack}{}{\renewcommand\given{\nonscript\:\delimsize\vert\nonscript\:\mathopen{}}#1} 

\DeclarePairedDelimiterXPP\variancedist[2]{\mathrm{Var}_{#1}}{\lbrack}{\rbrack}{}{\renewcommand\given{\nonscript\:\delimsize\vert\nonscript\:\mathopen{}}#2} 

\DeclarePairedDelimiterXPP\covariance[2]{\mathrm{Cov}}{(}{)}{}{#1,\mathopen{}#2} 
\newcommand{\Cov}[2]{\covariance*{#1}{#2}} 

\DeclarePairedDelimiterXPP\covariancedist[3]{\mathrm{Cov}_{#1}}{(}{)}{}{#2,\mathopen{}#3} 

\DeclarePairedDelimiterXPP\law[1]{\mathrm{Law}}{(}{)}{}{\renewcommand\given{\nonscript\:\delimsize\vert\nonscript\:\mathopen{}}#1} 
\newcommand{\Law}[1]{\law*{#1}}

\newcommand{\inv}[1]{\frac{1}{#1}}
\newcommand{\indicator}[2]{\mathbbm{1}\ensuremath{_{#1}}\ifblank{#2}{}{\set{#2}}} 

\newoperator{\supp}{\mathrm{supp}}{\nolimits}

\makeatletter
\providecommand*{\diff}
{\@ifnextchar^{\DIfF}{\DIfF^{}}}
\def\DIfF^#1{
	\mathop{\mathrm{\mathstrut d}}
	\nolimits^{#1}\gobblespace}
\def\gobblespace{
	\futurelet\diffarg\opspace}
\def\opspace{
	\let\DiffSpace\!
	\ifx\diffarg(
	\let\DiffSpace\relax
	\else
	\ifx\diffarg[
	\let\DiffSpace\relax
	\else
	\ifx\diffarg\{
	\let\DiffSpace\relax
	\fi\fi\fi\DiffSpace}

\providecommand*{\pdiff}
{\@ifnextchar^{\pDIfF}{\pDIfF^{}}}
\def\pDIfF^#1{
	\mathop{\mathrm{\mathstrut \partial}}
	\nolimits^{#1}\gobblespace}
\def\gobblespace{
	\futurelet\diffarg\opspace}
\def\opspace{
	\let\DiffSpace\!
	\ifx\diffarg(
	\let\DiffSpace\relax
	\else
	\ifx\diffarg[
	\let\DiffSpace\relax
	\else
	\ifx\diffarg\{
	\let\DiffSpace\relax
	\fi\fi\fi\DiffSpace}

\providecommand*{\deriv}[3][]{\frac{\diff^{#1}{#2}}{\diff {#3}^{#1}}}

\makeatother

\providecommand\given{}
\newcommand\SetSymbol[1][]{
	\nonscript\:#1\vert
	\allowbreak
	\nonscript\:
	\mathopen{}}
\DeclarePairedDelimiterX\Set[1]\{\}{
	\renewcommand\given{\SetSymbol[\delimsize]}
	#1
}

\newcommand{\set}[1]{\Set*{#1}}
\newcommand{\interval}[1]{\squarebrack{#1}}

\newcommand{\setinline}[1]{\Set{#1}}

\newcommand{\R}{\mathbb{R}}
\newcommand{\Rd}[1]{\mathbb{R}^{#1}}
\newcommand{\Natural}{\mathbb{N}}

\newcommand{\Integer}{\mathbb{Z}}

\newcommand{\id}{\mathrm{id}}

\newcommand{\G}{\mathbf{G}}
\newcommand{\Hmat}{\mathbf{H}}

\newcommand{\K}{\mathbf{K}}

\newcommand{\N}{\mathbb{N}}

\newcommand{\X}{\mathbf{X}}
\newcommand{\Y}{\mathbf{Y}}

\newcommand{\Bcal}{\mathcal{B}}
\newcommand{\Ccal}{\mathcal{C}}
\newcommand{\Dcal}{\mathcal{D}}
\newcommand{\Ecal}{\mathcal{E}}
\newcommand{\Fcal}{\mathcal{F}}
\newcommand{\Gcal}{\mathcal{G}}
\newcommand{\Hcal}{\mathcal{H}}

\newcommand{\Kcal}{\mathcal{K}}
\newcommand{\Lcal}{\mathcal{L}}
\newcommand{\Mcal}{\mathcal{M}}
\newcommand{\Ncal}{\mathcal{N}}

\newcommand{\Pcal}{\mathcal{P}}

\newcommand{\Scal}{\mathcal{S}}
\newcommand{\Tcal}{\mathcal{T}}

\newcommand{\Vcal}{\mathcal{V}}
\newcommand{\Wcal}{\mathcal{W}}

\newcommand{\Zcal}{\mathcal{Z}}

\newcommand{\Wfrak}{\mathfrak{W}}

\newcommand{\eps}{\epsilon}

\newcommand{\eq}{\begin{equation}}
\newcommand{\en}{\end{equation}}
\newcommand{\SP}[1]{\textcolor{red}{SP:#1}}

\newcommand{\ESBM}[1]{\mathrm{ESBM}{\left[#1\right]}}
\newcommand{\NN}{\mathbb{N}}
\newcommand{\hamil}{\mathcal{H}}
\newcommand{\statesp}{\mathcal{S}}
\newcommand{\mcal}[1]{\mathcal{#1}}

\providecommand*{\Graphons}{{\widehat\Wcal}}
\providecommand*{\mvGraphons}{{\widehat\Wfrak}}
\providecommand*{\cut}{\square}
\newcommand{\cutnorm}[1]{\norm{\cut}{#1}}

\newcommand{\W}{\mathbb{W}}

\newoperator{\tensor}{\otimes}{}
\newcommand{\Ber}[1]{\mathrm{Ber}\round{#1}}

\newcommand{\Sko}[1]{\mathrm{Sko}\round{#1}}

\def\dsub{d_{\sub}}
\def\sub{\mathrm{sub}}
 \def\babs#1{\bigl\vert#1\bigr\vert}

\newcommand{\RaghavS}[1]{{\color{orange}[Somani: #1]}}

\title{Path convergence of Markov chains on large graphs}
\author{Siva Athreya}
\address{Siva Athreya\\ ICTS-TIFR, Survey No. 151, Shivakote, 
Hesaraghatta Hobli,  Bengaluru - 560 089, India and
Indian Statistical Institute, Bangalore centre, Bengaluru, India 560059\\ {Email:athreya@isibang.ac.in}}
\author{Soumik Pal}
\address{Soumik Pal\\ Department of Mathematics \\ University of Washington\\ Seattle WA 98195, USA\\ {Email: soumikpal@gmail.com}}
\author{Raghav Somani}
\address{Raghav Somani\\ Paul G. Allen School of Computer Science \& Engineering \\ University of Washington\\ Seattle WA 98195, USA\\ {Email: raghavs@cs.washington.edu}}
\author{Raghavendra Tripathi}
\address{Raghavendra Tripathi\\ Department of Mathematics \\ University of Washington\\ Seattle WA 98195, USA\\ {Email: raghavt@uw.edu}}

\keywords{exchangeable arrays, exponential random graphs, gradient flows, graphons, measure-valued graphons, Metropolis algorithm, stochastic block model}
	
\subjclass[2000]{05C80, 60K35, 65C05}

\thanks{S. A. research was supported in part by Knowledge Exchange grant and Infosys Excellence grant at ICTS. S. P., R. S. and R. T. gratefully acknowledge the support from NSF grant DMS-2134012 (Scale MoDL). S. P. and R. T. are also partially supported by NSF grant DMS-2052239. Thanks to PIMS Kantorovich Initiative (KI) for facilitating this collaboration. KI is supported by a PIMS PRN and an NSF Infrastructure grant DMS-2133244. We also thank Zaid Harchaoui and Sewoong Oh for their support and helpful comments on this work.}

\date{\today}

\begin{document}

\begin{abstract}
We consider two classes of natural stochastic processes on finite unlabeled graphs. These are Euclidean stochastic optimization algorithms on the adjacency matrix of weighted graphs and a modified version of the Metropolis MCMC algorithm on stochastic block models over unweighted graphs. 
In both cases we show that, as the size of the graph goes to infinity, the random trajectories of the stochastic processes converge to deterministic curves on the space of measure-valued graphons. Measure-valued graphons, introduced by Lov\'{a}sz and Szegedy in \cite{lovasz2010decorated}, are a refinement of the concept of graphons that can distinguish between two infinite exchangeable arrays that give rise to the same graphon limit. We introduce new metrics on this space which provide us with a natural notion of convergence for our limit theorems. This notion is equivalent to the convergence of infinite-exchangeable arrays. Under  suitable assumptions and a specified time-scaling, the Metropolis chain admits a diffusion limit as the number of vertices go to infinity. We then demonstrate that, in an appropriately formulated zero-noise limit, the stochastic process of adjacency matrices of this diffusion converges to a deterministic gradient flow curve on the space of graphons introduced in~\cite{Oh2023}. A novel feature of this approach is that it provides a precise exponential convergence rate for the Metropolis chain in a certain limiting regime.  The connection between a natural Metropolis chain commonly used in exponential random graph models and gradient flows on graphons, to the best of our knowledge, is new in the literature as well.
\end{abstract}

\maketitle



\section{Introduction}\label{sec:Introduction}


Let $(\Omega, \mcal{F}, \mathbb{Q})$ be a probability space and let $\mcal{H}\colon \Omega \rightarrow [0,\infty]$ be a measurable function (called the Hamiltonian). We are often interested in the set of minimizers of $\mcal{H}$. Instead of finding the actual minimizers, which can be computationally expensive, one often considers a Gibbs measure on $(\Omega, \mcal{F})$ whose density, with respect to $\mathbb{Q}$, is proportional to $\exp(-\beta \hamil(\cdot))$, for some $\beta >0$. Here $\mathbb{Q}$ is usually taken to be some kind of a ``uniform'' probability distribution on $\Omega$. The parameter $\beta$ is often called the \textit{inverse temperature}. It is well known that, in many circumstances of interest, as $\beta \rightarrow \infty$, the Gibbs measure concentrates around the minimizers of $\hamil$ (see~\cite{hwa80}). Thus, one may replace the problem of optimization of $\hamil$ by a problem of sampling from the Gibbs measure for a large $\beta$. This is achieved by running suitable stochastic processes with an invariant distribution given by the Gibbs measure~\cite{LA87}. 


When $\Omega$ is continuous with a notion of differentiability of $\hamil$, such as $\R^d$, a natural stochastic process is the Langevin diffusion: 
\begin{equation*} \label{eq:sdeH} 
\diff X(t) = -\nabla \hamil (X(t))\diff t + \sqrt{\frac{2}{\beta}}\diff B(t),
\end{equation*} 
where $B$ is standard $d$-dimensional Brownian motion and $\beta$ is commonly called the inverse temperature parameter. In practice, stochastic gradient descent algorithms are used to mimic the paths of the Langevin diffusion in discrete time. As $\beta \rightarrow \infty$, the paths of the Langevin diffusion converge to that of the gradient flow of $\hamil$, namely $$\dot{x}(t)=-\nabla \hamil(x(t)),$$ which in a sense gives the fastest decay of the Hamiltonian. On the other hand, on discrete spaces or when the gradient of the Hamiltonian is not well defined, one employs a MCMC algorithm~\cite{Diaconis_MCMC}, such as the celebrated Metropolis algorithm~\cite[Section 2.4]{MCMC_survey}, to sample from the Gibbs distribution.


In this paper, $\Omega$ is the space of dense edge-weighted unlabeled graphs and we consider natural stochastic processes used to optimize a Hamiltonian $\hamil$. Examples of such stochastic processes are graph-valued Metropolis Markov chains and processes arising out of optimization algorithms on edge-weights such as stochastic gradient descent (SGD)~\cite{robbins1951stochastic,kiefer1952stochastic, benaim1999dynamics,kushner2003stochastic,borkar2009stochastic,moulines2011nonasymptotic,kushner2012stochastic}. A theme that we pursue is that with certain modifications and in a certain limiting regime, as the size of the graph goes to infinity, both processes are related to a gradient flow. More specifically, we analyze a Metropolis Markov chain on a stochastic block model (SBM). SBMs are a widely used family of models of random graphs (see~\cite{hll83, v18}). The base Markov chain runs on a SBM with $r$ communities, with $n$ members in each community, with an acceptance-rejection step specified by $\hamil$ and the inverse temperature parameter $\beta$. Our algorithm includes a novel relaxation procedure after each accept-reject step which introduces a further positive parameter $\sigma$. When we keep $r$ fixed and let $n \rightarrow \infty$, with other parameters suitably scaled, the edge densities between communities converge to a stochastic differential equation on $r \times r$ symmetric matrices (see Section~\ref{sec:relax_Metro} for an overview  and Proposition~\ref{thm:SBM_to_rDiffusion} for a precise statement of the result).  Further, in Proposition~\ref{prop:rDiffusion_to_McKeanVlasov} we prove that, as $r\rightarrow \infty$, the paths of the stochastic evolution of the edge density matrices converge to a deterministic curve on a metric space of \textit{measure-valued graphons} (MVG).

Although MVGs have already been introduced in the literature by~\cite{lovasz2010decorated}, our contribution here are twofold. 
First, we define two cut-like metrics, $\Delta_\blacksquare$ and $\W_\blacksquare$ (see Definitions~\ref{defn:mvg_cut_metric} and~\ref{defn:Wass_cut}), that capture the topology of convergence as in~\cite{lovasz2010decorated} under which the space of MVGs is a compact topological space. See Theorem~\ref{thm:equivalence_of_cut}  and Examples \ref{exp:moment_graphon}-\ref{exp:ternoulli} where we illustrate the strengths and nuances of this convergence. 
Second, we relate this convergence to the convergence of infinite exchangeable arrays (IEAs).  An (IEA) $\X$, see~\cite[Chapter 7]{kallenberg2005probabilistic}, is a doubly-indexed sequence of random variables $\left( X_{i,j}\right)_{(i,j)\in\Natural^{(2)}}$ defined on a single probability space whose joint distribution is invariant under finite permutations of its rows and columns. That is, if $\varsigma$ is any finite permutation on $\Natural$, then the joint distribution of $\left( X_{\varsigma_i,\varsigma_j}\right)_{(i,j)\in\Natural^{(2)}}$ is the same as that of the original array (see Section~\ref{sec:Setup_Results}). It follows from the Aldous-Hoover representation theorem~\cite{kallenberg1989representation}, that the IEAs are in one-to-one correspondence with random measure-valued graphons (MVG). In Theorem~\ref{thm:Correspondence} we prove a homeomorphism between the set of probability measures on the space of MVGs and the space of laws of IEA, each equipped with the corresponding weak topology. This is a generalization of \cite[Theorem 5.3]{diaconis2007graph} suited to our applications below. 
 
In Section~\ref{subsec:McKeanVlasov_mvg}, we define a general class of diffusions on symmetric $r \times r$ matrices and also show in Theorem~\ref{thm:mv} that under suitable assumptions these processes converge, as $r\rightarrow \infty$, to a deterministic curve on the space of MVGs specified by a McKean-Vlasov stochastic differential equation (SDE). See Proposition~\ref{prop:rate} for the explicit rate. This significantly strengthens a similar convergence results in~\cite{HOPST22} where a similar convergence statement had been obtained on the space of graphons as opposed to MVGs (or, equivalently, IEAs). One example, within the purview of Theorem~\ref{thm:mv}, is the $r\times r$ matrix of pairwise edge-densities from our Metropolis model above. We summarise this result in Proposition~\ref{prop:rDiffusion_to_McKeanVlasov}. The other kind of example covered by Theorem~\ref{thm:mv} is stochastic gradient descent of suitable functions on $r\times r$ symmetric matrices. In~\cite{HOPST22} it was argued that these stochastic gradient descents converge pathwise to the gradient flow on the space of graphons, a concept introduced in~\cite{Oh2023}. We show in Proposition~\ref{lem:Sigma0convergence} that, as the relaxation parameter $\sigma\rightarrow 0+$, the random trajectory of $r\times r$ random matrices of edge densities between communities generated by the Metropolis algorithm converges to a deterministic curve on graphons which the gradient flow of $\beta \hamil$ on the space of graphons. Therefore, under suitable convexity assumptions on $\hamil$, we obtain (see Proposition~\ref{prop:convergence_rate_MH}) an exponential rate of convergence of to the minimizer of $\hamil$. Combining this with Proposition~\ref{prop:rate}, which gives non-asymptotic error bounds, we obtain that, in a certain limiting regime, the adjanceny matrix from our Metropolis chain converges exponentially fast to the minimizer of $\hamil$.
This completes the cycle of ideas that constitute the theme of this paper that we had alluded to earlier. In Section~\ref{sec:graphons,mvgs,IEAs} we give basic definitions related to graphons and MVGs and introduce the Metropolis chain in Section~\ref{sec:intro_metropolis}. A reader familiar with these basics may directly go to Section~\ref{subsec:ComputationExample} where we give a computational example.

\vspace{-5pt}

\subsection{Graphons, measure-valued graphons and infinite exchangeable arrays}\label{sec:graphons,mvgs,IEAs}

We will fix  some objects and notations for the rest of the article. For any set $X$, we use $X^2$ to denote the usual Cartesian product $X \times X$ while $X^{(2)}$ is used to denote the set $X^2/{\sim}$ where we identify $(a, b) \sim (b, a)$ for all $a, b \in X$. We use this notation for domains of symmetric functions.

Let $\Wcal$ denote the space of all bounded measurable functions $w\colon~[0,1]^{(2)}\to\R$ that are symmetric, i.e., $w(x,y)=w(y,x)$ for all $(x,y)\in[0,1]^{(2)}$. Let {$\Wcal_{[0,1]}$} be the set of functions $w\in\Wcal$ with $0\le w(x,y)\le 1$ for $(x,y)\in[0,1]^{(2)}$. More generally, for a bounded interval $I\subset\R$, let $\Wcal_I$ be the set of all functions $w\in\Wcal$ with $w(x,y)\in I$. Given a function $w\in\Wcal$, we can think of the interval $[0,1]$ as the set of nodes, and of the value $w(x,y)$ as the weight of the edge $(x,y)\in[0,1]^{(2)}$. We call the functions in $\Wcal$ {\it kernels}. For every kernel $w$ we may define the homomorphism density function with respect to a simple graph $F$ as follows: If $F$ is a simple graph with $V(F)=\{1,\dots,k\}$, then let
\begin{align}
    t(F,w)\coloneqq \int_{[0,1]^k} \prod_{\set{i,j}\in E(F)} w(x_i,x_j)\diff x.\label{eq:homomorphism_density_graphon}
\end{align}
Every weighted graph $G$ with vertices labeled $[n]\coloneqq \{1,\dots,n\}$ and edge-weights $\beta_{i,j}$ between any two vertices labeled $i$ and $j$ induces a piecewise constant kernel as follows. Let $Q_n \coloneqq \set{Q_{n,i}}_{i\in[n]}$ be defined as $Q_{n,1}=\squarebrack{0,\frac{1}{n}}$, $Q_{n,2}=\left(\frac{1}{n},\frac{2}{n}\right\rbrack$, $\dots$, and $Q_{n,n}=\left(\frac{n-1}{n},1\right\rbrack$. Set $w_G(x,y)=\beta_{v(x),v(y)}$, where $v(x)=i$ whenever $x\in Q_{n,i}$ for some $i\in[n]$. Informally, we consider the symmetric adjacency matrix of $G$, and replace each entry $(i,j)$ by a square of size $\frac{1}{n}\times\frac{1}{n}$ with the constant function $\beta_{i,j}$ on this square. Let us denote the set of all $r\times r$ symmetric matrices with elements in $[0,1]$ and $[-1,1]$ by $\Mcal_{r,+}$ and $\Mcal_r$ respectively, i.e., $\Mcal_{r, +} \coloneqq [0,1]^{[r]^{(2)}}$ and similarly for $\Mcal_r$. It follows from our discussion that, for any $r \in \Natural$, the set $\Mcal_r$ can be naturally identified with a subset of finite dimensional kernels, $\Wcal_r \subset \Wcal$. This identification/embedding will be denoted by $K$ (as in $K(A)$ is the kernel corresponding to the matrix $A$) and its inverse will be denoted by $M_r$ (as in matrix). A simple unweighted graph $G$ can be thought of as a weighted graph with edge-weights being the indicators that the edges exist. Therefore, the notion of homomorphism density $t(F, \slot{})$ extends naturally to weighted or unweighted simple graphs as well.


If two graphs are isomorphic, the two may be identified by permuting the labels on the vertices. This may be done for kernels by the following equivalence relation. Recall that a map $\varphi\colon~[0,1]\to[0,1]$ is {\it Lebesgue measure-preserving}, if it is measurable and push-forwards the Lebesgue measure to itself, i.e. $\abs{\varphi^{-1}(A)} = \abs{A}$ for all Borel measurable sets $A\subseteq [0,1]$. Throughout this paper, we will refer to a Lebesgue measure-preserving map simply as measure-preserving map and denote the set of all such maps by $\Tcal$. For $w\in \Wcal$ and $\varphi\colon~[0,1]\to[0,1]$, we define $w^\varphi$ by $w^\varphi(x,y)=w(\varphi(x),\varphi(y))$ when $(x,y) \in [0,1]^{(2)}$. By definition, it is immediate that, if $\varphi$ is measure-preserving, $t(F,w) = t(F,w^\varphi)$ for every finite simple graph $F$. So one defines an equivalence relation $\cong$ on $\Wcal$ such that $w_1\cong w_2$ if there exist measure preserving transformations $\varphi_1,\varphi_2\colon~[0,1]\to[0,1]$ and $w\in\Wcal$ such that $w_1=w^{\varphi_1}$, and $w_2 = w^{\varphi_2}$. We will call $\widehat{\Wcal}\coloneqq \Wcal/{\cong}$ as the space of graphons and we will refer to any element as a graphon. Wherever it is clear from the context, for any kernel $w\in\Wcal$, we will use an abuse of notation and use the same symbol $w$ to denote the equivalence class, or the graphon, corresponding to the kernel. 



It is shown in~\cite{lovasz2006limits} or~\cite[Theorem 3.1]{borgs2008convergent} that $\widehat{\Wcal}_{[0,1]}\coloneqq \Wcal_{[0,1]}/{\cong}$, the space $[0,1]$-valued graphons, is a \textit{completion} of the space of finite unlabeled graphs with respect to the \textit{cut metric} $\delta_{\cut}$, which first appeared, for matrices, in the work of Frieze and Kannan~\cite{frieze1999quick} (see Definition~\ref{def:cut_metric}). It is also well known that  $\widehat{\Wcal}_{[0,1]}$ is related to infinite exchangeable arrays taking values in $\set{0,1}$~\cite{austin2012exchangeable,diaconis2007graph,janson2010graphons}. 
To understand the relation between IEAs and graphons, it is useful to consider an analogy from the classical de Finetti's Theorem~\cite[Theorem 1.1]{kallenberg2005probabilistic}. An infinite exchangeable sequence of random variables $(X_i)_{i\in \Natural}$ corresponds to a random probability distribution in the sense that the sequence of finite empirical distributions $\frac{1}{n} \sum_{i=1}^n \delta_{X_i}$ converges weakly, in probability, to this (random) probability distribution, as $n\rightarrow \infty$. Moreover, given an instance of this random distribution, the conditional distribution of the exchangeable sequence is i.i.d. with the same law. Similar correspondence between IEAs and (random) graphons is achieved by a powerful generalization of de Finetti's theorem due to Aldous and Hoover~\cite{aldous1981representations,hoover1982row,aldous1982exchangeability,kallenberg1989representation}. It states that for every infinite symmetric exchangeable array ${\bf X}=(X_{i,j})_{(i,j)\in \N^{(2)}}$ there exists a Borel measurable function $f\colon [0, 1]^{4}\to \R$ satisfying $f(\slot{}, x, y, \slot{})=f(\slot{}, y, x, \slot{})$ for a.e. $(x,y)\in[0,1]^{(2)}$, and a collection of i.i.d. $\mathrm{Uniform}[0, 1]$ random variables $U,\set{U_i}_{i\in\Natural},\set{U_{i,j}=U_{\set{i,j}}}_{i,j\in\Natural}$ on some probability space such that the IEA $\Y$ defined as $Y_{i, j}=f(U, U_i, U_{j}, U_{\{i, j\}})$ for all $(i,j)\in\Natural^{(2)}$,  has the same distribution as $\X$.

Thus, via an Aldous-Hoover representation, an IEA ${\bf X}$ induces a random graphon $w^{(U)}$, defined as
\begin{equation}\label{eqn:AH_Graphon}
    w^{(u)}(x, y) \coloneqq \E{f(U, U_1, U_2, U_{1,2})\given U=u, (U_1, U_2)=(x,y)},
\end{equation}
for $(x,y)\in[0,1]^{(2)}$ and $u\in[0,1]$.

%

In~\cite[Section 5]{diaconis2007graph} it is shown that, when the entries in IEA take values in $\set{0,1}$, there is a one-to-one correspondence between IEAs and probability distributions on the space of graphons.
However, when the IEA takes more nontrivial values, this is no longer the case. We illustrate this via an example.
Let $\G=(G_{i,j})_{(i, j)\in \N^{(2)}}$ be an IEA such that every $G_{i,j}$ is an i.i.d. $\mathrm{Ber}(1/2)$ random variables. 
Let $\K=(K_{i,j})_{(i, j)\in \N^{(2)}}$ be a constant IEA such that $K_{i,j}\equiv 1/2$ for all $i, j\in\Natural$. Both the IEAs $\G$ and $\K$ correspond to a constant graphon $w_{p}\equiv 1/2$. 

We propose a remedy for this by expanding the definition of a graphon to take values in $\mathcal{P}\left([-1,1]\right)$, the set of Borel probability measures on $[-1,1]$.

\begin{definition}[Measure-valued kernel]\label{def:mvKernel}
    A \emph{measure-valued kernel} is a measurable function $W\colon [0, 1]^{(2)}\to \Pcal([-1, 1])$ such that $W(x, y)=W(y, x)$ for a.e. $(x,y)\in[0,1]^{(2)}$. Here $\Pcal([-1, 1])$ is the space of probability measures on the interval $[-1, 1]$ equipped with the Borel sigma-algebra generated by the topology of weak convergence. 
    We will denote the set of all measure-valued kernels by $\Wfrak$.
\end{definition}

Measure-valued graphons are defined similarly in Definition~\ref{def:mvGraphons}. More details can be found in~\cite{lovasz2010decorated,kovacs2014multigraph} where a notion of convergence based on \textit{decorated} homomorphism functions is also discussed. In this paper we show that this convergence is metrizable via a metric similar to the Wasserstein metric for probability measures.

\begin{definition}[Natural projection from $\mvGraphons$ to $\Graphons$]\label{def:natural_projection}
Given a measure-valued kernel $W\in\Wfrak$, we can define a corresponding kernel $w\in\Wcal$ defined as $\displaystyle w(x, y)\coloneqq \int_{[-1, 1]} \zeta \, \, W(x, y)(\diff \zeta)$ for a.e. $(x,y)\in[0,1]^{(2)}$. This naturally defines a projection from $\mvGraphons$ to $\Graphons$. We will often refer to this projection as \emph{natural projection} and denote $w=\E{W}$. This map from $(\mvGraphons, \Delta_\blacksquare)$ to $(\Graphons, \delta_{\cut})$ is $1$-Lipschitz  as seen from Definition~\ref{defn:mvg_cut_metric}.
\end{definition}

Suppose an IEA ${\bf X}$ has an Aldous-Hoover representation given by a function $f$. Analogous to equation~\eqref{eqn:AH_Graphon}, one can define a random measure-valued kernel $W\in\Wfrak$ (and hence an MVG) as follows. For $(x,y)\in[0,1]^{(2)}$, set 
\begin{align}\label{eqn:AH}
    W^{(u)}(x, y) &\coloneqq \Law{X_{i,j}\,\vert\, U=u, (U_i,U_j)=(x,y)}, \quad u\in[0,1].
\end{align}
Conversely, an MVG $W$ generates an IEA in the following manner. Let $(U_i)_{i \in \Natural}$ denote an i.i.d. sequence of $\mathrm{Uniform}[0,1]$ random variables. Define an IEA $\X$, where, given $(U_i=u_i)_{i \in \Natural}$, $X_{i,j}$ for $(i,j)\in\Natural^{(2)}$, is independently sampled from the probability measure $W(U_i, U_j)\in\Pcal([-1,1])$. This indeed remedies the issue raised in the example above with IEAs $\G$ and $\K$. The MVG corresponding to $\G$ is $W_{\G}(x, y)\equiv \frac{1}{2}(\delta_0+\delta_{1})$ while the MVG corresponding to $K$ is $W_{\K}(x, y)\equiv \delta_{1/2}$, where $\delta_{\cdot}$ refers to the delta mass.  Note that the random graphon $w$ corresponding to IEA $\X$ in~\eqref{eqn:AH_Graphon} can be obtained from the random MVG $W$ as defined in~\eqref{eqn:AH} via the natural projection of $W$. If we take $n\times n$ blocks of the IEA $\X$, one can show that, as $n\rightarrow \infty$, this sequences of random exchangeable matrices converge to deterministic limits in the spaces of both graphons and measure-valued graphons. The two limits are related by the natural projection. We carry this idea in the dynamics setting as well.



This correspondence between IEAs and random MVGs recovers the results in~\cite{diaconis2007graph} for general exchangeable arrays. In Section~\ref{sec:Setup_Results}, we define the space of MVGs and notion of convergence that defines the topology on the space of MVGs. We introduce two new metrics on MVGs analogous to the usual cut metric and the Wasserstein metric respectively. In Theorem~\ref{thm:equivalence_of_cut} we show the equivalence of all these. We, then, prove a correspondence theorem between the compact metric space of probability measures on MVGs and the space of IEA equipped with the weak topology in Theorem~\ref{thm:Correspondence}. This allows us to consider processes on exchangeable matrices/graphs and take their limits either as IEAs or as MVGs. 

\subsection{Limits of Markov Processes on Weighted Graphs}\label{sec:intro_metropolis}



We introduce the following general class of deterministic curves in the space of MVGs described by a stochastic differential equation (SDE). Suppose we are given a pair of  functions 
\begin{equation}\label{eqn:Def_b_Sigma}
    \begin{split}
        b\colon [-1,1]\times \Wfrak \to L^\infty\big([0,1]^{(2)}\big),\; \text{and}\quad 
        \Sigma\colon [-1, 1]\times \Wfrak \to L^\infty([0,1]^{(2)}),
    \end{split}
\end{equation}
where $L^\infty([0,1]^{(2)})$ is the set of all functions $f\colon[0, 1]^{(2)}\to \R$ such that $\infnorm{f}< \infty$.

Let $(\Omega, \Fcal, \mathbb{P})$ be a probability space that supports a standard Brownian motion $B(\cdot)$ and a pair of independent $\mathrm{Uniform}\interval{0,1}$ random variables $U, V$. Given $W_0\in \Wfrak$, consider the following coupled system $(X, W, U, V)$ of one-dimensional reflected diffusion $X$, a curve $W$ on $\Wfrak$ and the pair of uniform random variables, such that given $(U, V)=(u,v)$,  the process $X(\cdot)$ satisfies the initial condition $X(0)\sim W_0(u, v)$ and the SDE
\begin{equation}\label{eq:McKeanV1}
\begin{split}
    \diff X(t) &= b\left( X(t), W(t)\right)(u, v)\diff t + \Sigma(X(t), W(t))(u, v)\diff B(t)\\
    &\qquad\qquad\qquad+ \diff L^-(t) - \diff L^+(t),\\
     W(t)(x,y) &\coloneqq  \Law{ X(t) \given (U,V)=(x,y)},\quad \forall\; (x,y)\in\interval{0,1}^{(2)},
\end{split} 
\end{equation}
 where $(X,L^+,L^-)$ solves the Skorokhod problem~\cite{kruk2007explicit} with respect to $\interval{-1,1}$ (see Section~\ref{sec:Skorokhod} for details). The system described by equation~\eqref{eq:McKeanV1} will be referred to as the MVG \textit{McKean-Vlasov SDE} (MVSDE). Under appropriate assumptions on $b$ and $\Sigma$, Proposition~\ref{prop:McKean-Vlasov_existence} shows that the MVSDE admits a pathwise unique solution. Notice that $W$ is a deterministic curve on measure-valued kernels, and thus, on measure-valued graphons. A similar McKean-Vlasov SDE was introduced in~\cite{HOPST22} but the convergence was obtained only in the sense of graphons and, hence, cannot capture the convergence of general exchangeable arrays. However, there is a corresponding deterministic curve on graphons $w(t)=\E{W(t)}$ given by the natural projection map. It is useful to think of $w$ as capturing the evolution of macroscopic properties while $W$ describing the microscopic properties.   

Where do such processes appear? In Section~\ref{subsec:McKeanVlasov_mvg} we consider a general class of diffusion on symmetric $r \times r$ matrices whose coordinates are exchangeable and are evolving under a suitable mean-field interaction. In Theorem~\ref{thm:mv} we prove that processes in this general class have corresponding deterministic limits that are examples of~\eqref{eq:McKeanV1}. This is natural since one can spot that~\eqref{eq:McKeanV1} is equivalently characterized by an IEA of independent diffusions satisfying the McKean-Vlasov SDE generated by an i.i.d. sequence of $\mathrm{Uniform}[0,1]$ random variables. Such diffusions naturally arise in the context of stochastic gradient descent of functions defined on the space of graphs. For another example, consider the problem of ``soft'' optimization described at the very beginning where, for $\beta>0$, one may wish to consider a Gibbs measure on $\Graphons$ with a ``density'' proportional to $\exp\left( - \beta \hamil \right)$. However, as there is no canonical measure on the space of graphons, this does not seem feasible. On the other hand, consider $\hamil$ restricted to the space of $r\times r$ graphons $\Graphons_r$. By pulling back the natural map from kernels to graphons and identifying $r\times r$ kernels with $r\times r$ symmetric $[0,1]$-valued matrices, one can think of $\hamil$ as a function $H_r$ on symmetric matrices, i.e., $H_r = \Hcal \circ K$ on $\Mcal_{r,+}$. One can define a natural Gibbs measure on $\Mcal_{r, +}$ corresponding to $H_{r}$. A large class of commonly used models fall in this umbrella. See the thesis \cite{ChernThesis} for a historical development and some beautiful real-world applications. In particular, it appears in statistical physics models such as the Curie-Weiss models~\cite[Chapter 4]{ChernThesis}, the exponential random graph models (ERM)~\cite[Chapter 5]{ChernThesis}. We may wish to sample from such a Gibbs measure whether we are trying to find graphs that approximately minimize the Hamiltonian (i.e., an approximate nonparametric maximum likelihood estimator such as MCMLE \cite[Chapter 3.3]{ChernThesis}) or we are sampling from a Bayesian posterior distribution. Although Metropolis or the Gibbs sampling algorithms are popular choices to run MCMC algorithms, their mixing times are generally not known. Another example comes from a series of works of Radin, Sadun and others \cite{Radin18, Radin20, Radin23, radin2023optimal} on the so-called edge-triangle model. Their focus is on a typical graphs with a given number of edges and triangles and to show that they exhibit phase transitions. In \cite[Section 3.1]{Radin23} the authors construct an MCMC scheme to sample from an edge-triangle model. They justify convergence, not theoretically, but empirically. Given a target edge density $e$ and a triangle density $t$, one may easily construct a Hamiltonian that gets minimized when the edge-density and the triangle densities are $e$ and $t$, respectively. Then, sampling from this Gibbs measure will approximately sample from an edge-triangle model. 

We will show that, for suitable $\hamil$, satisfying a semiconvexity condition (Assumption~\ref{asmp:hamil}), the edge density matrix obtained from the Metropolis chain admits limits that are particular cases of \eqref{eq:McKeanV1}. With stricter convexity assumptions we will also be able to say something about the exponential rate of convergence. In particular, this is true for all linear combinations of homomorphism functions ~\cite[Section 5.1.2]{Oh2023}. Notice that given $W\in \Wfrak$, consider the corresponding kernel $w=\E{W}\in\Wcal_{[0,1]}$ via the natural projection. Therefore, any function $b_0:[-1, 1]\times \Wcal\to \rightarrow L^\infty\big([0,1]^{(2)}\big)$ naturally gives a function $b:[-1,1]\times \Wfrak\rightarrow L^\infty\big([0,1]^{(2)}\big)$ via the pullback $b(z, W)=b_0(z, w)$. Such drift functions will naturally arise in our examples. 
Let $\hamil$ be a Hamiltonian on $\Wfrak$ that admits a Fr\'echet-like derivative $D\Hcal$ (see Definition~\ref{def:frechet_like_derivative}), and fix a parameter $\beta >0$. The solution to the McKean-Vlasov SDE~\eqref{eq:McKeanV1} with the drift function induced by $b_0(w)=-\beta D\hamil(w)$ and constant $\Sigma\equiv \sigma$ is analogous to the Langevin diffusion on Euclidean spaces. This family of processes arises as the limit of both stochastic gradient descent on  symmetric matrices as well as the following Metropolis chain on the popular stochastic block models (SBM) \cite{eldan2018exponential}.

\begin{definition}[Empirical Stochastic Block Model (ESBM)] \label{def:ESBM}
    For $r, n \in \NN$ let $q\equiv(q_{i,j})_{1\leq i,j\leq r}\in\Mcal_{r,+}$, and let $N=rn$. A random simple graph with $N$ vertices is called $\ESBM{r,n,q}$ if
    \begin{itemize}
        \item for $i \in [r]$, there are $n$ many vertices having color $i$,
        \item for $i,j\in[r]$, $i\neq j$, $n^2q_{i,j}$ many edges (unordered pairs of vertices $\{u,v\}$) are drawn by randomly sampling without replacement where one vertex has color $i$ and the other has color $j$,
        \item for $i\in[r]$, $\binom{n}{2}q_{i,i}$ many edges are drawn by randomly sampling without replacement unordered pairs of vertices of color $i$, and
        \item the samplings in the last two items are done independently for all pairs $(i,j)\in [r]^{(2)}$.
    \end{itemize}
\end{definition}

To construct the Gibbs probability measure on $\Mcal_{r, +}$, we will be interested in $\ESBM{r,n,q}$ random graphs where the entries of $q$ are also random. For each $n\in \NN$, consider the uniform distribution $\mu_n$ on the discrete set $\set{i/n^2\given i \in \{0\}\cup [n^2]}$  and $\nu_n$ the uniform distribution on the discrete set $\set{i/\binom{n}{2}\given i \in \{0\}\cup \left[\binom{n}{2}\right]}$. Define $U_{n,r}$ to be the probability measure on $\Mcal_{r, +}$ where each entry above the diagonal is independently distributed as $\mu_n$ and the diagonal entries are independently distributed as $\nu_n$. Thus $U_{n,r}$ can be viewed as a discrete uniform distribution on the set of possible edge-densities.

Recall that $H_{r}$ is the restriction of $\hamil$ on the space of $r\times r$ symmetric matrices for each $r\in\Natural$. Fix a positive sequence $(\gamma_n)_{n\in \NN}$ such that 
\begin{equation}\label{eqn:gammagrowth}
    \lim_{n\rightarrow \infty} \gamma_n\log^2 n=0,\quad \text{and} \quad \lim_{n\rightarrow \infty}\frac{\gamma_n n^2}{\log n}=\infty.
\end{equation}
Fix $\beta>0$ and let $\beta_{n,r}\coloneqq \beta r^{-2}/\gamma_n$. 
Consider a family of Gibbs probability measures on $\Mcal_{r,+}$ given by
\[
    \widehat{Q}_{n,r,\beta}(\diff q) = \frac{1}{Z_{n,r, \beta}} \eu^{-\beta_{n, r} H_r(q)} U_{n,r}(\diff q)= \frac{1}{Z_{n,r, \beta}} \eu^{-\beta \gamma_n^{-1} r^{-2} H_r(q)} U_{n,r}(\diff q).
\]
where $Z_{n,r, \beta}$ is the normalizing constant. As each $q\in\Mcal_{r,+}$ corresponds to a simple random graph in $\ESBM{r,n,q}$,  $\widehat{Q}_{n,r,\beta}$ can be thought of as  a random probability distribution on simple graphs over $rn$ vertices. We will denote the model specified  by $\widehat{Q}_{n,r,\beta}$, as $\ESBM{r,n,\beta,\hamil}$. It should be emphasized that the measure $\widehat{Q}_{n, r, \beta}$ depends on the choice of the parameter $\gamma_n$. Note that the above model closely resembles commonly used framework in exponential random graphs (see~\cite{chatterjee2017large} and references therein).


The following Metropolis chain algorithm (see~\cite[Section 3.2]{levin2017markov}) can be used to sample from $\ESBM{r,n,\beta,\hamil}$. 


\begin{itemize}
    \item {\em Base Markov Chain}: The state space of the chain is the set $\statesp_{n,r}$ of all simple graphs on $rn$ vertices with $r$ colors assigned to equal number of vertices. The base chain starts at an arbitrary graph $G(0)=G$ in the state space. Suppose, for $\ell \geq 0$, the Markov chain has completed $\ell$ steps, $\setinline{G(p)}_{p=0}^{\ell}$, and is at graph $G(\ell)$. For $(i,j)\in [r]^{(2)}$, let $m_{i,j}(\ell)$ denote the number of edges between vertices of color $i\in[r]$ and color $j\in[r]$ in $G(\ell)$. The next step in the Markov chain is generated as follows.  
    \begin{itemize}
        \item  For every $(i,j)\in [r]^{(2)}$, $i\neq j$, if $m_{i,j}(\ell)\notin \{0, n^2\}$, then, toss a fair coin. If the coin comes up heads, then delete an edge between color $i$ and color $j$, chosen at random, and if the coin turns up tails, place an additional edge between color $i$ and color $j$ at random. Replace $n^2$ by $\binom{n}{2}$ if $i=j$.
        \item For every $(i,j)\in [r]^{(2)}$, if $m_{i,j}(\ell)=0$, then toss a fair coin. If the coin comes up heads, then add an additional edge, chosen at random, and if the coin turns up tails, do nothing. Similarly, if $m_{i,j}(\ell)=n^2$, $i\neq j$ (or $\binom{n}{2}$, if $i=j$), then toss a fair coin. If the coin turns up heads then delete an existing edge, chosen at random, otherwise do nothing. 
        \item Do these independently for every pair $(i,j)\in [r]^{(2)}$.
    \end{itemize}
    The resulting graph is $G(\ell+1)$ and $q(\ell+1)=\round{{q}_{i,j}(\ell+1)}_{1\leq i, j \leq r}$ be its edge density matrix.  It is not hard to see that the base chain viewed as a process on edge densities is also a Markov chain that is reversible with respect to the uniform distribution $U_{n,r}$. 
    \item {\em Metropolis Chain}: We run the base chain for $s_n\approx \gamma_n^2n^{4}$ many steps followed by an accept-reject step. Suppose we started the base chain at graph $G$ and edge density matrix $q$. After running the base chain for $s_n$ many steps we arrive at a graph $G'$ and a corresponding edge density matrix $q'$. 
    \begin{itemize}
        \item {Accept-reject step}: Accept $G'$ as the next state of the Metropolis chain with probability
        $\exp\left( -\beta_{n,r}\left(  \hamil(q') - \hamil(q)\right)^+ \right)$,
        otherwise, remain at $G$. Here $x^+\coloneqq \max\set{x,0}$.
    \end{itemize}
\end{itemize}
It is standard to see the unique invariant distribution of this Metropolis Markov chain is the Gibbs measure $\widehat{Q}_{n,r,\beta}$. We will explore scaling limits of the chain as $n, r\rightarrow \infty$, $\gamma_n, \beta_{n,r}$ as specified above and when $s_n=O(\gamma_n^2n^4)$. 
But, first, we introduce an additional relaxation step.

   
\begin{itemize}
    \item {\em Relaxed Metropolis Chain}: After every Metropolis {accept-reject} step, we run the base chain for an additional $\ell_{n,r}(\sigma)=O(\sigma^2r^{-4}\gamma_n n^4)$ many steps, for some $\sigma>0$, and always accept the last state. 
\end{itemize}

Thus, our final Markov chain repeatedly runs the base chain for $s_n$ many steps, performs an {accept-reject step} and then runs another $\ell_{n,r}(\sigma)$ many steps of the base chain. We call this the \textit{relaxed Metropolis chain}. Since $\ell_{n,r}(0)=0$, when $\sigma=0$, we recover the true Metropolis chain. However, note that the relaxed chain has a different invariant distribution for any positive $\sigma$. 



Consider the resulting process of $r\times r$ matrix of pairwise edge-densities $q(\cdot)$. One may think of this exchangeable process as a Markov chain on symmetric matrices. In Proposition~\ref{thm:SBM_to_rDiffusion} and Proposition~\ref{prop:rDiffusion_to_McKeanVlasov} we show that, as $n\to\infty$ and $r\to\infty$, and the other parameters are scaled as above, the paths of this process  converge to a deterministic curve on MVG, describe by a McKean-Vlasov SDE~\eqref{eq:McKeanV1} with drift $b$ given by $-\beta D\hamil$ and $\Sigma\equiv \sigma$. By taking the natural projection from MVG to graphons, this implies that the paths of $q$ also converge (see Remark \ref{rem:grapon_curve_MKV}) to a deterministic curve on the space of graphons in the same scaling limit. For a fixed $r$, as $n\to \infty$ the adjanceny matrix of $\ESBM{r, n, q}$ converges to the edge density matrix $q$ in the cut metric. Therefore, this deterministic curve on graphons can be interpreted as the limiting evolution of the adjacency matrices of the sequence of graphs $G(\cdot)$ and is parameterized by $\beta>0$ and $\sigma >0$. 

Notice that the drift is a constant multiple of $-D\hamil$, the direction of steepest descent of $\hamil$. When $\sigma=0$, it is clear that the limiting curve is a time-reparametrization of the gradient flow of $\hamil$ on $\Wcal$. Proposition~\ref{lem:Sigma0convergence} shows that, as $\sigma\rightarrow 0$, the family of limiting curves on graphons converges to a time-changed gradient flow of $\hamil$. Finally, Proposition~\ref{prop:convergence_rate_MH} establishes the exponential convergence rate of this flow under appropriate convexity conditions on the Hamiltonian.
\subsection{A computational example from extremal graph theory}\label{subsec:ComputationExample}

To illustrate our results, we give a concrete example with numerical simulations. To motivate our example, we first recall the celebrated Mantel's theorem~\cite{mantel1907problem} from extremal graph theory. It states that the maximum number of edges in an $n$-vertex triangle-free graph is $n^2/4$. Further, any Hamiltonian graph with at least $n^2/4$ edges must either be the complete bipartite graph $K_{n/2,n/2}$ or it must be pancyclic~\cite{bondy1971pancyclic}. One may attempt to computationally verify this theorem by considering a ``softer" version of the problem. That is, consider the Hamiltonian $\Hcal(\slot{})\coloneqq t(\triangle,\slot{}) - \alpha t(\mathrel{-},\slot{})$ for sufficiently small $\alpha>0$. Here $\triangle$ and $\mathrel{-}$ are the triangle and the edge graphs respectively. Recall that the homomorphism density function $t(F,\slot{})$ of simple graph $F$, defined over unweighted graphs, simply computes the density of the simple graph $F$ in the unweighted graph. Thus, minimizing $\Hcal$ can be roughly thought of as an attempt to minimize the number of triangles in a graph while simultaneously maximizing the number of edges. The linear combinations of homomorphism densities also appear in the study of exponential random graph models (ERGMs) which is usually defined as a probability measure on finite graphs with density proportional to $\exp{(-\hamil)}$ where $\hamil$ is a linear combination of homomorphism density function~\cite{chatterjee2011large}. Hence, in either case the behavior of the Metropolis algorithm to simulate samples from the Gibbs measure is of interest.


We simulate the relaxed Metropolis chain sampling algorithm for $\Hcal$ with $\alpha = 1/4$, $n=16$, $r=16$, $\sigma = 1$, $\gamma_n = 1/4n$ and $\beta = 1/4$. In particular, $\Hcal(\slot{})\coloneqq t(\triangle,\slot{}) - \frac{1}{4} t(\mathrel{-},\slot{})$. The Fr\'echet-like derivative of the Hamiltonian is given by $D\hamil(w)(x,y) = 3\int_{[0,1]}w(x,z)w(z,y)\diff z - {1}/{4}$, for $(x,y)\in[0,1]^{(2)}$, which is an affine transformation of the homomorphism density of $2$-stars in the graphon.

The limit of the adjacency matrix process of the relaxed Metropolis chain as $n\to\infty$, followed by $r\to\infty$, and finally $\sigma\to 0$, is given by the a curve $w\colon \R_+ \to \Wcal_{[0,1]}$ given by
\begin{equation}
    w(t)(x, y)=w(0)(x, y)-\beta \int_0^t D\hamil(w(s))(x, y)\indicator{G_{w(s)}}{}\diff s, \qquad (x,y)\in[0,1]^{(2)},
\end{equation}
where the starting point $w(0)\in\Wcal_{[0,1]}$ is the $L^2$-limit of
the community edge density kernel of the initialization of the Metropolis chain as $r\to\infty$.  The set function $G_u$ for any $u\in\Wcal_{[0,1]}$, defined in equation~\eqref{eqn:G_function}, ensures that the velocity field does not point outside the domain of $\Wcal$ when any coordinate of the flow hits the boundary $[0,1]$.

\sloppy Since the drift is a constant multiple the Fr\'echet-like derivative of $\hamil$, the curve $w$ is a time reparametrization of the gradient flow of $\hamil$ defined in equation~\eqref{eqn:GF}. In Figure~\ref{fig:Mantel_MH} we see that the iteration sequence of the MCMC chain has a close resemblance with the curves shown in~\cite[Figure 1, Section 1.2]{Oh2023} which is a forward Euler discretization of the gradient flow of $\Hcal$ on $\big(\Graphons,\delta_2\big)$. After sufficiently many iterations, we see that the community density kernel corresponding to the graph $G^{(n)}_{r,3.7\times 10^5}$ is close to the graph the one corresponding to a complete bipartite graph as one would expect from Mantel's theorem.

\begin{figure}[!t]
	\centering
        \includegraphics[width=0.3\linewidth]{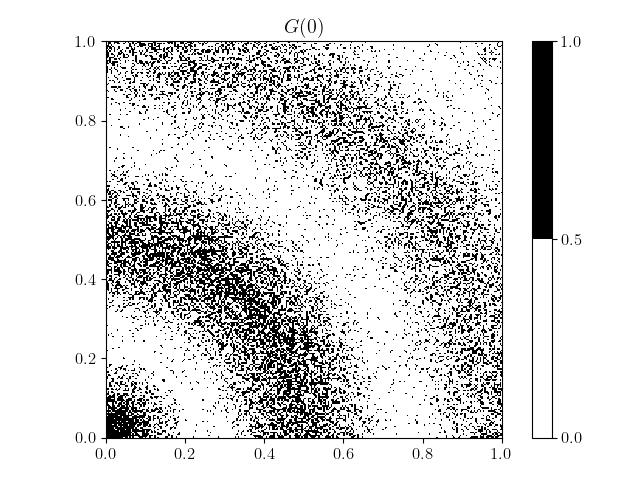}
        \includegraphics[width=0.3\linewidth]{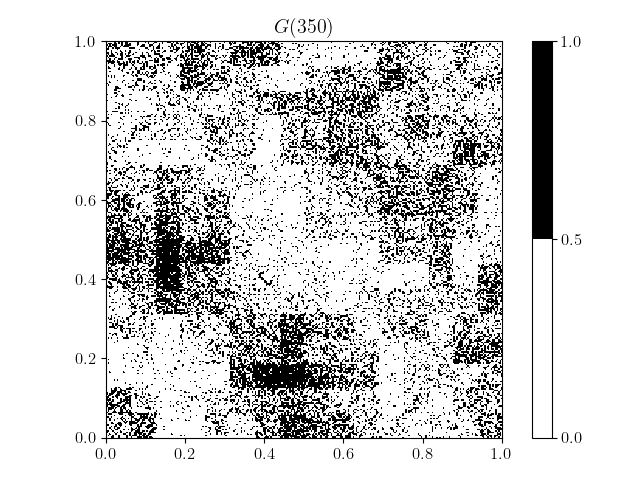}
        \includegraphics[width=0.3\linewidth]{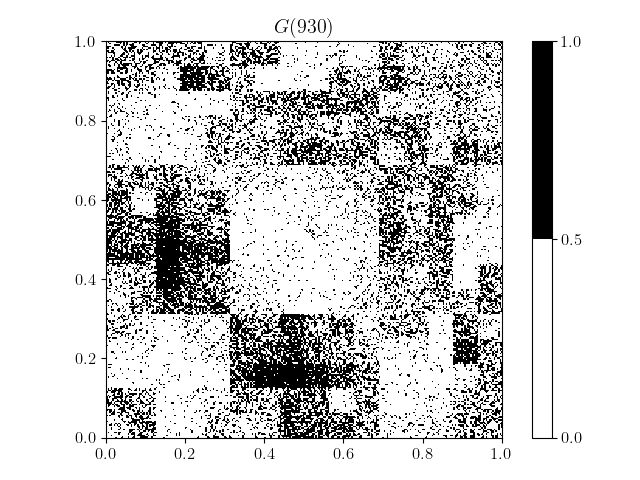}
        \includegraphics[width=0.3\linewidth]{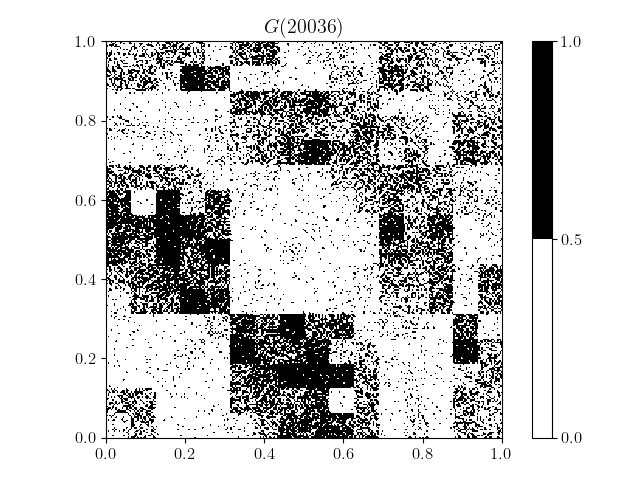}
        \includegraphics[width=0.3\linewidth]{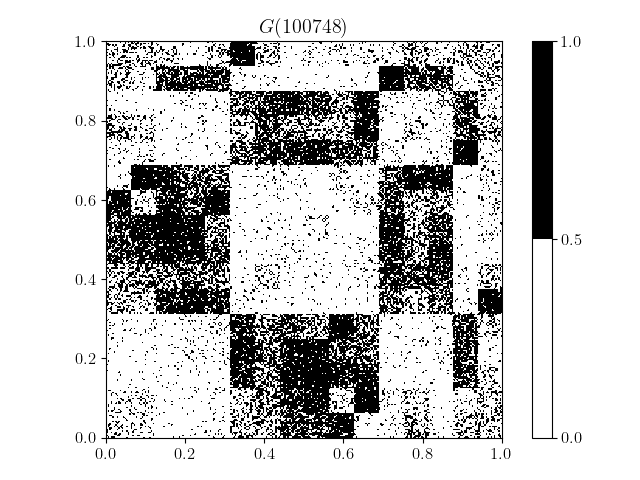}
        \includegraphics[width=0.3\linewidth]{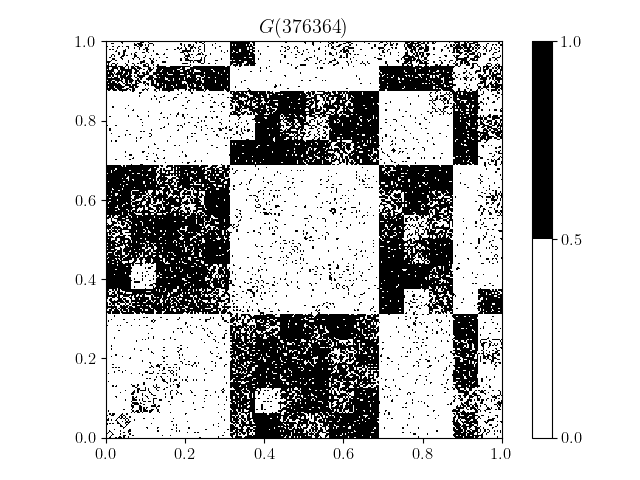}
	\caption{A relaxed Metropolis chain algorithm simulation for $\Hcal = t(\triangle,\slot{}) - \frac{1}{4} t(\mathrel{-},\slot{})$ at initialization and after $3.5\times 10^2$, $9.3\times 10^2$, $2.0\times 10^4$, $1.0\times 10^5$ and $3.7\times 10^5$ iterations respectively (order: from left to right).}
	\label{fig:Mantel_MH}
\end{figure}

\sloppy In Proposition~\ref{prop:convergence_rate_MH} we show that if the Hamiltonian is strongly convex, the curve $w$ converges to the minimizer of $\hamil$ with an exponential rate. Homomorphism density functions, although semiconvex, are not generally strongly convex. To remedy, one may add to the Hamiltonian a multiple of the scalar entropy function~\cite[Section 5.1.3]{Oh2023} that makes it strongly convex enough to guarantee an exponential rate of convergence. That is, for $\gamma>0$ large enough, the following new Hamiltonian $\hamil_\gamma$ defined as $\hamil_\gamma(w)\coloneqq  t(\triangle,w) - \frac{1}{4} t(\mathrel{-},w) + \gamma \int_{[0,1]^2} h(w(x,y))\diff x\diff y$, where $h(p)=p\log p + (1-p)\log(1-p)$ for $p\in(0,1)$ and zero if $p\in \{0,1\}$, is strongly convex and the corresponding gradient flow curve converges exponentially fast. In fact, in this particular example, $\hamil_{\gamma}$ as defined above is strongly convex for any $\gamma > 9/2$. In particular, if we set $\gamma = 5$, following Proposition~\ref{prop:convergence_rate_MH}, we obtain an exponential rate of convergence to the minimizer with rate $\beta\lambda$ with respect to the $\delta_2$ metric, where the semiconvexity constant $\lambda > \gamma-9/2 = 1/2$.  However, to compare our current simulation with those in~\cite{Oh2023} we do not add the entropy regularization here.

\section{Metric and Topology on Measure-valued graphons}\label{sec:Setup_Results}
The key results of this section are contained in Section~\ref{subsec:Metric_mvg} and Section~\ref{subsec:ExchangeableArray_and_mvg}. In Section~\ref{subsec:Background}, we setup the notations and the necessary background on graphons and MVGs. 
In Section~\ref{subsec:Metric_mvg} we introduce two novel metrics on $\mvGraphons$. In Theorem~\ref{thm:equivalence_of_cut} we show that these metrics induce the usual topology (see Definition~\ref{def:mvGraphons}) on $\mvGraphons$. In Section~\ref{subsec:ExchangeableArray_and_mvg}, we establish the correspondence between the distributions of IEAs and probability distributions on $\mvGraphons$ (see Theorem~\ref{thm:Correspondence}). We further use this connection to make precise the notion of convergence of symmetric exchangeable matrices to IEAs. 

\subsection{Background and notations}\label{subsec:Background}
We first recall some preliminaries on graphons which serves as a useful parallel. One key concept that is frequently used later is the notion of Fr\'echet-like differentiability (see Definition~\ref{def:frechet_like_derivative}) introduced in~\cite{Oh2023}.

\subsubsection{Topology, metric and gradient flow on graphons}\label{sec:BackgroundGraphons}
We begin with a brief introduction to dense graph limits introduced in~\cite{lovasz2006limits} and refer the reader to the monograph~\cite{lovasz2012large} and papers~\cite{borgs2008convergent,borgs2012convergent} for in-depth discussions. We consider  simple edge-weighted graphs. A simple graph $G$ is a graph without loops or multiple edges. We denote the vertex and edge set of $G$ by $V(G)$ and $E(G)$, respectively. An {\it edge-weighted graph} $G$ is a graph with a weight $\beta_{i,j}=\beta_{i,j}{(G)}\in\R$ associated with each edge $\{i,j\}\in E(G)$. Set $\beta_{i,j}=0$ if $\set{i,j}\notin E(G)$. Let~$(G_n)_{n\in\Natural}$ be a sequence of simple edge-weighted graphs with vertex set $[n] \coloneqq \{1,\dots,n\}$. We call~$(G_n)_{n\in\Natural}$ a \emph{dense graph sequence} if the number of edges~$E(G_n)$ in~$G_n$ is~$\Theta(n^2)$.

Let $F$ and $G$ be two simple graphs with edge-weights $1$. We define $\hom(F,G)$ as the number of homomorphisms from $F$ to $G$, i.e., the number of adjacency preserving maps $V(F)\to V(G)$, and the {homomorphism density} of $F$ in $G$ as
\[
    t(F,G)=\frac {1}{\abs{V(G)}^{\abs{V(F)}}}\hom(F,G).
\]
For  weighted graphs $G$ let $\displaystyle {\rm hom}(F,G)=\sum_{\phi\colon V(F)\to V(G)} \prod_{i,j\in E(F)}\beta_{\phi(i),\phi(j)}(G)$
and define the {homomorphism density} of $F$ in $G$ as
\[
    t(F,G)=\frac{{\rm hom}(F,G)}{\abs{V(G)}^k},
\]
where $k$ is the number of vertices in $F$. When $\abs{F} > \abs{G}$ we define~$t(F,G) = 0$.  Note that~$t(F,G)\in [0,1]$. Let ${\mathcal F}$ denote the class of isomorphism classes on finite graphs and let~$(F_{i})_{i\in\Natural}$ be a particular enumeration of~$\mathcal F$, where each element is the representative of an isomorphism class. We can define a distance function between graphs by $\dsub(G,G') = \sum_{i\in\Natural} 2^{-i}\babs{t(F_i,G) - t(F_i,G')}$.

A key feature of $\dsub$ is that there is a natural completion of $({\mathcal F}, \dsub)$ by {\em graphons}.  It is shown in~\cite{lovasz2006limits} or~\cite[Theorem 3.1]{borgs2008convergent} that a sequence $(G_n)_{n\in\Natural}$ of simple graphs is $\dsub$-Cauchy if and only if there is a graphon $w$ with values in $[0,1]$ such that $t(F,G_n)\to t(F,w)$ for every simple graph $F$ as $n\to\infty$.

There is a more convenient metric that captures the above completion. For any kernel~$w\in\Wcal$ the \emph{cut-norm} of~$w$ is defined as
\begin{equation}\label{eq:cut-norm} 
    \norm{\cut}{w} \coloneqq \sup_{S,T \subseteq [0,1]} \abs{\int_{S \times T} w(x,y) \diff x \diff y},
\end{equation}
where the supremum is taken over Lebesgue-measurable subsets of~$[0,1]$. 

The cut norm $\cutnorm{}$ as defined in equation~\eqref{eq:cut-norm} induces a metric on graphons called the 
\emph{cut metric}, denoted by $\delta_{\cut}$, when restricted to the quotient space $\widehat{\Wcal}\coloneqq\mathcal{\Wcal}/{\cong}$. Moreover, $(\Graphons, \delta_{\cut})$ is compact~\cite{Lovsz2007SzemerdisLF}.
\begin{definition}[Cut Metric~{\cite[Section 3.2]{borgs2008convergent}}]\label{def:cut_metric}
    For any two graphons $w_1, w_2\in \Graphons$, let us pick kernel representatives from these graphons and denote it by the same symbol. Then, 
    \[
        \delta_{\cut}(w_1, w_2)\coloneqq \inf_{\varphi, \psi\in\Tcal} \cutnorm{w_1^{\varphi}-w_2^{\psi}}.
    \]
     where $\Tcal$ is the set of all Lebesgue measure preserving transformations from $[0,1]$ to itself.
\end{definition}
The convergence in cut-metric is equivalent to the convergence of the homomorphism density functions $t(F, \slot{})$ for all finite simple graph $F$. This notion readily extends to weighted graphs as well and it has been employed to study the limit of processes on weighted graphs (or equivalently symmetric matrices) in~\cite{CraneAAP,bayraktar2020graphon,athreya2016dense,Oh2023,HOPST22} as the size of graphs grows to infinity. More generally, given any norm $\norm{}{}$ on $\mathcal{W}$, one can similarly define an induced metric $\delta_{\norm{}{}}$ on $\widehat{\Wcal}$. In particular, the induced metric due to the $L^2$ norm, $\enorm{}\colon L^2(\interval{0,1}^{(2)})\to\R_+$, is called the \emph{invariant $L^2$ metric}, $\delta_2$, and it will be used in our discussion.

The space $(\Graphons, \delta_2)$ enjoys a nice geometry like the Wasserstein space. In particular, it allows one to make sense of geodesic convexity of functions defined on $\Graphons$, have  a notion of differentiability, called the Fr\'echet-like differentiability, for $\Hcal\colon\Wcal \to \R$ and gradient flows on it. We define both of these next.

\begin{definition}[Fr\'echet-like derivative on $\Wcal$ {\cite[Definition 20]{Oh2023}}]\label{def:frechet_like_derivative}
    The Fr\'echet-like derivative of $\Hcal\colon \Wcal\to\R$ at $v\in\Wcal$, denoted $D\hamil(v)$, is given by some $\phi(v) \in L^\infty\big(\interval{0,1}^{(2)}\big)$ that satisfies the following condition,
    \begin{align}
        \lim_{\substack{w\in\Wcal,\\\norm{2}{w-v}\to 0}}\frac{\Hcal(w) - \Hcal(v) - \round{ \inner{\phi(v),w}- \inner{\phi(v),v}}}{\norm{2}{w-v}} &= 0,\label{eq:frechet_limit_def}
    \end{align}
    where $\inner{\slot{},\slot{}}$ is the usual inner product on $L^2\big(\interval{0,1}^{(2)}\big)$. If $\Hcal$ admits a Fr\'echet-like derivative at every $v\in\Wcal$, we say that $\Hcal$ is Fr\'echet-like differentiable. 
\end{definition}
The Euclidean gradient is related to this notion of differentiability by scaling. From~\cite[Lemma 4.10]{Oh2023}, $r^2\nabla H_r \equiv D\Hcal$ where $H_r = \Hcal\circ K$ for all $r\in\Natural$.  We will be only concerned  in functions $\Hcal\colon\Wcal\to \R$ that are \emph{invariant}, that is, $\Hcal(w)=\Hcal(w^{\varphi})$ for any Lebesgue measure-preserving map $\varphi\colon[0, 1]\to [0, 1]$. For such functions, the Fr\'echet-like derivative of $\Hcal$ can be thought of as a well-defined function on $\Graphons$ (See ~\cite[Section 4.2]{Oh2023}).

Let $\Hcal\colon \Graphons\to \R$ such that $\Hcal$ is Fr\'echet-like differentiable. For any $W_0\in \Wcal$ one can define curve on $\Wcal$ (and hence a curve on $\Graphons$ via projection) given by 
\begin{equation}\label{eqn:GF}
    w(t)(x, y)=w(0)(x, y)- \int_0^t D\Hcal(w(s))(x, y)\indicator{G_{w(s)}}{}\diff s, \qquad (x,y)\in[0,1]^{(2)},
\end{equation}
for $t \in [0,1]$, where $G_u\subseteq [0,1]^{(2)}$ for any $u\in\Wcal_{[0,1]}$ is defined as
\begin{align}
    \begin{split}
        G_u &\coloneqq \set{(x,y)\in[0,1]^{(2)}\given u(x,y)=1 , D\Hcal(u)(x,y)<0}\\
        &\qquad\cup \set{(x,y)\in[0,1]^{(2)}\given 0<u(x,y)<1}\\
        &\qquad\qquad\cup \set{(x,y)\in[0,1]^{(2)}\given u(x,y)=0 , D\Hcal(u)(x,y)>0}.
    \end{split}
    \label{eqn:G_function}
\end{align}
It is shown in~\cite[Theorem 4.14]{Oh2023} that the curve in $\Graphons$ obtained by projection of the curve defined in~\eqref{eqn:GF} is a \emph{gradient flow} of $\Hcal$ on $(\Graphons, \delta_2)$. Moreover, under appropriate assumptions on $\Hcal$, the Euclidean gradient flow of $H_r$ on $\Mcal_{r, +}$ converges in cut-metric, uniformly on compact subsets of time, to the curve described by~\eqref{eqn:GF}.
\subsubsection{Measure-valued graphons}\label{sec:mvg}
As mentioned already, the convergence of the homomorphism density functions $t(F, \slot{})$ can be used to define a notion of convergence for weighted graphs as well. However, a better approach to the convergence of weighted graphs is given by the convergence of \emph{decorated homomorphism density functions} that we describe below (see~\cite{lovasz2010decorated}). In the following, we will use $I$ to denote the compact interval $[-1, 1]$ and $\Ccal \equiv C(I)$, to denote the space of continuous functions on $I$. 
\begin{definition}[Decorated graph~{\cite[Section 2.1]{lovasz2010decorated}}] 
    Let $m \geq 1$ and $\Dcal\subseteq \Ccal$. Let $F=([m], E)$ be a simple graph. The pair $(F, f)$ is called a $\Dcal$-{\em decorated graph} where $f\colon E(F)\to \Dcal$ is a function from the edges $E(F)$ of $F$ to $\Dcal$. We will refer to $F$ as the \emph{skeleton} and $f$ as the \emph{decoration} of the decorated graph $(F, f)$. If there is no confusion, the decoration of a graph will be implicitly assumed without mention and we will denote $f(\{i, j\})$ by $F_{i,j}$ for $\{i, j\}\in E(F)$.
\end{definition}

Throughout this paper, a decorated graph will mean a $\Ccal$-decorated graph unless stated otherwise.  Let  $W\in\Wfrak$ be a measure-valued kernel (See Definition \ref{def:mvKernel}) and  $F$ a decorated graph. Following~\cite[Section 2.5]{lovasz2010decorated} one can define the (decorated) homomorphism density $t\ped{d}(F, W)$ of $F$ in $W$ as 
\begin{align}
     t\ped{d}(F, W) \coloneqq \int_{[0, 1]^{m}}\round{ \prod_{\{j, k\}\in E(F)}\int_{[-1,1]} F_{j,k}(\zeta)W(x_{j}, x_{k})(\diff \zeta)}\prod_{i=1}^{m}\diff x_{i}.\label{eq:decorated_homomorphism_density_W}
\end{align}

\begin{definition}[Measure-valued graphon~{\cite[Definition 2.4]{lovasz2010decorated}}]\label{def:mvGraphons}
    Define an equivalence relation $\sim$ on $\Wfrak$ such that $W\sim U$ if $t\ped{d}(F, W)=t\ped{d}(F, U)$ for every decorated graph $F$. Let $\mvGraphons\coloneqq \Wfrak/{\sim}$ be equipped with the weakest topology that makes $W\mapsto t\ped{d}(F, W)$ continuous for every decorated graph $F$. We will call $\mvGraphons$ (equipped with this topology), the space of measure-valued graphons. A measure-valued graphon (MVG) is an element in $\mvGraphons$. Naturally,  $W_n\to W$ in $\mvGraphons$ if $t\ped{d}(F, W_n)\to t\ped{d}(F, W)$ for every $\Ccal$-decorated graph $F$. We refer to this topology as the \emph{usual topology} on MVG throughout this paper. 
\end{definition}

Analogous to the space of kernels $\Wcal$, one defines an equivalence relation $\cong$ on $\Wfrak$ such that $W_1\cong W_2$ if there exist measure preserving transformations $\varphi_1,\varphi_2\colon~[0,1]\to[0,1]$ and $W\in\Wfrak$ such that $W_1=W^{\varphi_1}$, and $W_2 = W^{\varphi_2}$. It follows from~\cite[Theorem 11(ii)]{kunszenti2019uniqueness} that $W_1\cong W_2$ if and only if $t\ped{d}(F, W_1)=t\ped{d}(F, W_2)$ for every decorated graph $F$. In particular, the space of MVGs can be equivalently defined as $\mvGraphons\coloneqq \Wfrak/{\cong}$. Wherever it is clear from the context, for any measure-valued kernel $W\in\Wfrak$, we will use an abuse of notation and use the same symbol $W$ to denote the equivalence class, or the measure-valued graphon, corresponding to the measure-valued kernel.

\subsubsection{Embedding matrices and graphons into MVG}\label{subsec:Embedding}
Recall that a weighted graph or (equivalently a symmetric matrix) can be identified with a kernel (and hence a graphon). Similarly, a weighted graph or a graphon can be identified with a measure-valued kernel (and hence a measure-valued graphon). Let $M$ be an $n\times n$ symmetric matrix with entries in $I=[-1, 1]$. Let $\Mcal_n$ denote the set of all such matrices. Let $F$ 
be a $\Ccal$-decorated graph  on $[m]$ vertices. One can define the homomorphism density of $F$ in $M$, denoted $t\ped{d}(F, W)$, as 
\begin{equation}\label{eqn:Def_homDensity_M}
    t\ped{d}(F, M)\coloneqq \frac{1}{n^{m}}\sum_{i_1, \ldots, i_m}\prod_{\set{j,k}\in E(F)} F_{j, k} (M_{i_j, i_k}),
\end{equation}
where the summation runs over all indices $i_1,i_2,\ldots,i_m$ taking values in $[n]$. We make some simple observations. Observe that $t\ped{d}(F, M)=t\ped{d}(F, M^{\sigma})$ where $\sigma$ is any permutation of $[n]$ and  $M^\sigma_{i,j} = M_{\sigma(i),\sigma(j)}$ for all $(i,j)\in[n]^{(2)}$. Also note that one can naturally associate a measure-valued kernel, say $W_M\in\Wfrak$, with a symmetric matrix $M\in\Mcal_n$ as follows. For $n\in\Natural$, let $Q_n \coloneqq \set{Q_{n,i}}_{i\in[n]}$ be a partition of the interval $[0, 1]$ into contiguous intervals of equal length as defined in Section~\ref{sec:graphons,mvgs,IEAs}. Set $W_M(x, y) = \delta_{M_{i,j}}$ whenever $(x,y)\in Q_{n,i}\times Q_{n,j}$ for some $(i,j)\in[n]^{(2)}$. For any decorated graph $F$ we have $t\ped{d}(F, W_M)=t\ped{d}(F, M)$. Therefore, when there is no scope of ambiguity we make no distinction between a symmetric matrix $M$ and the corresponding MVG $W_M$. Similarly, if $w\in\Wcal$ is a graphon, we can define a corresponding MVG, say $W$ by setting $W(x, y)=\delta_{w(x, y)}$ for a.e. $(x,y)\in[0,1]^{(2)}$ and the notion of homomorphism density extends naturally. We denote this map taking a matrix/graphon to the corresponding MVG by $\Kcal$.

Let $\round{M_n\in\Mcal_n}_{n\in\Natural}$ be a sequence of matrices with growing dimension and let $W\in \mvGraphons$. We say that $\round{M_n}_{n\in\Natural}\to W$ as $n\to\infty$ if $t\ped{d}(F, M_n)\to t\ped{d}(F, W)$ as $n\to\infty$ for every decorated graph $F$. In particular, we will often say $\round{M_n}_{n\in\Natural}\to W$ as $n\to\infty$ where $\round{M_n}_{n\in\Natural}$ is a sequence of matrices, or $\round{w_n}_{n\in\Natural}\to W$ as $n\to\infty$ where $\round{w_n}_{n\in\Natural}$ is a sequence of graphons and $W$ is an MVG. It is to be understood that this convergence is with respect to decorated graphs, or equivalently these statements mean that MVG corresponding to $M_n$ (or $w_n$) converge to $W$ in $\mvGraphons$ as $n\to\infty$. For an $n\times n$ finite exchangeable random matrix $X$, we can define a measure valued kernel $W_X$ as $W_X(x,y) = \mathrm{Law}(X_{i,j})$ whenever $(x,y)\in Q_{n,i}\times Q_{n,j}$ for some $(i,j)\in[n]^{(2)}$. We will denote this map by $\Kcal$, i.e., $\Kcal(X) = W_X$. Note that the measure valued kernel cannot recover the joint distribution among the entries of $X$, unless they are mutually independent. 

\begin{remark}\label{rem:MVG_graphon_cnvg}
 Recall that $W\mapsto w\coloneqq \E{W}$ is Lipschitz (see Definition~\ref{def:natural_projection}). It follows that if $(M_n)_{n\in \Natural}$ is a sequence of matrices such that $\round{M_n}_{n\in\Natural}\to W$ as $n\to\infty$ for some $W\in \Wfrak$ in the MVG sense, then $\round{M_n}_{n\in\Natural}\to w$ as $n\to\infty$ in cut-metric as well. This illustrates that convergence in MVG sense is stronger than cut convergence.


\end{remark}

\subsection{Topology and metrics on measure-valued graphons}\label{subsec:Metric_mvg}

In this section we introduce an alternate notion of convergence for MVGs and two metrics on $\mvGraphons$. We then show that this new notion of convergence and the metrics introduced in this section give the same topology on $\mvGraphons$ as defined in Definition~\ref{def:mvGraphons}. 
\begin{definition}[Homomorphism density]
    Let $F$ be a finite simple connected graph and let $W\in \Wfrak$. The \emph{homomorphism density} of $F$ into $W$, denoted $t(F, W)$, is a probability measure on $I_F\coloneqq [-1, 1]^{E(F)}$ is defined as a mixture of probability measures as
    \begin{equation}\label{eqn:Def_homdensity_W}
        t(F, W) \coloneqq \int_{[0, 1]^{V(F)}} \tensor_{\set{i, j}\in E(F)} W(x_i, x_j) \prod_{v\in V(F)} \diff x_{v}.
    \end{equation}
\end{definition}
The measure in~\eqref{eqn:Def_homdensity_W} is to be interpreted as the unique measure on $I_{F}$ such that for any bounded measurable function $\varphi\colon I_F \to \R$, we have 
\begin{align}\label{eqn:Def_homdensity_W_alt}
    \int_{I_F} \varphi(\zeta) t(F, W)(\diff \zeta) &= \int_{[0, 1]^{V(F)}}\round{\int_{I_F} \varphi(\zeta) \tensor_{\set{i, j}\in E(F)} W(x_i, x_j)(\diff \zeta)}\prod_{k\in V(F)}\diff x_{k}.
\end{align}
Let $\{U_i\}_{i\in V(F)}$ be a collection of i.i.d. $\mathrm{Uniform}\squarebrack{0,1}$ random variables, then
\begin{align}\label{eqn:Def_homdensity_W_alt1}
     \int_{I_F} \varphi(\zeta) t(F, W)(\diff \zeta)&= \E{\inner{ \varphi, \tensor_{\set{i, j}\in E(F)} W(U_i, U_j)}},
\end{align}
where $\inner{\slot{},\slot{}}$ is the usual duality between continuous functions on $I_{F}$ and probability measures on $I_{F}$ and expectation is taken with respect to the random variables $\{U_{i}\}_{i\in \Natural}$. It is important to note that the homomorphism density of a simple graph $F$ into a graphon $w$ (see Section~\ref{sec:graphons,mvgs,IEAs}), $t(F, w)$ is a real number in $[0,1]$ whereas the homormorhism density of a simple graph $F$ into an MVG $W$, $t(F, W)$, is a (mixture) of probability measures. Secondly, in the context of MVGs, $t(F,W)$ is defined for a simple graph $F$ and it is a (mixture of) probability measure on $I_F$, on the other hand $t\ped{d}(F,W)$ is defined for a decorated graph $F$ and it is a real number.

\begin{definition}[Convergence of MVGs]\label{def:hom_density_convergence}
    A sequence of MVGs $\round{W_n}_{n\in\Natural}$ converge to a MVG $W$ \emph{in hom-density sense} if
    $\lim\limits_{n\to \infty} t(F, W_n)= t(F, W)$ weakly for every finite simple graph $F$.
\end{definition}

The above definition naturally extends to any measure-valued symmetric matrix $M$. And, using the embedding defined in Section~\ref{subsec:Embedding}, the definition can be naturally extended to symmetric matrices and graphons. We skip the details to avoid repetitions. 

We now introduce the metrics on MVGs. Let $\Lcal$ be the set of all Lipschitz functions $\psi \colon [-1,1] \to \R$ with bounded Lipschitz norm, i.e., $\norm{\rm BL}{\psi}=\max\{\norm{\infty}{\psi}, \norm{{\rm Lip}}{\psi}\}\leq 1$. Define an operator $\Gamma\colon \Lcal\times \Wfrak \to \Wcal$ defined as
\begin{align}
    \Gamma(\psi,W)(x,y) \coloneqq \int_{-1}^{1}\psi(\zeta)W(x,y)(\diff\zeta).\label{def:Gamma}
\end{align}
\begin{definition}[Generalized Cut norm on $\Wfrak$]\label{defn:CutNorm}
    For any $W\in\Wfrak$, define $\norm{\blacksquare}{\slot{}}\colon\Wfrak\to\R_+$ as
    \[
        \norm{\blacksquare}{W} \coloneqq \sup_{\psi\in\Lcal}\cutnorm{\Gamma(\psi,W)},
    \]
    where $\Gamma$ is as defined in~\eqref{def:Gamma}. 
\end{definition}
\begin{remark}\label{rem:matrixGraphon_2_MVG}
    Recall from Section~\ref{subsec:Embedding} that both a kernel and a finite matrix can be associated with a corresponding MVG. With this association, we can reference $\norm{\blacksquare}{w}$ or $\norm{\blacksquare}{A}$ for $w\in\Wcal$ or $A\in \cup_{r\in\Natural}\Mcal_{r}$. That is, the definition of \emph{generalized cut norm} extends to both kernels and matrices. We'll adopt this notation moving forward.
\end{remark}
Recall that $\Tcal$ is the set of all Lebesgue measure preserving maps $\varphi\colon[0, 1]\to [0, 1]$ and for any $W\in\Wfrak$ and $\varphi\in\Tcal$, we define $W^\varphi \in\Wfrak$ as $W^\varphi(x,y) \coloneqq W(\varphi(x),\varphi(y))$ for a.e. $(x,y)\in[0,1]^{(2)}$.

\begin{table}[ht]
    \centering
    \renewcommand{\arraystretch}{2}
    \begin{tabular}{|l c|c l|}
         \multicolumn{2}{c}{\bf Graphons} &  \multicolumn{2}{c}{\bf Measure-Valued Graphons} \\
         \hline
        Cut norm on $\Wcal$ & $\cutnorm{}$ & $\norm{\blacksquare}{}$ & Generalized Cut norm on $\Wfrak$\\
        \hline
        $L^2$ metric on $\Wcal$ & $d_2$ & $D_2$ & $\W_2$ metric on $\Wfrak$\\
        \hline
        Invariant $L^2$ metric on $\Graphons$ & $\delta_2$ & $\Delta_2$ & Invariant $\W_2$ metric on $\mvGraphons$\\
        \hline
        Cut metric on $\Graphons$ & $\delta_\cut$ & \makecell{$\Delta_\blacksquare$\\$\W_\blacksquare$} & \makecell{Generalized Cut metric on $\mvGraphons$\\Wasserstein Cut metric on $\mvGraphons$}\\
        \hline
        Curve on $\Graphons$ & $w\colon t \mapsto w(t)$ & \qquad $W\colon t \mapsto W(t)$ & Curve on $\mvGraphons$\\
        \hline
    \end{tabular}
    \renewcommand{\arraystretch}{1}
    \caption{Table contains notations used for graphons and measure-valued graphons. Each row contains the corresponding notation used in both these settings in the article.}
    \label{tab:notations}
\end{table}
\begin{definition}[Generalized Cut metric on $\mvGraphons$]\label{defn:mvg_cut_metric}
    Define  $\Delta_\blacksquare \colon \mvGraphons \times \mvGraphons \to \R_+$ as
    \begin{align*}
        \Delta_\blacksquare(W_1, W_2) &\coloneqq  \inf_{\varphi_1,\varphi_2\in\Tcal} \norm{\blacksquare}{W_1^{\varphi_1} - W_2^{\varphi_2}}, \qquad W_1,W_2\in\mvGraphons.
    \end{align*}
\end{definition}

Let $\mu_1$ and $\mu_2$ be two finite measures with the same total mass $m$. The extension of the Wasserstein distance between $\mu_1$ and $\mu_2$ is defined as  $\W_1(\mu_1, \mu_2)= \sup_{\psi\in \Lcal}\int \psi \diff(\mu_1-\mu_2)$, where $\Lcal$ is the set of all bounded Lipschitz functions with bounded Lipschitz norm at most $1$. Since we are working with a bounded metric space, this definition is equivalent to the standard definition (see~\cite[Section 1.2.1, Corollary 1.16]{V03}) which considers all Lipschitz functions.

\begin{definition}[Wasserstein Cut metric on $\mvGraphons$]\label{defn:Wass_cut}
    Define $\W_\blacksquare \colon \mvGraphons \times \mvGraphons \to \R_+$ as
    \begin{align*}
        \W_\blacksquare(W_1, W_2) &\coloneqq \inf_{\varphi_1,\varphi_2\in\Tcal} \sup_{S, T\subseteq [0, 1]} \mathbb{W}_1\left(\int_{S\times T} W_1^{\varphi_1}(x, y)\diff x\diff y, \int_{S\times T}W_2^{\varphi_2}(x, y)\diff x\diff y\right).
    \end{align*}
\end{definition}
We also make the following definition. For $W_1, W_2\in \Wfrak$, define the Wasserstein-$2$ metric $D_2$ on $\Wfrak$ as
\begin{align}
    D^{2}_2(W_1, W_2)\coloneqq \int_{[0,1]^2}\mathbb{W}_2^{2}(W_1(x, y), W_{2}(x, y))\diff x\diff y,
\end{align}
where $\mathbb{W}_2$ is the Wasserstein-$2$ metric on $\Pcal([-1,1])$.
\begin{definition}[Invariant Wasserstein-$2$ metric on $\mvGraphons$]\label{def:invariant_wass}
    Define $\Delta_2 \colon \mvGraphons \times \mvGraphons \to \R_+$ as
    \begin{align*}
        \Delta_2^2(W_1, W_2) &\coloneqq \inf_{\varphi_1,\varphi_2\in\Tcal} D_2^2(W_1^{\varphi_1}, W_2^{\varphi_2}),\qquad W_1, W_2\in \mvGraphons.
    \end{align*} 
\end{definition}
Lemma~\ref{lem:ComparisonMetric} in Appendix~\ref{sec:Proof_metric_mvg} shows that $\Delta_{\blacksquare}\leq \Delta_{2}$. 

\begin{proposition}\label{prop:d_D_are_metric}
    Let $\W_\blacksquare$ and $\Delta_\blacksquare$ be as defined in Definitions~\ref{defn:mvg_cut_metric} and~\ref{defn:Wass_cut}. Then, $\W_\blacksquare$ and $\Delta_\blacksquare$ are metrics on $\mvGraphons$. Furthermore, $\W_\blacksquare = \Delta_\blacksquare$.
\end{proposition}
The proof of Proposition~\ref{prop:d_D_are_metric} is deferred to Appendix~\ref{sec:Proof_metric_mvg}. We can now state our first main result. 
\begin{theorem}\label{thm:equivalence_of_cut}
    Let $W, \round{W_n}_{n\in\Natural}\subset\mvGraphons$. Then, the following limits are equivalent, as $n\to \infty$.
    \begin{enumerate}
        \item\label{item:mvg_convg} $W_n\to W$ in $\mvGraphons$, that is, $t\ped{d}(F, W_n)\to t\ped{d}(F, W)$ for every decorated graph $F$.
        \item\label{item:hom_convg} $W_n\to W$ in homomorphism density sense, i.e., $t(F, W_n)\to t(F, W)$ weakly for every finite simple graph $F$.
        \item\label{item:D_cut_convg} $\Delta_\blacksquare(W_n, W)\to 0$.
        \item\label{item:d_cut_convg} $\W_\blacksquare(W_n, W)\to 0$.
    \end{enumerate}
\end{theorem}
   
The proof of Theorem~\ref{thm:equivalence_of_cut} relies on several lemmas which are proved in Appendix~\ref{sec:Proof_metric_mvg}. The proof of Theorem~\ref{thm:equivalence_of_cut} is also deferred to Appendix~\ref{sec:Proof_metric_mvg}.
\begin{remark}\label{rem:Compactness}
    It follows from~\cite[Theorem 17.9]{lovasz2012large} that $\mvGraphons$ is compact (with respect to the usual topology). In order to apply that theorem, notice that our MVG is a $K$-graphon in the terminology of Lov\'asz for $K=[-1,1]$. The set $\mathcal{B}$ can be taken to the countable set of polynomials. By Theorem~\ref{thm:equivalence_of_cut},
    $\big(\mvGraphons,\W_\blacksquare\big)$ (or $(\mvGraphons, \Delta_\blacksquare)$) is a compact metric space.
\end{remark}

It is clear from Definition~\ref{defn:mvg_cut_metric} and Theorem~\ref{thm:equivalence_of_cut} that if $\round{W_n}_{n\in\Natural}\to W$ in $\mvGraphons$ then $\round{\Gamma(\psi, W_n)}_{n\in\Natural}\to \Gamma(\psi, W)$ in $\delta_{\cut}$ for every bounded continuous function $\psi$ defined on $[-1,1]$.
However, the convergence  $W_n\to W$ is stronger since it implies \textit{simultaneous convergence} of all kernels $\Gamma(\psi, W)$. We now give some examples to illustrate the difference between the convergence of graphons and the convergence of MVGs.
\begin{example}\label{exp:moment_graphon}
    For $k\in \Integer_+$, let $\psi_k\colon[-1, 1]\to \R$ be the map given by $\zeta\mapsto \zeta^k$. Let $W\in \Wfrak$. We will call $\Gamma(\psi_k, W)$ the \emph{moment graphon} of $W$ (if we need to emphasize $k$, we will call it \emph{$k$-th moment graphon}). For simplicity, we will also denote $\Gamma(\psi_k, W)$ by $m_k(W)$. It is easy to see that $\round{W_n}_{n\in\Natural}\to W$ in $\mvGraphons$ as $n\to\infty$ implies $\round{m_k(W_n)}_{n\in\Natural}\to m_k(W)$ in $\delta_\cut$ as $n\to\infty$, for all $k\in\Integer_+$. Since the convergence in $\delta_{\cut}$ metric implies that for each $k$, there is a sequence of Lebesgue measure preserving transforms $\varphi_{n, k}\colon[0, 1]\to [0, 1]$ such that $\norm{\cut}{m_k(W_n^{\varphi_{n, k}})-m_k(W)}\to 0$ as $n\to \infty$. However, $W_n\to W$ in $\mvGraphons$ as $n\to\infty$ implies that $\varphi_{n, k}$ could be chosen to be independent of $k$. I.e., there exists a sequence of common `labellings' $(\varphi_n)$ such that $\norm{\cut}{m_k(W_n^{\varphi_{n}})-m_k(W)}\to 0$ as $n\to \infty$. This is what we mean by simultaneous convergence. 
\end{example}

\begin{example}\label{exp:dirac_embedding}
    Consider a sequence of kernels $(a_n)_{n\in\Natural \cup \{\infty\}}$, i.e. $a_n\colon[0, 1]^{(2)}\to [-1, 1]$, for $n\in \Natural \cup \{\infty\}$. For every $n\in\Natural$, define a measure-valued kernel $W_n\in\Graphons$ by setting $W_n(x, y)=\delta_{a_n(x, y)}$, $(x,y)\in[0,1]^2$.
    Let $\psi\in C(I)$ be a continuous test function such that $\norm{\infty}{\psi}\leq 1$. Then
    $\Gamma(\psi, W_n)(x, y)= \psi(a_n(x, y))$, $(x,y)\in[0,1]^{(2)}$.
    Suppose $\round{W_n}_{n\in\Natural}\to W_{\infty}$ in $\mvGraphons$, then $\round{\Gamma(\psi, W_n)}_{n\in\Natural}\to \Gamma(\psi, W_{\infty})$ in $\delta_\cut$. In particular, taking $\psi(z)=z^k$, for every $k\in \N$, it follows that simultaneously $\round{a_n^k}_{n\in\Natural}\to a_{\infty}^k$ in $\delta_\cut$. It is well-known that $\delta_{\cut}(a_n, a)\to 0$ does not imply $\delta_{\cut}(a_n^2, a^2)\to 0$ in general. This illustrates that convergence in the MVG sense is a stronger notion than the cut convergence. 
    
\end{example}
    

\begin{example}\label{exp:Bernoulli_embedding}
    Let $a\colon [0, 1]^{(2)}\to [0, 1]$ be a kernel. Define a measure-valued kernel $W_a$ as $W_{a}(x, y) \coloneqq (1-a(x, y))\delta_0+a(x, y)\delta_1$ for $(x,y)\in\interval{0,1}^{(2)}$. That is, $W(x, y)$ is $\mathrm{Ber}(a(x, y))$ for $(x,y)\in\interval{0,1}^{2}$. Let $\psi$ be any bounded measurable function on $[0, 1]$. Then, 
    $\Gamma(\psi, W_{a})(x,y) = (1-a(x, y))\psi(0)+a(x, y)\psi(1)$.
    If $\round{a_n}_{n\in\Natural}$ is a sequence of graphons such that $\round{W_{a_n}}_{n\in\Natural}\to W_{a}$ then $\round{a_n}_{n\in\Natural}\to a$ in $\delta_{\cut}$. Conversely, in this example, it is easy to verify that if $\round{a_n}_{n\in\Natural}\to a$ in $\delta_{\cut}$ then $\sup_{\psi}\norm{\cut}{\Gamma(\psi, a_n)-\Gamma(\psi, a)}\to 0$ as $n\to\infty$ where the supremum is taken over all continuous and bounded functions $\psi\in C([0,1])$. In particular, we conclude that $\round{W_{a_n}}_{n\in\Natural}\to W_{a}$ in $\mvGraphons$ if and only if $\round{a_n}_{n\in\Natural}\to a$ in $\delta_{\cut}$. 
\end{example}
\begin{example}\label{exp:ternoulli}
    Let $a, b\in \Wcal$ such that $a(x, y)\geq 0, b(x, y)\geq 0$ and $a(x, y)+b(x, y)\leq 1$. Define a ``ternoulli'' MVG as $W_{a, b}(x, y)=a(x, y)\delta_{-1}+(1-a(x, y)-b(x, y))\delta_0+b(x, y)\delta_{-1}$. If $\round{W_n}_{n\in\Natural}\to W_{a, b}$ as $n\to\infty$ in $\mvGraphons$ then $\round{a_n}_{n\in\Natural}\to a, \round{b_n}_{n\in\Natural}\to b$ as $n\to\infty$. Conversely, suppose that $(a_n, b_n)$ are ``coupled graphons'' and $\round{a_n}_{n\in\Natural}\to a$ and $\round{b_n}_{n\in\Natural}\to b$ under a common labeling as $n\to\infty$. I.e., there exists a sequence $\varphi_n\in \Tcal$ such that $\norm{\cut}{a_n^{\varphi_n}-a}+\norm{\cut}{b_n^{\varphi_n}-b}\to 0$ as $n\to \infty$. Note that there exists a common sequence of measure-preserving transforms for both $a_n$ and $b_n$. Then, $\round{W_{a_n, b_n}}_{n\in\Natural}\to W_{a,b}$ as $n\to\infty$. 
\end{example}
\subsubsection{Sampling from MVGs}\label{subsec:samplingMVG}
Recall that one can generate an IEA from an MVG. We now describe a similar procedure to generate a measure-valued random matrix from an MVG. 
Let $\round{U_i}_{i\in \N}$ be an i.i.d. sequence of $\mathrm{Uniform}[0, 1]$ random variables defined on a common probability space, say $(\Omega, \Fcal, \mathbb{P})$. Let $W\in\Wfrak$. For any $n\in \N$  we define the \emph{sampled $n$-MVG}, denoted $\mu(n, W)$, as
\begin{align}
    \mu(n, W)(i, j) \coloneqq W(U_i, U_j),\qquad (i,j)\in [n]^{(2)}.\label{eq:sampled_mvGraphon}
\end{align}
Note that we can identify $\mu(n, W)$ with a random MVG. In the next lemma we show that the random MVG $\mu(n, W)$ converges to $W$ as $n\to\infty$, $\mathbb{P}$-almost surely. 


\begin{lemma}\label{lem:Convergence_of_smaples}
Let $W\in \mvGraphons$. For $n\in\Natural$, let $\mu(n, W)$ be defined as in~\eqref{eq:sampled_mvGraphon}. Then $\mu(n, W)\to W$ in $\mvGraphons$ as $n\to \infty$, $\mathbb{P}$-almost surely. That is,  $\mathbb{P}$-almost surely, for any finite simple graph $F$ we have $t(F, \mu(n, W))\to t(F, W)$, weakly as $n\to\infty$.
\end{lemma}
Lemma~\ref{lem:Convergence_of_smaples} follows directly from~\cite[Theorem 3.8]{kovacs2014multigraph} by taking $\Bcal=C[-1, 1]$ and $\Zcal=M([-1, 1])$ the space of finite Radon measures on $[-1, 1]$. We, therefore, skip the proof of Lemma~\ref{lem:Convergence_of_smaples}. 

We now describe a sampling procedure to generate weighted graphs from an MVG. Let $n\in \N$ and $W\in\Wfrak$. For every $n\in\Natural$, define $\mathbb{G}(n, W)$ to be a random (weighted) graph on $[n]$ with edge-weights $\mathbb{G}(n, W)(i, j)\sim W(U_i, U_j)$ and are conditionally independent given $(U_i, U_j)$ for every $(i,j)\in[n]^{(2)}$. Note that the adjacency matrix of $\mathbb{G}(n, W)$ is a random $n\times n$ symmetric matrix with entries in $I=[-1, 1]$ and we will not make any distinction between the adjacency matrix and the graph. Lemma~\ref{lem:convergence_of_sample_weightedgraph} shows that almost surely $\mathbb{G}(n, W)\to W$ as $n\to\infty$ (see Subsection~\ref{subsec:Embedding}).

\begin{lemma}\label{lem:convergence_of_sample_weightedgraph}
    Let $W\in\mvGraphons$ and let $\mathbb{G}(n, W)$ be defined for every $n\in\Natural$ as above. Then, $\mathbb{P}$-almost surely, $\mathbb{G}(n, W)\to W$ as $n\to\infty$. That is, $\mathbb{P}$-almost surely, for every decorated graph $F$,
    \[
        t\ped{d}(F, \mathbb{G}(n, W))\to t\ped{d}(F, W), \qquad \text{as }n\to\infty.
    \]
\end{lemma}



The proof of Lemma~\ref{lem:convergence_of_sample_weightedgraph} follows essentially the same idea as the proof of~\cite[Theorem 3.8]{kovacs2014multigraph}. We defer the proof to Appendix~\ref{sec:Proofs}. An immediate consequence of Lemma~\ref{lem:convergence_of_sample_weightedgraph} is that every MVG can be obtained as the limit of finite weighted graphs. In fact, MVGs were introduced as the limits of finite weighted in~\cite[Section 2.5]{lovasz2010decorated}.

\subsection{Infinite exchangeable arrays and measure-valued graphons}\label{subsec:ExchangeableArray_and_mvg}
Recall the correspondence between IEAs and random MVGs described in the Introduction. In this section, we formalize that correspondence. 

Let $\Scal$ be the set of all symmetric infinite arrays with their elements taking values in $[-1, 1]$ with $0$ diagonal. That is, let
\[
    \Scal \coloneqq \set{{\bf x} \in \Rd{\Natural^2} \given x_{i,j}=x_{j,i} \in [-1, 1],\;\; x_{i,i}=0\quad \forall\ i,j\in\Natural}.
\]
Equip $\Scal$ with the product topology under which it is compact. Equip $\Scal$ with the corresponding Borel sigma-algebra. Let $\Pi_{\infty}$ be the set of all finite permutations of $\N$. Observe that $\Pi_{\infty}$ has a natural action on $\Scal$ given by ${\bf x}^{\sigma}(i, j)\coloneqq {\bf x}(\sigma(i), \sigma(j))$ for all $(i,j)\in\Natural^2$.  Observe that an IEA is an $\Scal$-valued random variable ${\bf X}$ whose law is invariant under the action of $\Pi_{\infty}$.  Let $\Pcal(\Scal)$ be the space of Borel probability measures on $\Scal$. Let $\Pcal\ped{e}(\Scal)\subseteq \Pcal(\Scal)$ be the set of {\em exchangeable probability measures} on $\Scal$, that is, $\Pcal\ped{e}(\Scal)\coloneqq \set{ \rho \in\Pcal(\Scal) \given \rho = \Law{\X}, \X \text{ is an IEA}}$. Throughout our discussion we will assume that $\Pcal\ped{e}(\Scal)$ inherits the subspace topology from $\Pcal(\Scal)$, that is, it is equipped with the topology of weak convergence unless stated otherwise. 

Recall the correspondence between IEAs and (random) MVGs defined in~\eqref{eqn:AH} in Section~\ref{sec:graphons,mvgs,IEAs}. Theorem~\ref{thm:Correspondence} generalizes~\cite[Theorem 5.3]{diaconis2007graph} and makes formal the idea that IEA are in one-to-one correspondence with random measure-valued graphons. Moreover, the convergence of IEAs is equivalent to convergence of the corresponding MVGs. We first extend the definition of decorated homomorphism densities to an IEA.

\begin{definition}[Homomorphism density of IEAs w.r.t. decorated graphs]\label{def:hom_density_IEA}
    Let ${\bf X}$ be an IEA. For every decorated graph $F$, define 
    $
        t\ped{d}(F, {\bf X})\coloneqq \E{\prod_{\{i, j\}\in E(F)} F_{i,j}(X_{i,j})}.
    $
\end{definition}
The following assertion is immediate from the definition and Theorem~\ref{thm:Correspondence}. For the importance of it, we record it as a Proposition. We skip the proof.
\begin{proposition}\label{prop:Equal_hom_density_excharray_mvg}
    Let ${\bf Y}$ be an IEA and let $W_{\bf Y}$ be the corresponding (random) measure-valued graphon as described above. Then, for any decorated graph $F$ we have $t\ped{d}(F, {\bf Y})= \E{t\ped{d}(F, W_{{\bf Y}})}$.
    In particular, if ${\bf X}, \round{{\bf X}_n}_{n\in \N}$ are infinite exchangeable arrays then $\round{{\bf X}_n}_{n\in\Natural}\to {\bf X}$ weakly as $n\to\infty$ (with respect to the product topology) if and only if $\lim_{n\to\infty}t\ped{d}(F, {\bf X}_n) = t\ped{d}(F, {\bf X})$ for every decorated graph $F$.
\end{proposition}

We now state and prove the main result of this section.

\begin{theorem}[Homeomorphism Theorem]\label{thm:Correspondence}
    Let $\mvGraphons$ be the compact space of MVG equipped with its usual topology. Let $\Pcal(\mvGraphons)$ be the space of Borel probability measures on $\mvGraphons$ equipped with the weak topology. Then, $\Pcal(\mvGraphons)$ is homeomorphic to $\Pcal\ped{e}(\Scal)$.
\end{theorem}
\begin{proof}
Recall that the Aldous-Hoover representation provides a one-to-one correspondence (see~\eqref{eqn:AH}) between IEAs and random MVGs, in other words, between $\Pcal\ped{e}(\Scal)$ and $\Pcal(\mvGraphons)$. Also note that $\Pcal\ped{e}(\Scal)$ and $\Pcal(\mvGraphons)$ are both compact and metrizable (and hence Hausdorff). To show that $\Pcal\ped{e}(\Scal)$ and $\Pcal(\mvGraphons)$ are homeomorphic, it suffices to show that the ${\bf X}\mapsto W_{{\bf X}}$ is continuous. Let ${\bf X}_n$ be a sequence of exchangeable arrays such that ${\bf X}_n\to {\bf X}$ weakly as $n\to \infty$ for some exchangeable array $\X$. Let $W_{n}$ and $W$ be the corresponding (random) measure valued graphons. We want to show that $W_n\to W$ weakly, that is, $\E{t\ped{d}(F, W_n)}\to \E{t\ped{d}(F, W)}$ for every decorated graph $F$. 
To see this, fix a decorated graph $F$. Since $\X_n\to \X$ weakly as $n\to \infty$, it follows that $t\ped{d}(F, \X_n)\to t\ped{d}(F, \X)$ as $n\to \infty$. Observe that $\E{t\ped{d}(F, W_n)}= t\ped{d}(F, {\bf X}_n)$ and $ t\ped{d}(F, {\bf X}) = \E{t\ped{d}(F, W)}$ by Proposition~\ref{prop:Equal_hom_density_excharray_mvg}. Hence, $\E{t\ped{d}(F, W_n)}\to \E{t\ped{d}(F, W)}$ as $n\to \infty$. This completes the proof.
\end{proof}


Recall that by the Aldous-Hoover representation theorem, we know that every exchangeable array $\X$ can be written as $X_{i,j}=f(U, U_i, U_j, U_{i,j})$ for some Borel measurable function $f$. Throughout our discussion, we always assume that $U,\set{U_i}_{i\in\Natural},\set{U_{i,j}=U_{\set{i,j}}}_{i,j\in\Natural}$ is a collection of i.i.d. $\mathrm{Uniform}[0,1]$ random variables on some probability space.  We now give examples of IEAs and their Aldous-Hoover representation which in turn yields the corresponding (random) MVG. These examples emphasize that the graphons do not capture general IEAs while MVGs do.  We first begin with a definition.
\begin{definition}[Vertex, extrinsic, and edge dependence]\label{defn:IEATypes}
    Let $\X$ be an IEA and let $f\colon[0, 1]^{4}\to \R$ be a corresponding Aldous-Hoover function. We say that $f$ has {\em vertex} dependence if $f$ depends on the second and third argument. Similarly, we will say that $f$ has {\em extrinsic} (respectively, {\em edge}) dependence if $f$ depends on the first (respectively, fourth) argument. An IEA that doesn't have {\emph extrinsic} dependence is called \emph{pure} and corresponds to a deterministic MVG.
\end{definition}
We must emphasize that Aldous-Hoover function for an IEA is not unique and is often not known explicitly. However, the above definition does not depend on the choice of Aldous-Hoover function (see~\cite{kallenberg1989representation}).

\begin{example}[Edge dependence - Mixture of two Dirac masses]\label{ex:Bernoulli_finer_topology}
Let ${\bf X}$ be an infinite exchangeable array such that $X_{i,j}$s are all i.i.d. Bernoulli random variables, $\mathrm{Ber}(1/2)$. Let $f\colon[0, 1]^{4}\to \R$ be defined as $f(u, x, y, z)=\indicator{}{z\leq 1/2}$ for $(u, x, y, z)\in [0,1]^4$. We see that $\X$ is directed by $f$.
On the other hand, let $\widetilde{{\bf X}}$ be an IEA such that $\widetilde{X}_{i,j}$s are all i.i.d. (up to matrix symmetry) and $\widetilde{X}_{i,j}\sim \inv{2}\delta_{-1/2} + \inv{2}\delta_{3/2}$. Then, $\widetilde{{\bf X}}$ is directed by an Aldous-Hoover function $g$
where $g\colon[0, 1]^{4}\to \R$ is defined as $g(u, x, y, z)=\frac{1}{2}-\indicator{}{z\leq 1/2}+\indicator{}{z> 1/2}$. Note that the graphons and MVGs corresponding to $\X$ and $\widetilde{\X}$ (see~\eqref{eqn:AH_Graphon} and~\eqref{eqn:AH} in Section~\ref{sec:graphons,mvgs,IEAs}) are given by $w_{{\bf X}} \equiv \frac{1}{2} \equiv w_{\bf \widetilde{X}}$ while $W_{\bf X} \equiv \frac{1}{2}\delta_0+\frac{1}{2}\delta_{1}$ and $W_{\bf \widetilde{X}} \equiv \frac{1}{2}\delta_{-1/2}+\frac{1}{2}\delta_{3/2}$.
Note that the graphons and the measure-valued graphons corresponding to $\X$ and $\widetilde{\X}$ are deterministic. This is reflected by the Aldous-Hoover representations $f$ and $g$ which are both independent of their first coordinates. 
\end{example}

\begin{example}[Extrinsic and edge dependence - correlated Gaussians]\label{ex:correlated_gaussians}
Consider an infinite exchangeable array ${\bf X}$ such that $\set{X_{i,j}}_{(i,j)\in\Natural^{(2)}}$ are standard Gaussian random variables such that $\Cov{X_{i,j}}{X_{l,m}}=1/2$ whenever $\{i, j\}\neq \{l, m\}$. 
Let $\Phi\colon[0, 1]\to \R$ be a function that pushes forward the Lebesgue measure on $[0, 1]$ to the standard Gaussian measure on $\R$. And, let $f\colon[0, 1]^{4}\to \R$ be defined by $f(u, x, y, z)=\frac{1}{\sqrt{2}}\Phi(x)+ \frac{1}{\sqrt{2}}\Phi(z)$ for all $(u,x,y,z)\in[0,1]^4$. It is easy to verify that $\X$ is directed by $f$.
We, therefore, obtain for a.e. $(x,y)\in[0,1]^{(2)}$,
\begin{align*}
    w_{\bf X}(x, y)\equiv \frac{\Phi(U)}{\sqrt{2}},\quad
    W_{\bf X}(x, y)\equiv \Law{\mathcal{N}\round{\frac{\Phi(U)}{\sqrt{2}}, \frac{1}{2}}}.
\end{align*}
Note that $\Phi(U)$ is a standard normal random variable. Also note that $w_{\X}$ is a random kernel and $W_{\X}$ is a random MVG.

Following the same approach as above, one can more generally take $f(u, x, y, z)=\alpha\Phi(u)+\beta(\Phi(x)+\Phi(y))+\gamma\Phi(z),$
say, where $\alpha^2+2\beta^2+\gamma^2=1$. And, define $X_{i,j}=f(U, U_i, U_j, U_{i,j})$ to obtain Gaussian exchangeable arrays with various correlation structures. This would yield 
\[
    w_{\bf X}^{(u)}(x, y) = \alpha\Phi(u)+\beta(\Phi(x)+\Phi(y)),\qquad W_{\bf X}^{(u)}(x, y)= \Law{\round{\Ncal\round{w_{\bf X}^{(u)}(x, y), \gamma^2}}},
\]
for $u,x,y\in[0,1]$.
Because of the extrinsic dependence, this IEA is not pure. Note that in this case the graphons $w_{\bf X}$ and the measure-valued graphon $W_{\bf X}$ are indeed random. 
\end{example}

\begin{example}[Vertex and edge dependence - Stochastic Block Model (SBM)]\label{ex:SBM}
    We now describe an exchangeable array that can be seen as limits of a certain sequence of SBM. Fix $p\in [0, 1]$. For every $n\in \N$, color every vertex $i\in [n]$ blue with probability $p$ and red with probability $(1-p)$ independently of each other. More formally, this is associating an independent $p\delta_1 + (1-p)\delta_{-1}$ distributed random variable $C(i)$ with $i\in [n]$, where $1$ represents color `blue' and $-1$ represents the color `red'. Fix $p_{bb}, p_{rr}, p_{rb}\in [0, 1]$. For each $\set{i, j}\subseteq [n]$, create an edge with probability $p_{bb}$ if both $i$ and $j$ are colored blue, with probability $p_{rr}$ if both are colored red and with probability $p_{rb}$ otherwise. This defines an SBM with two communities `blue' and `red'. Let $A_n$ denote the adjacency matrix of this SBM on vertex set $[n]$. It is easy to see that $A_n$ is an exchangeable matrix, that is, $\Law{A_n}=\Law{A_n^{\sigma}}$. It is natural to ask, if $A_n$ converges to some infinite exchangeable array as $n\to \infty$. This is indeed the case. Here we describe the infinite exchangeable array ${\bf X}$ that arises as the limit of $\round{A_n}_{n\in\Natural}$.
     
    In order to define the Aldous-Hoover function for infinite exchangeable arrays, we first define some sets for notational simplicity. Fix $p\in[0,1]$. Define $B=[0,p]^2$, $R=[p,1]^2$ and $D=[0,p]\times [p,1] \cup [p,1]\times[0,p]$. Let $f\colon[0, 1]^{4}\to \{0, 1\}$ be defined as 
    \[
        f(u, x, y, z) = \indicator{B}{(x,y)}\indicator{[0,p_{bb}]}{z}+\indicator{R}{(x,y)}\indicator{[0,p_{rr}]}{z}+\indicator{D}{(x,y)}\indicator{[0,p_{rb}]}{z},
    \]
    for $u,x,y,z\in[0,1]$. The infinite exchangeable array ${\bf X}$ can be defined as $X_{i,j}\coloneqq f(U, U_i, U_j, U_{i,j})$ for $i,j\in\Natural$. The corresponding graphon and measure-valued graphon are $w^{(u)}$ and $W^{(u)}$ defined as
    \begin{align*}
        w^{(u)}(x, y) &= p_{bb}\indicator{B}{(x,y)}+p_{rr}\indicator{R}{(x,y)}+p_{rb}\indicator{D}{(x,y)},\\
        W^{(u)}(x, y) &= \mathrm{Ber}(p_{bb})\indicator{B}{(x,y)}+\mathrm{Ber}(p_{rr})\indicator{R}{(x,y)}+\mathrm{Ber}(p_{rb})\indicator{D}{(x,y)},
    \end{align*}
    for a.e. $(x,y)\in[0,1]^2$. This example can be generalized to distributions other than Bernoulli in a straightforward manner.  
\end{example}

\subsubsection{From finite exchangeable matrices to IEAs}
The previous section establishes that the weak convergence of a sequence of IEAs is equivalent to the weak convergence of corresponding (random) MVGs (see Theorem~\ref{thm:Correspondence}). In practice, we are often interested in taking limits of finite exchangeable matrices. For instance, we would like to say that $G(n, 1/2)$ converges to the IEA $\G$.  One way to do this is to identify $G(n, 1/2)$ with the corresponding (random) MVG, say $W_{\G_n}$ and show that $W_{\G_n}\to W_{\G}$ where $W_{\G}$ is the MVG corresponding to the IEA $\G$ (see Section~\ref{subsec:Embedding}). However, it is more natural to show the convergence of a sequence of finite exchangeable matrices to an IEA and deduce the convergence to an MVG from there. This is what we do in this section.


A (random) symmetric  matrix $A\in \Mcal_n$ is called (finite) \emph{exchangeable} if $\mathrm{Law}(A)=\mathrm{Law}(A^{\sigma})$ for every permutation $\sigma$ of $[n]$. Given an $n\times n$ exchangeable matrix $A$, we can construct an IEA follows. Let $\set{U_i}_{i\in \N}$ be a family of i.i.d. $\mathrm{Uniform}[0, 1]$ random variables independent of $A$. Define an IEA $\X$ such that $X_{i, j}\coloneqq A_{\lceil nU_i \rceil, \lceil nU_j \rceil }$. 
In plain words, each coordinate (up to matrix symmetry) of $\X$ is chosen independently and uniformly at random from the coordinates of $A$. With this correspondence, for every decorated graph $F$, define $t\ped{finite}^{(0)}(F,\slot{})$ over exchangeable matrices as $t\ped{finite}^{(0)}(F, A)\coloneqq t(F, {\bf X})$. On the other hand, analogous to Definition~\ref{def:hom_density_IEA} we have following definition for homomorphism density into exchangeable matrices.
\begin{definition}[Homomorphism density for exchangeable matrices]
    Let $A$ be an $n\times n$ exchangeable matrix. Let $F$ be a decorated graph such that $\abs{V(F)}<n$. Define 
    $
        t\ped{finite}^{(1)}(F, A) \coloneqq \E{\prod_{\{i, j\}\in E(F)} F_{i,j}\round{A_{i,j}}}.
    $
\end{definition}
\begin{remark}\label{rem:finite_hom} 
    It is easy to see that $\abs{t\ped{finite}^{(1)}(F, A)-t\ped{finite}^{(0)}(F, A)}\leq C(F)n^{-1}$ where $C(F)$ is a constant depending only on $F$. Therefore, both $t\ped{finite}^{(0)}$ and $t\ped{finite}^{(1)}$ determine the same limit as $n\to \infty$.
    Also note that using the embedding described in Section~\ref{subsec:Embedding}, we can define $t\ped{d}(F, A)$ as in equation~\eqref{eqn:Def_homDensity_M}. Notice that $t\ped{finite}^{(0)}(F, A)=\E{t\ped{d}(F, A)}$. 
\end{remark}
This motivates the following definition for the convergence of finite exchangeable matrices to an IEA. 
\begin{definition}[Convergence of exchangeable matrices]\label{def:convegence_matrix_exchangeable_array}
    Let $\round{A_n}_{n\in\Natural}$ be a sequence of $n\times n$ symmetric exchangeable matrices. We say that $\round{A_n}_{n\in\Natural}\to {\bf X}$ as $n\to\infty$ if for every decorated graph $F$ we have $t\ped{finite}^{(1)}(F, A_n)=\E{t\ped{d}(F, A_n)}\to t\ped{d}(F, {\bf X})$ as $n\to\infty$.
\end{definition}
We end this section with some examples of finite exchangeable matrices converging to an IEA. 
\begin{example}
Let $V\colon\mathbb{R}^{(2)}\to [-1, 1]$ be a $C^{2}$ function such that $V(x, y)=V(y, x)$ and $\norm{\infty}{\nabla^2V}\leq 1$, where $\nabla^2V$ is the Hessian of $V$. Define $\Vcal\colon\Pcal(\R) \to \R$ as 
$\Vcal(\mu)\coloneqq \frac{1}{2}\iint_{\Rd{2}} V(x, y)\mu(\diff x)\mu(\diff y)$. Define $\Vcal_n\colon\mathbb{R}^n\to \R$ as $\Vcal_n(x_1, \ldots, x_n)=\Vcal(\mu_n)$ where $\mu_n=\frac{1}{n}\sum_{i=1}^{n} \delta_{x_i}$ for every $\set{x_i}_{i\in[n]}\subset \R$. In particular, $\Vcal_n(x_1, \ldots, x_n)\coloneqq \frac{1}{2n^2}\sum_{i, j=1}^n V(x_i, x_j)$.
Let $H_n({\bf x})\in\Mcal_n$ be the Hessian matrix of $\Vcal_n$ at ${\bf x}\in \R^{n}$. Then, $n^2H_n({\bf x})_{(i, j)} = \pdiff_{1,2}V(x_i, x_j)$ if $i\neq j$, and $n^2H_n({\bf x})_{(i, i)}=\pdiff_{1,1}V(x_i, x_i)$ for $(i,j)\in[n]^{(2)}$. Now, let $\set{X_i}_{i\in\Natural}$ be i.i.d. random variables distributed according to some probability measure $\mu\in \Pcal([-1, 1])$ and let ${\bf X}_n=(X_1, \ldots, X_n)$. Then, $\Hcal^{(n)}=n^2H_n({\bf X}_n)$ is an exchangeable matrix and $\Hcal^{(n)} \to \Hcal^{(\infty)}$, where $\Hcal^{(\infty)}$ is an exchangeable array defined as
\[
    \Hcal^{(\infty)}_{(i, j)} = \begin{cases}
        \partial_{1,2} V(X_i, X_j), & \text{if } i\neq j,\\
        \partial_{1,1} V(X_i, X_i), & \text{if } i=j,
        \end{cases}
        \qquad (i,j)\in\Natural^{(2)}.
\]


For concreteness, assume that $V(x, y)=c(x-y)$ for $(x,y)\in\R^{(2)}$ for some even $C^2$ function $c\colon \R \to [-1,1]$. In that case, notice that $\Hcal^{(\infty)}_{(i, j)}=-c''(X_i- X_j)$ for $(i,j)\in\Natural^{(2)}$. Also assume, for simplicity, that $\set{X_i}_{i\in\Natural}$ are i.i.d. $\mathrm{Uniform}[0, 1]$. Then, $c''$ is the Aldous-Hoover representation function and the MVG corresponding to $\Hcal^{(\infty)}$ is nothing but $W^\infty\in\mvGraphons$ defined as $W^{\infty}(x, y)\coloneqq \delta_{-c''(x-y)}$ for a.e. $(x,y)\in\Rd{(2)}$.
\end{example}

\begin{example}
One can consider higher order polynomials of measures. That is, for any $k\in\Natural\setminus\set{1}$, define $\Vcal\colon\Pcal(\R) \to \R$ as
$
    \Vcal(\mu)\coloneqq \int_{\R^k}V(x_1, \ldots, x_k)\mu(\diff x_1)\ldots \mu(\diff x_k).
$
Define $V_n\colon \Rd{n}\to \R$ as $x\mapsto V_n(x)=\Vcal(\mu_n)$ where $\mu_n\coloneqq \frac{1}{n}\sum_{j=1}^{n}\delta_{x_i}$. This amounts to evaluating the expectation of $V$ when its arguments are sampled uniformly with replacement from the entries of $x\in\Rd{n}$. Let $H_n(x)$ be the Hessian matrix of $V_n$ at $x\in\Rd{n}$. Let us define $G\colon \Rd{2} \to \R$ as $G(a, b)\coloneqq$
\begin{align*}
     & \sum_{i,j\in[k], i\neq j}\int_{\Rd{k-2}}\pdiff_{i,j}V(x_1, \ldots,x_{i-1},a,x_{i+1},\ldots,x_{j-1},b,x_{j+1}, \ldots, x_k)\;\prod_{m\in[k]\setminus \set{i, j}}\mu(\diff x_m).
\end{align*}
Let $\round{X_i}_{i\in\Natural}$ be i.i.d. random variables with distribution $\mu\in\Pcal(\R)$. Then, $ n^{k}H_{n}(X_1, \ldots, X_n)\to \Hmat$, as $n\to\infty$, where $\Hmat_{i,j}=G(X_i, X_j)$ for $(i,j)\in\Natural^{(2)}$.
\end{example}

\begin{example}
    Consider a Markov chain $\round{X_n}_{n\in\Natural}$ on $[0, 1]$ with a unique stationary distribution $\pi\in \Pcal([0, 1])$. Let $W\colon[0,1]^2 \to [0,1]$ be a kernel such that $W$ is continuous $\pi\times \pi$ a.e. For each $n\in\Natural$, let $(Y_1, \ldots, Y_n)$ be a uniform permutation of $(X_1, \ldots, X_n)$ and let $\Hcal^{(n)}$ be an exchangeable matrix defined by
    $\Hcal^{(n)}_{i, j}=W(Y_i, Y_j)$, $i,j\in[n]$.
    Let $\set{V_i}_{i\in\Natural}$ be a collection of i.i.d. random variables with distribution $\pi$ and let $\Hcal^{(\infty)}$ be an exchangeable array such that $\Hcal^{(\infty)}_{i,j}=W(V_i, V_j)$. Then, $\Hcal^{(n)}\to \Hcal^{(\infty)}$ as $n\to \infty$.
\end{example}

\section{Dynamics}\label{subsec:McKeanVlasov_mvg}
In this section we study the limit of exchangeable processes on symmetric matrices as the dimension grows to infinity. In Theorem~\ref{thm:mv}, we show that a general class of processes on symmetric matrices converges to a deterministic curve on $\mvGraphons$ that is described by~\eqref{eq:McKeanV1} as the dimension grows to infinity. In Section~\ref{sec:relax_Metro}, we study the relaxed Metropolis chain algorithm on SBMs described in Section~\ref{sec:intro_metropolis}. We begin with preliminaries on Skorokhod map.
\subsection{Skorokhod map on cube}\label{sec:Skorokhod}
For $n\in\Natural$, consider the domain $\Mcal_n$ (and $\Mcal_{n,+}$). Notice that $\Mcal_n$ is a cube, and is closed with respect to the usual topology.
Consider the SDE:
\begin{align}
    \diff X_{n,(i,j)}(t) &= b_{n,(i,j)}\round{X_{n,(i,j)}(t),X_n(t)}\diff t + \Sigma_{n,(i,j)}\round{X_{n,(i,j)}(t),X_n(t)} \diff B_{n,(i,j)}(t)\nonumber\\
    &\qquad\qquad\qquad+\diff L_{n,(i,j)}^-(t) - \diff L_{n,(i,j)}^+(t),\label{eq:SDE_with_L}
\end{align}
for all $(i,j)\in[n]^{(2)}$ and $t\in\interval{0,T}$ for some fixed $T\in\R_+$ and starting at $X_n(0)=X_{n,0}\in \Mcal_n$. Here $b_n$ and $\Sigma_n$ are maps from $[-1,1]\times \Mcal_n$ to $\Rd{[n]^{(2)}}$ and $\R_+^{[n]^{(2)}}$ respectively. The process $B_{n}$ is a $C\round{[0, T],\Rd{[n]^{(2)}}}$-valued process such that $\round{B_{n,(i,j)}}_{(i,j)\in\squarebrack{n}^{(2)}}$ is a collection of i.i.d standard Brownian motion up to matrix symmetry, and the processes $L_{n}^-$ and $L_{n}^+$ are local times at the boundary. More precisely, they satisfy the following conditions:
\begin{enumerate}
    \item The processes $X_{n}$, $L_{n}^-$ and $L_n^+$ are adapted processes.
    \item The process $L_{n}^-$ and $L_n^+$ are coordinatewise non decreasing processes a.e.
    \item For every $(i,j)\in\squarebrack{n}^{2}$,
    \begin{align}
        \begin{split}
            \int_0^\infty \indicator{}{X_{n,(i,j)}(t) > -1} \diff L_{n,(i,j)}^-(t) &= 0, \quad\text{and}\\
            \int_0^\infty \indicator{}{X_{n,(i,j)}(t) < +1} \diff L_{n,(i,j)}^+(t) &= 0.
        \end{split}
    \end{align}
\end{enumerate}
We say that $\round{X_{n},L_{n}^+,L_{n}^-}$ solves the Skorokhod problem with respect to the set $\Mcal_n$. Following~\cite[Definition 1.2]{kruk2007explicit}, the strong solution $\round{X_{n},L_{n}^+,L_{n}^-}$ of the Skorokhod problem exists and is unique if $b_n$ and $\Sigma_n$ are Lipschitz with respect to $\normF{}$.


Let $Y_1$ and $Y_2$ be two real valued stochastic processes. Let $\Lambda_{\interval{-1,1}}$ denote the Skorokhod map that maps the set c\`adl\`ag functions on $\interval{0,T}$ to itself. If $(X_1 \coloneqq \Lambda_{\interval{-1,1}}(Y_1),L_1^+,L_1^-)$ and $(X_2 \coloneqq \Lambda_{\interval{-1,1}}(Y_2),L_2^+,L_2^-)$ solve the Skorokhod problem with respect to the set $\interval{-1,1}$, then the Skorokhod map $\Lambda_{\interval{-1,1}}$ is $4$-Lipschitz under the uniform metric~\cite[Corollary 1.6]{kruk2007explicit}, i.e.,
\begin{align}
    \sup_{t\in\interval{0,T}}\abs{X_1(t) - X_2(t)} \leq 4\sup_{t\in\interval{0,T}}\abs{Y_1(t) - Y_2(t)}, \qquad\forall\ T\in\R_+.\label{eq:skorokhod_lipschitz}
\end{align}

\subsection{McKean-Vlasov SDE}\label{subsec:McKeanVlasov}
Recall MVG McKean-Vlasov SDE~\eqref{eq:McKeanV1} described in Section~\ref{sec:Introduction}. Following a standard Picard's iteration argument, as done in~\cite[Proposition 4.5]{HOPST22}, it can be shown that MVG McKean-Vlasov SDE~\eqref{eq:McKeanV1} admits a pathwise unique solution under appropriate assumptions on $b$ and $\Sigma$. For completeness, we record this as Proposition~\ref{prop:McKean-Vlasov_existence} but skip the proof.
\begin{assumption}\label{asmp:general_Lipschitz}
    Recall the definition of the generalized cut norm, $\norm{\blacksquare}{}$, from Definition~\ref{defn:CutNorm}. Let $b, \Sigma$ be as in~\eqref{eqn:Def_b_Sigma} and satisfy global Lipschitz conditions, that is, there exists $L,\kappa_\blacksquare \in \R_+$ such that
    \begin{align*}
        \sup_{W\in\Wfrak} \norm{\infty}{b(z_1,W) - b(z_2,W)}, \sup_{W\in\Wfrak} \norm{\infty}{\Sigma(z_1, W) - \Sigma(z_2, W)}&\leq L\abs{z_1-z_2},\\
        \sup_{z\in[-1,1]} \norm{\infty}{b(z,W_1) - b(z,W_2)}, \sup_{z\in[-1,1]} \norm{\infty}{\Sigma(z, W_1) - \Sigma(z, W_2)} &\leq \kappa_\blacksquare \norm{\blacksquare}{W_1-W_2},
    \end{align*}
    for all $W_1,W_2\in\Wfrak$ and $z_1, z_2 \in [-1,1]$. 
\end{assumption}

\begin{proposition}\label{prop:McKean-Vlasov_existence}
Let $b$ and $\Sigma$ be as above. Let $\round{\Omega, \Fcal, \round{\Fcal_t}_{t\in\R_+}, \mathbb{P}}$ be a filtered probability space satisfying the usual conditions that supports a pair of independent $\mathrm{Uniform}[0, 1]$ random variables $U, V$ (measurable with respect to $\Fcal_0$) and a Brownian motion $B$ (adapted to the filtration $\round{\Fcal_t}_{t\in\R_+}$). Then, for any $W_0\in \Wfrak$, there exists a pathwise unique strong solution $(X, W)$ to the MVG McKean-Vlasov SDE~\eqref{eq:McKeanV1}.
\end{proposition}

To motivate the study of McKean-Vlasov SDE~\eqref{eq:McKeanV1}, consider the problem of minimizing triangle density while maximizing the edge density as in Section~\ref{subsec:ComputationExample}. As explained in Section~\ref{subsec:ComputationExample}, a close proxy would be minimize the function $\Hcal(\slot{})\coloneqq t(\triangle,\slot{}) - \alpha t(\mathrel{-},\slot{})$ for sufficiently small $\alpha>0$ over the space of $\Graphons$. Notice that $\Hcal$ naturally restricts to a well-defined function $H_{r}$, on $\Mcal_{r, +}$, the space of symmetric matrices with entries in $[0, 1]$. Consider the following SDE on $\Mcal_{r, +}$
\[
    \diff X_{r, (i, j)}(t)= -r^2\nabla H_r(X_r)+ \sigma \diff B_{r, (i, j)}(t)+ \diff L^{-}_{r, (i, j)}(t) -\diff L^{+}_{r, (i, j)}(t),
\]
with where $\set{B_{r, (i, j)}}_{(i, j)\in \Natural^{(2)}}$ is a collection of i.i.d. Brownian motion. Note that the Euclidean gradient $\nabla H_r$ is scaled by a factor of $r^{2}$. The above SDE can be thought of as a time-scaling of noisy Euclidean gradient flow of $H_{r}$. These processes are considered in~\cite{HOPST22} and they are indeed obtained as the continuous time limit of \emph{projected noisy stochastic gradient descent} (PNSGD) of $H_r$~\cite[Definition 1.2, Theorem 3.2]{HOPST22}. It is clear that $X_r$ is a symmetric matrix valued process whose coordinates are exchangeable. It is shown in~\cite[Theorem 1.4]{HOPST22} that, under appropriate assumptions on the initial condition $X_{r}(0)$, the matrix valued process $(X_r(t))_{t\in\R_+}$ converges to a deterministic curve $(w_{\sigma}(t))_{t\in\R_+}$ on the space of graphons, uniformly on compact time intervals, as $r\to \infty$. Moreover, the curve $(w_{\sigma}(t))_{t\in\R_+}$ is described by a McKean-Vlasov SDE similar to~\eqref{eq:McKeanV1}. 

Note that the graphon convergence of $X_r(\cdot)$ comes with a loss of microscopic information as discussed in Section~\ref{sec:mvg}.
It may be reasonable to expect that the matrix valued process $X_{r}(\cdot)$ converges to an IEA as $r\to \infty$ and one should rather consider the convergence of $X_{r}(\cdot)$ to a deterministic curve on MVG space. In this section we accomplish this convergence and do so for a larger class of $\hamil$ than those allowed in~\cite[Theorem 1.4]{HOPST22}. We begin with an illustrative example.
\begin{example}\label{exp:Example}
    Let $F$ be the triangle graph where two of its edges are decorated by $x\mapsto x^2$, and the third edge is decorated by $x\mapsto x$. Consider the function $\Hcal\coloneqq t\ped{d}(F,\slot{})$. Note that $\Hcal$ restricts naturally to a function $H_r\colon\Mcal_r \to \R$ as defined in equation~\eqref{eqn:Def_homDensity_M}. One can easily see that $\pdiff_{(i,j)} t\ped{d}(F,\slot{})(X_r) = \frac{4}{r^3}\sum_{k=1}^r X_{r,(i,k)}^2X_{r,(k,j)}$ for every $(i,j)\in[r]^{(2)}$. Let $B_{r, (i, j)}$ be a collection of i.i.d. Brownian motions. Consider the following SDE on $\Mcal_r$:
    \[
        \diff X_{r,(i,j)}(t) = -\frac{4}{r}\sum_{k=1}^r X_{r,(i,k)}^2(t)X_{r,(k,j)}(t)\diff t + \diff B_{r,(i,j)}(t) + \diff L^-_{r,(i,j)}(t) - \diff L^+_{r,(i,j)}(t),
    \]
    where $(i,j)\in[r]^{(2)}$ and $(X, L^{+}, L^{-})$ solves the Skorokhod problem. The above SDE can be recovered as a continuous time limit of the PNSGD algorithm~\cite[Definition 1.2, Theorem 3.2]{HOPST22} when we consider $t\ped{d}(F,\slot{})$ to be the optimization objective. We remark that the function $\Hcal$ does not satisfy the assumptions of~\cite[Theorem 1.4]{HOPST22} as it is not continuous in the cut-metric (see~\cite[Example C.3]{janson2010graphons}). However, the function $\Hcal$ does satisfy Assumption~\ref{asmp:general_Lipschitz} for $b$ defined as $b(z,W)(x,y)\coloneqq$ 
    \begin{align}
        &\sum_{\ell=1}^m \E{\prod_{s=1}^{\ell-1}\Gamma(F_{e_s},W)(Z_{e_s})\cdot \Gamma(F'_{e_\ell},W)(Z_{e_\ell})\cdot \prod_{s=\ell+1}^{m}\Gamma(F_{e_s},W)(Z_{e_s})\given Z_{e_\ell}=(x,y)},\nonumber\\
        &\eqqcolon \sum_{\ell=1}^m \mathbf{t}_{x,y}(\pdiff_{e_\ell}F,W),\qquad z\in[-1,1], W\in\Wfrak, (x,y)\in[0,1]^{(2)},
    \end{align}
    where $\set{e_s}_{s=1}^m$ is the set of edges of the skeleton of the triangle graph with $m=3$ edges, $Z_e \coloneqq (Z_{e(1)},Z_{e(2)})$ for an edge $e\in E(F)$ and $\pdiff_e F$ denotes the graph obtained by replacing the decoration at edge $e\in E(F)$ with its derivative. This can be seen by following a very similar argument in~\cite[Example 5]{HOPST22} and Lemma~\ref{lem:D_cut_implies_mvG}.
    
    As a consequence of our main result (Theorem~\ref{thm:mv}), we will see that the solution of the SDE on $\Mcal_r$ as defined above, converges to the solution $X$ to the MVG McKean-Vlasov SDE~\eqref{eq:McKeanV1}, which in this example, takes the form:
    \[
        \begin{split}
            \diff X(t) &= -4m_2(W)(u,v)m_1(W)(u,v)\diff t + \diff B(t) + \diff L^{-}(t) - \diff L^+(t),\\
            W(t)(x,y) &= \Law{X(t)\given (U,V)=(x,y)},\qquad (x,y)\in[0,1]^{(2)},\qquad t\in\R_+,
        \end{split}
    \]
    given $(U,V)=(u,v)$. Here $m_1$ and $m_2$ evaluate the first and second moment graphons as defined in Example~\ref{exp:moment_graphon}. This example, naturally extends to all decorated homomorphism density functions, thereby expanding the scope of~\cite[Theorem 1.4]{HOPST22}. 
\end{example}

More generally we consider the following family of diffusions on symmetric matrices. Let $b$ and $\Sigma$ be as defined in Section~\ref{sec:Introduction}. For $r\in \Natural$, let $\Sigma_r\colon [-1, 1]\times \Mcal_{r}$ be the restrictions of $\Sigma$, i.e., $\Sigma_r(x, X) = \Sigma(x, \Kcal(X))$ and similarly define $b_r\colon [-1, 1]\times \Mcal_{r}$. Consider the diffusions $X_r$ defined on $\Mcal_{r}$ as follows 
\begin{align}
    \begin{split}
        \diff X_{r, (i, j)}(t) &= b_{r,(i,j)}(X_{r, (i,j)}(t), X_r(t))\diff t + \Sigma_{r,(i,j)}(X_{r, (i, j)}(t), X_n(t)) \diff B_{r, (i, j)}(t)\\
        &\qquad\qquad\qquad\qquad\qquad+ \diff L_{r, (i, j)}^{-}(t) - \diff L_{r, (i, j)}^{+}(t),
    \end{split}
    \label{eq:SDE_n}
\end{align}
for each $(i, j)\in [r]^{(2)}$ and $t\in\R_+$, with the initial condition $X_r(0)\in \Mcal_r$.

In Theorem~\ref{thm:mv} we show that under appropriate assumption on $\round{X_r(0)}_{r\in\Natural}$, the process $X_r(\cdot)$ converges, uniformly on compact intervals of time, to a deterministic curve $W(\cdot)$ on MVGs as $r\to \infty$. And, this curve $W(\cdot)$ is described by the McKean-Vlasov system~\eqref{eq:McKeanV1}. 

\begin{theorem}\label{thm:mv}
	Suppose Assumption~\ref{asmp:general_Lipschitz} holds. Let $W_0\in \Wfrak$ and let $W$ be described by the MVG McKean-Vlasov SDE~\eqref{eq:McKeanV1} with initial condition $W(0)=W_0$. Let $X_n$ be the solution of equation~\eqref{eq:SDE_n} for $r=n$ with initial conditions $X_{n}(0)\in \Mcal_n$. Suppose that
    $
        \lim_{n\to \infty}D_2({X_{n}(0), W_0})= 0.
    $
    Then, for any finite time horizon $T>0$, almost surely, 
	\begin{align}
		\lim_{n\to\infty}\sup_{t\in[0,T]}\Delta_\blacksquare\round{\Kcal(X_n(t)),W(t)} &= 0,
	\end{align}
\end{theorem}
The proof of Theorem~\ref{thm:mv} closely parallels the proof of~\cite[Proposition 4.9]{HOPST22}. Therefore, we only give a sketch of the proof highlighting only the crucial differences.

\begin{proof}
    Consider the probability space satisfying Assumption of Proposition~\ref{prop:McKean-Vlasov_existence} and an infinite exchangeable array of diffusions $(X_{i,j})_{(i,j)\in\Natural^{(2)}}$ on it. For $k\in[n]$ and any $t\in\R_+$, consider the sampled $k\times k$ symmetric measure-valued matrix $W(t)[k]$ defined as $W(t)[k](i,j) = W(t)(U_i,U_j)$ for $(i,j)\in[k]^{(2)}$. Consider also the corresponding $k\times k$ matrix of diffusions $X[k](\cdot) \coloneqq \round{X_{i,j}}_{(i,j)\in[k]^{(2)}}$. Now consider $\Kcal(X_n(t))$, the measure-valued finite dimensional kernel from a solution of SDE~\eqref{eq:SDE_n}. One may construct a sampled $k\times k$ measure-valued matrix from this measure-valued finite dimensional kernel as well. We estimate the cut distance of this sampled measure-valued matrix from $W(t)[k]$ by coupling this sampled matrix with $\Kcal(X[k])$ in a particular way.

    Divide $[0, 1]$ into $n$ contiguous intervals of equal length. Let $E_k(n)$ denote the event that that $U_i\in ((m_i-1)/n, m_i/n]$ where each $m_i$, $i\in [k]$, is distinct.
    On this event, we can couple $X_{n, m_i, m_j}(\cdot)$ and $X_{i, j}$ so that they are driven by the same copies of independent Brownian motion and having starting laws $W^{(n)}_0(U_i, U_j)$ and $W_0(U_i, U_j)$ respectively. Our subsequent analysis will be on the event $E_k(n)$ and it is unimportant how the coupling is done on $E^c_k(n)$. For any $i\neq j$ we have $\Prob{\abs{U_i-U_j}}\leq \frac{1}{n}$. Since there are at most $\binom{k}{2}$ distinct pairs $(i, j)\in[k]^2$, a simple union bound yields that $\Prob{E^c_k(n)}\leq k^2/n$.

    Define, $\widetilde{X}_{n,i,j}(t) \coloneqq K(X_n(t))(U_i, U_j),\; (i,j)\in [k]^{2}$. The evolution of $\widetilde{X}_{n,1,2}$, for example, can be described by the SDE
    \begin{align*}
        \begin{split}
            \diff \widetilde{X}_{n,1,2}(t) &= b\round{\widetilde{X}_{n,1,2}(t),\Kcal(X_n(t))}(U_1,U_2)\diff t + \Sigma\round{\widetilde{X}_{n,1,2}(t),\Kcal(X_n(t))}(U_1,U_2) \diff B_{1,2}(t)\\
            &\qquad\qquad + \diff L^-_{n,1,2}(t) - \diff L^+_{n,1,2}(t), 
        \end{split}
    \end{align*}
    with the initial condition $\Law{\widetilde{X}_{n,1,2}(0)}=W^{(n)}_{0}(U_1, U_2)$. Define 
    \[
        M^{(n)}(s) \coloneqq \int_0^s\round{\Sigma(X_{1,2}(s),W(r))(u_1,u_2) - \Sigma\round{\widetilde{X}_{n,1,2}(s),\Kcal(X_n(r))}(u_1,u_2)} \diff B_{1,2}(r),
    \]
   for $s\in[0,t]$. 
   Note that 
   \begin{align*}
       \Prob{\sup_{s\in[0,t]} M^{(n)}(s)\geq \sqrt{\lambda_k\E{M^{(n)}(t)^2}}}&=\Prob{\sup_{s\in[0,t]}\exp\round{u M^{(n)}(s)}\geq \exp\round{\lambda_k}},
   \end{align*}
   where $u=\sqrt{\lambda_k/\E{M^{(n)}(t)^2}}$. Using Markov's inequality followed by Doob's maximal inequality~\cite[page 14, Thoerem 3.8.iv]{KS91}, we obtain that with probability at least $1-4\eu^{-\lambda_k/2}$,
    \begin{align}
        \begin{split}
            \sup_{s\in[0,t]}M^{(n)}(s)^2
            &\leq 2\lambda_k\squarebrack{\kappa_\blacksquare^2 \int_0^t \norm{\blacksquare}{W(s) - \Kcal(X_n(s))}^2 + L^2\abs{X_{1,2}(s) - \widetilde{X}_{n,1,2}(s)}^2\diff s},
        \end{split}
        \label{eq:markov_doob_sigma}
    \end{align}
    where the parameter $\lambda_{k}\to \infty$ will be chosen later. Redefining the event $E_k(n)$ to intersect with the event where the above bound holds. Since $X_{1,2}$ is also driven by the same Brownian motion on this event, using~\eqref{eq:markov_doob_sigma} and the Lipschitz property of the Skorokhod map, triangle inequality and Assumption~\ref{asmp:general_Lipschitz} (replacing $(1,2)$ by any other $(i,j)\in [k]^{(2)}$ and summing over),
   \begin{equation}
        \begin{split}
            \sup_{s\in \interval{0,t}} &\norm{\blacksquare}{\Kcal\round{\widetilde{X}_{n}[k](s)} - \Kcal\round{X[k](s)}}^2 \le \frac{48}{k^2}\sum_{(i,j)\in [k]^{(2)}}\abs{\widetilde{X}_{n,i,j}(0) - X_{i,j}(0)}^2\\
            &+ 96(\lambda_k + 1)\kappa_{\blacksquare}^2 \int_0^t \norm{\blacksquare}{W(s) - \Kcal(X_n(s))}^2\diff s\\
            &+ 96(\lambda_k + 1)L^2\int_0^t\sup_{s\in \interval{0,t}}\norm{\blacksquare}{\Kcal\round{\widetilde{X}_{n}[k](s)} - \Kcal\round{X[k](s)}}^2\diff s.
        \end{split}\label{eq:triangle_lipschitz}
    \end{equation}

\sloppy We now want to replace $\norm{\blacksquare}{\Kcal\round{\widetilde{X}_{n}[k](s)} - \Kcal\round{W(s)[k]}}^2$ by $\norm{\blacksquare}{\Kcal(X_n(s)) - W(s)}^2$ up to some error that goes to zero as $k\to \infty$. This is achieved by exploiting the first sampling lemma~\cite[Lemma 10.6]{lovasz2012large} for cut norm in~\cite{HOPST22}. The first sampling lemma is not available directly to us for the $\norm{\blacksquare}{\slot{}}$. However, we notice that using the first sampling lemma~\cite[Lemma 10.6]{lovasz2012large} and equation~\eqref{eqn:InftoFinite} we obtain that 
 for every $\epsilon>0$ there exists a constant $F_{\epsilon}<\infty$ such that 
    \begin{align*}
        \abs{\norm{\blacksquare}{\Kcal\round{\widetilde{X}_{n}[k](s)} - \Kcal\round{W(s)[k]}}^2 - \norm{\blacksquare}{\Kcal(X_n(s)) - W(s)}^2} \leq \frac{1}{k^{1/4}}+\epsilon,
    \end{align*}
    with probability at least $F_{\epsilon}\eu^{-\sqrt{k}/10}$. Moreover, from Lemma~\ref{lem:finitesufficiency}, we can choose $\epsilon_{k}=\frac{64}{k^{1/4}}$ so that $F_{\epsilon_{k}}\leq \eu^{\sqrt{k}/40}$. In particular, setting $C_{k}=\frac{65}{k^{1/4}}$ and $c_{k}=\sqrt{k}/20$ we can repeat the same proof as in~\cite{HOPST22} to obtain
    \begin{align}
        \begin{split}
            \sup_{s\in \interval{0,t}}&\norm{\blacksquare}{\Kcal\round{\widetilde{X}_{n}[k](s)} - \Kcal\round{X[k](s)}}^2 \geq \inv{2}\norm{\blacksquare}{ \Kcal(X_n(s)) - W(s) }^2 - C_k \\
            &\qquad- \sup_{s\in \interval{0,t}}\norm{\blacksquare}{ \Kcal\round{W(s)[k]} - \Kcal\round{X[k](s)}}^2.
    \end{split}\label{eq:after_using_lovasz}
    \end{align}
    with probability at least $1-\eu^{-c_k}$. Once again we redefine the event $E_{k}(n)$ to intersect with the event where the above bound holds and note that we still have $\Prob{E_k(n)}\geq 1-4\eu^{-\lambda_k/2}-\frac{k^2}{n}-\eu^{-c_k}$.
We can now repeat the same argument as in~\cite{HOPST22}. After doing some rearrangement and applying Gr\"onwall's inequality~\cite{gronwall1919note} we obtain that on the event $E_k(n)$,
    \begin{align}
        \begin{split}
            \sup_{s\in \interval{0,t}} &D_2^2\round{\Kcal\round{\widetilde{X}_n[k](s)},\Kcal\round{X[k](s)}} + \sup_{s\in \interval{0,t}}\norm{\blacksquare}{ \Kcal(X_n(s)) - W(s) }^2\\
            &\le 2\left(  A_k + B_k(n)  \right) \exp\round{192(L^2+2\kappa_\blacksquare^2)(\lambda_k + 1)t},
        \end{split}
        \label{eq:semifinal_rate}
    \end{align} where $A_k=\sup_{s\in \interval{0,t}}\norm{\blacksquare}{\widetilde{A}_k(s)}^2$ and
     \begin{align}
        \begin{split}
            \widetilde{A}_k(s) &\coloneqq  \Kcal\round{W(s)[k]} - \Kcal\round{X[k](s)},\\
            B_k(n) &\coloneqq C_k + \frac{96}{k^2}\sum_{(i,j)\in [k]^{(2)}}\abs{\widetilde{X}_{n,i,j}(0) - X_{i,j}(0)}^2.
        \end{split}
        \label{eq:def_Ak_Bk}
    \end{align}
    Note that $\lim_{n\rightarrow \infty}\E{\abs{\widetilde{X}_{n,i,j}(0) - X_{i,j}(0)}^2}=0$, by the assumption. Using a variance bound  and the fact that $\lim_{k\to \infty}C_k\to 0$ it follows that $\lim_{k\rightarrow \infty} \lim_{n\rightarrow \infty} B_k(n)=0$, in probability. By Lemma~\ref{lem:convergence_of_sample_weightedgraph} and Lemma~\ref{lem:Convergence_of_smaples} we have $\norm{\blacksquare}{\widetilde{A}_k(s)}\to 0$ in probability for each fixed $s\in[0,t]$ as $k\to \infty$. It can be shown following the proof of~\cite[Proposition 4.5]{HOPST22}) that $\round{\widetilde{A}_k}_{k\in\Natural}$ is equicontinuous over $[0, t]$, almost surely, for sufficiently large $k$. Therefore, we conclude that $A_k\to 0$ in probability as $k\to \infty$. Since $\lim_{k\to\infty}\lim_{n\rightarrow \infty}\Prob{E_k(n)}=1$,
    \begin{align*}
        \begin{split}
            \lim_{n\rightarrow \infty}\sup_{s\in\interval{0,t}}&\norm{\blacksquare}{\Kcal(X_n(s)) - W(s)} =0,\ 
            \lim_{k\to\infty}\lim_{n \to \infty}\sup_{s\in \interval{0,t}} D_2^2\round{\Kcal\round{\widetilde{X}_n[k](s)},\Kcal\round{X[k](s)}} = 0,
        \end{split}
    \end{align*}
    in probability, by choosing $\round{\lambda_k}_{k\in\Natural}$ (depending on $\round{A_k,\lim_{n\to\infty} B_k(n)}_{k\in\Natural}$) that increases sufficiently slowly to infinity as $k\to\infty$. Moreover, it is clear that one can choose $k=o(\sqrt{n})$. This completes the proof.
\end{proof}
\begin{remark}\label{rem:Perturbation}
    Note that the proof is stable with respect to small perturbations of drift. More precisely, suppose $b_r$ in~\eqref{eq:SDE_n} is replaced by $\widetilde{b}_r$ such that $\norm{\infty}{\widetilde{b}_r-b_r}\leq \alpha_{r}$ for some $\alpha_r\to 0$ as $r\to \infty$. Then, the proof continues to hold and we still obtain the limiting McKean-Vlasov SDE with the same drift $b$ as in Theorem~\ref{thm:mv}.
\end{remark}

\begin{remark}\label{rem:MVG_Graphon_MKV}
    \sloppy In practice, we are often interested in the functions $b$ and $\Sigma$ in~\eqref{eq:SDE_n} that are obtained as the pull back of some functions $b_0, \Sigma_0$ defined on $[-1, 1]\times \Wcal$ as defined in Section~\ref{sec:intro_metropolis}. Throughout this remark, we always assume this. It follows from Remark~\ref{rem:MVG_graphon_cnvg} (i.e. by Lipschitzness of $W\mapsto w=\E{W}$) that under appropriate conditions on the $X_r(0)$, solution of SDE~\eqref{eq:SDE_n} converges, in probability, to a deterministic curve $(w(t))_{t\geq 0}$ on the space of graphons uniformly on compact time intervals.
    
    Furthermore, note that the McKean-Vlasov equation~\eqref{eq:McKeanV1} depends only on $w(t)=\E{W(t)}$. One can, therefore, say that $w(t)$ satisfies a graphon McKean-Vlasov SDE that satisfies on $(U, V)=(u, v)$
    \begin{align}\label{eq:Graphon_MKV}
    \begin{split}
       \diff X(t) &= b_0(X(t), w(t))+ \Sigma_0(X(t), w(t))+ \diff L^{-}(t) - \diff L^{+}(t),\\
       w(t)(x, y) &= \E{X(t)\given (U, V)=(x, y)},
    \end{split}
    \qquad t\in\R_+.
    \end{align}
    Thus we recover~\cite[Theorem 1.4]{HOPST22}. 

   We should emphasize the point made in Remark~\ref{rem:MVG_graphon_cnvg} once again here. The same IEA gives rise to both McKean-Vlasov SDEs~\eqref{eq:McKeanV1} and~\eqref{eq:Graphon_MKV}. The crucial difference is $X_r(\cdot)$ converging to this IEA in cut-metric is equivalent to only checking the convergence of homomorphism densities with respect to simple graphs, while  $X_r(\cdot)$ converging to this IEA in MVG is equivalent to the convergence of homomorphism densities with respect to decorated simple graphs which is a bigger class of test functions. 
\end{remark}

\subsection{Analysis of the relaxed Metropolis chain}\label{sec:relax_Metro}
Recall that our goal is to minimize some function $\Hcal$ defined on large networks. One class of functions that is of interest is linear combination of homomorphism densities. The key takeaway of the above discussion is that one can either solve for the McKean-Vlasov system described above or perform (stochastic) gradient flow of $\Hcal$ restricted to $\Mcal_r$ for large $r$. Both of these techniques will yield approximate minimizers of $\hamil$. However, notice that both these techniques give a graphon/MVG or symmetric matrices with entries in $[-1, 1]$. In other words, these methods do not directly yield a curve on the space of graphs. Naturally, given a function $\Hcal$ defined on all graphs, one would want to run a Markov chain directly on the graphs that mimics the behaviour of McKean-Vlasov SDE~\eqref{eq:McKeanV1}. This is what is achieved by our relaxed Metropolis chain that we study in this section. Another crucial feature of the relaxed Metropolis chain is that it does not need the access to the gradient of $\Hcal$ but it can still mimic the gradient flow.


\begin{assumption}\label{asmp:hamil}
    Let $\Hcal\colon\Graphons \to \R$ be bounded below, Fr\'echet-like differentiable with $D\Hcal$ denoting its Fr\'echet-like derivative (see Definition~\ref{def:frechet_like_derivative}) and satisfy
    \begin{align}
        \frac{\lambda}{2}\enorm{u-v}^2 \leq \Hcal(v) - \Hcal(u) - \inner{D\Hcal(u),v-u} \leq \frac{L}{2}\enorm{u-v}^2,
    \end{align}
    for every $u,v\in\Wcal_{[0,1]}$, for some constants $\lambda\in \mathbb{R}$ and $L>0$.
    Furthermore, assume that $D\Hcal$ is Lipschitz, that is, there exists $\kappa_\blacksquare>0$ such that for all $u, v\in \Wcal$
    \begin{align}\label{eqn:Asmp_hamilLip}
         \norm{\infty}{D\Hcal(u)-D\Hcal(v)}\leq \kappa_\blacksquare\norm{\blacksquare}{u-v},
    \end{align}
    where $\norm{\blacksquare}{u-v}$ is defined as in Remark~\ref{rem:matrixGraphon_2_MVG}.
\end{assumption}

\sloppy  Recall that $W\mapsto \E{W}$ is $(\norm{\blacksquare}{} \to \cutnorm{})$-Lipschitz. In particular, if $w\mapsto D\hamil(w)$ is $(\cutnorm{}\to \norm{\infty}{})$-Lipschitz, then~\eqref{eqn:Asmp_hamilLip} holds. All decorated homomorphism density functions satisfy Assumption~\ref{asmp:hamil}. Using a similar argument as in~\cite[Section 5.1.2]{Oh2023}, it can be shown that $\max\set{\abs{\lambda},L}$ for the decorated homomorphism density function of a decorated graph $H$ is bounded by $\abs{E(H)}\abs{V(H)}(\abs{V(H)}-1)$, and $\kappa_\blacksquare = \abs{E(H)}\round{\abs{E(H)}-1}$ following~\cite[Section 5, equation (84)]{HOPST22}.

\begin{definition}\label{def:bFunction}
    Let $\Hcal$ satisfy Assumptions~\ref{asmp:hamil}. Let $\beta >0$. Define $b_0\colon \Wcal_{[0,1]}\to L^\infty\big([0,1]^{(2)}\big)$ as
    \begin{align*}
        b_0(w) \coloneqq -2\beta D\Hcal(w)\exp\round{\beta^2r^{-2}\enorm{D\Hcal(w)}^2}\overline{\Phi}\round{\sqrt{2}\beta r^{-1}\enorm{D\Hcal(w)}},\qquad w\in\Wcal_{[0,1]},
    \end{align*}
    where $\overline{\Phi}$ is the right tail of standard Gaussian, i.e., $\overline{\Phi}(x) = \frac{1}{\sqrt{2\pi}}\int_{x}^{\infty}\exp\round{-y^2/2}\diff y$, for $x\in\R_+$. For any $r\in\Natural$, we will denote the restriction of $b_0$ to $\Mcal_{r}$ by $b_{r}$. That is, $b_{r}=M_r\circ b_0\circ K$.
\end{definition}
Recall that $b_0$ defined as above naturally defines a function on $\Wfrak$ via pullback. In the context of the Metropolis chain algorithm, with an abuse of notation, we will use the notation $b$ and $b_0$ interchangeably. By Lemma~\ref{lemma:ExplicitDrift} $r^{-4}b_{r}(w) = \E{Z\exp\round{-\beta_{n, r}\gamma_n\inner{\nabla H_{r}, Z}_{F}^{+}}}$ where $Z$ is $r\times r$ symmetric matrix with i.i.d. Gaussian entries, and $\beta_{n,r}\coloneqq \beta r^{-2}/\gamma_n$ as defined in Section~\ref{sec:intro_metropolis}. Thus $\norm{\infty}{b_{r}}< \infty$. 
\begin{remark}\label{rem:DriftConvergence}
    It follows from~\eqref{eqn:Asmp_hamilLip} that $\norm{\infty}{D\hamil}\leq C$ for some $C$ and therefore $\norm{2}{D\hamil(v)}\leq C$ for every $v\in \Wcal$. Since $\eu^{x}\to 1$ as $x\to 0$ and $\overline{\Phi}(x)\to \frac{1}{2}$ as $x\to 0$, it follows that $\norm{\infty}{b_0(w)+\beta D\Hcal}\to 0$ as $r\to \infty$. 
\end{remark}


We now recall, from Section \ref{sec:intro_metropolis}, the Metropolis algorithm to sample from the $\ESBM{r,n,\beta,\hamil}$. Given $G(k)\in \statesp_{n,r}$ and the matrix $q^{(n)}_{r,k}\in\Mcal_{r,+}$ of edge-densities for any $k\in\Integer_+$, we run the relaxed Metropolis chain consisting of the following steps.
\begin{enumerate}
    \item\label{step1} Run the base chain for $s_n \coloneqq \ceil{\gamma_n^2 n^4}$ many steps. Let $\widetilde{G}(k+1)\in\statesp_{n,r}$ be the graph obtained after such $s_n$ many steps. Let $\widetilde{q}^{(n)}_{r,k+1}$ denote the matrix of edge densities of $\widetilde{G}(k+1)$. 
    \item Given $G(k), \widetilde{G}(k+1)$ for any $k\in\Integer_+$, define
    \[
        Y(k+1) = \begin{cases}
                    \widetilde{G}(k+1),& \quad \text{w.p.}\quad \exp\left( -\beta_{n,r}\squarebrack{ H_r\round{\widetilde{q}^{(n)}_{r,k+1}} - H_r\round{q^{(n)}_{r,k}}}^+ \right),\\
                    G(k), & \quad \text{otherwise},
                \end{cases}
    \]
    where $a^+=\max\set{0,a}$ for $a\in\R$ and $\beta_{n,r}\coloneqq \beta r^{-2}/\gamma_n$ as defined in Section~\ref{sec:intro_metropolis}. Let $p^{(n)}_{r,k+1}$ be the matrix of edge-densities for $Y(k+1)$. Observe that
    \[
        p^{(n)}_{r,k+1} = \begin{cases}
                    \widetilde{q}^{(n)}_{r,k+1},& \quad \text{if }\quad Y(k+1)=\widetilde{G}(k+1),\\
                    q^{(n)}_{r,k}, & \quad \text{if }\quad Y(k+1)=G(k).
                \end{cases}
    \]
    \item After the accept-reject step, we again run the base chain starting from $Y(k+1)\in\statesp_{n,r}$ for $\ell_{n,r} \coloneqq \ceil{r^{-4}\sigma^2\gamma_nn^4}$ many steps for some $\sigma>0$. Let the graph obtained thereafter be $G(k+1)$, and let $q^{(n)}_{r,k+1}$ be the edge density matrix of $G(k+1)$.
\end{enumerate}

This procedure gives a Markov chain $(G(k))_{k\in \Natural}$ on the state space $\statesp_{n,r}$ with corresponding process of edge-density matrix $\round{q^{(n)}_{r,k}}_{k\in \Natural}$. In the following we show that, as $n\rightarrow \infty$ and $r\rightarrow \infty$, the latter process converges to a MVG McKean-Vlasov SDE with drift $b$ as defined in Definition~\ref{def:bFunction}. Taking a natural projection of this MVG curve to the space of graphons, we recover a deterministic curve on the space of graphons. For fixed $r$, the adjacency matrix of $G(k)$ converges to the corresponding edge-density matrix $q^{(n)}_{r, k}$ in the cut metric as $n\to \infty$ uniformly. We can thus interpret the limiting deterministic curve on the space of graphons as the cut limit of the process of adjacency matrices of $(G(k))_{k\in \Natural}$ as $n\to\infty$ followed by $r\to \infty$.

\subsubsection{A heuristic analysis of the edge-density process}
Before we state our result we analyse heuristically the process of edge-density matrix $\round{q^{(n)}_{r,k}}_{k\in\Integer_+}$ as defined above. To this end, for any $k\in\Integer_+$, define $\Delta q^{(n)}_{r,k} \coloneqq q^{(n)}_{r,k+1}-q^{(n)}_{r,k}$ and let $\Fcal_k$ be the sigma algebra generated by $\set{q^{(n)}_{r,i}\given i=0,\ldots,k}$. Given $k\in \Integer_+$, let us analyze the $\E{\Delta q^{(n)}_{r,k}\given \Fcal_k}$. Notice that 
\[
    \E{\Delta q^{(n)}_{r,k}\given \Fcal_k} = \E{\widetilde{\Delta} q^{(n)}_{r,k}\given \Fcal_k}+ \E{ q^{(n)}_{r,k+1}-p^{(n)}_{r,k+1}\given \Fcal_k},
\]
where $\widetilde{\Delta} q^{(n)}_{r,k} \coloneqq p^{(n)}_{r,k+1} - q^{(n)}_{r,k}$.
Notice that given $p^{(n)}_{r,k+1}$, the increment of the $(i, j)$-th coordinate, $q^{(n)}_{r,k+1,(i,j)}-p^{(n)}_{r,k+1,(i,j)}$, has the same distribution as the reflected random walk of step-size $\frac{1}{n^2}$ run for $\ell_{n,r}$ steps. Assuming that the random walk does not hit the boundary during relaxation (hitting the boundary is rare), this is very small. It follows that $\E{\Delta q^{(n)}_{r,k}\given \Fcal_k} \approx \E{\widetilde{\Delta} q^{(n)}_{r,k}\given \Fcal_k}$
\begin{align}\label{eqn:DeltaApprox}
    &=\E{\round{\widetilde{q}^{(n)}_{r,k+1}-q^{(n)}_{r,k}}\exp\round{-\beta_{n, r}\squarebrack{\Hcal\round{K\round{\widetilde{q}^{(n)}_{r,k+1}}}-\Hcal\round{K\round{q^{(n)}_{r,k}}}}^{+}}}.
\end{align}
By Assumption~\ref{asmp:hamil}, 
\begin{align*}
    \Hcal\round{K\round{\widetilde{q}^{(n)}_{r,k+1}}}-\Hcal\round{K\round{q^{(n)}_{r,k}}} &- \inner{D\Hcal\round{K\round{q^{(n)}_{r,k}}}, K\round{\widetilde{q}^{(n)}_{r,k+1}}-K\round{q^{(n)}_{r,k}}}\\
    &\qquad \approx \norm{2}{K\round{\widetilde{q}^{(n)}_{r,k+1}}-K\round{q^{(n)}_{r,k}}}^2.
\end{align*}

On the other hand, given $q^{(n)}_{r,k}$, the increment of each coordinate $\widetilde{q}^{(n)}_{r,k+1,(i,j)}-q^{(n)}_{r,k,(i,j)}$ for every $(i,j)\in[r]^{(2)}$ has the same distribution as a symmetric random walk (with reflections at the boundary) with step-size ${1}/{n^2}$ run for $s_n=\gamma_n^2 n^4$ steps. In particular, $\E{\norm{2}{K\round{\widetilde{q}^{(n)}_{r,k+1}}-K\round{q^{(n)}_{r,k}}}^2}=\gamma_n^2$. Two important and non-trivial consequences of this heuristic are the following:
\begin{enumerate}
    \item Due to a concentration of measure argument, $\norm{2}{K\round{\widetilde{q}^{(n)}_{r,k+1}}-K\round{q^{(n)}_{r,k}}}^2\leq C\gamma_n^2\log n$ for some constant $C>0$ with high probability. 
    \item $\widetilde{q}^{(n)}_{r,k+1}-q^{(n)}_{r,k}$ has approximately the same distribution as $\gamma_n Y_r$ where $Y_r$ is an $r\times r$ symmetric matrix of independent standard Gaussians. Notice that if $q^{(n)}_{r,k,(i,j)}\in \{0, 1\}$ for any $(i,j)\in[r]^{(2)}$, this is not true. This approximation is valid only when all the coordinates of $q^{(n)}_{r,k}$ are sufficiently away from $\{0, 1\}$. With a careful analysis, one can show that this is indeed the case except for a negligible fraction of time.
\end{enumerate}
Assuming the above heuristics and using equation~\eqref{eqn:DeltaApprox} we obtain that with high probability, $\E{\Delta q^{(n)}_{r,k}\given \Fcal_k}$
\begin{align}\label{eqn:DeltaApprox1}
    &\approx \E{\round{\widetilde{q}^{(n)}_{r,k+1}-q^{(n)}_{r,k}}\exp\round{-\beta_{n, r}\inner{D\Hcal\round{K\round{q^{(n)}_{r,k}}}, K\round{\widetilde{q}^{(n)}_{r,k+1}}-K\round{q^{(n)}_{r,k}}}^{+}}}\nonumber\\
    &=\gamma_n\E{Y_r\exp\round{-\beta_{n, r}\gamma_n\inner{\nabla H_r\round{q^{(n)}_{r,k}}, Y_r}_{\rm F}^{+}}},
\end{align}
where we used the fact that for any two $r\times r$ symmetric matrices $A, B\in \Mcal_{r}$ we have $\inner{A, B}_{F}=r^2\inner{K(A), K(B)}$ and that the fact that $D\Hcal = r^{-2}\nabla H$ (see~\cite[Lemma 4.10]{Oh2023}). 
The expectation in the last expression above is very amenable to analysis. It follows from Lemma~\ref{lemma:ExplicitDrift} that $\E{Y_r\exp\round{-\beta_{n, r}\gamma_n\inner{\nabla H_r\round{q^{(n)}_{r,k}}, Y_r}_{\rm F}^{+}}} = r^{-4}b_r\round{q^{(n)}_{r,k}}$ where $b_r$ is defined in Definition~\ref{def:bFunction}.

The above heuristic can now be summarised as follows. With high probability 
\begin{equation}\label{eqn:DriftHeuristicapprox}
    \E{\Delta q^{(n)}_{r,k}\given \Fcal_k}\approx \gamma_nr^{-4}b_r\round{q^{(n)}_{r,k}},
\end{equation}
provided that the all coordinates of $q^{(n)}_{r,k}$ are away from $\{0, 1\}$. Now let us analyse the conditional covariance of $\Delta q^{(n)}_{r,k}$. Recall that $\E{\norm{2}{K\round{\widetilde{q}^{(n)}_{r,k+1}}-K\round{q^{(n)}_{r,k}}}^2\given \Fcal_k}\leq \gamma_n^2$. On the other hand, given $p^{(n)}_{r,k+1}$, the increment $q^{(n)}_{r,k+1,(i,j)}-p^{(n)}_{r,k+1,(i,j)}$ 
of coordinate $(i, j)\in[r]^{(2)}$ has the same distribution as the symmetric random walk with step-size $\frac{1}{n^2}$ (reflected at the boundary $\set{0,1}$) running for $\ell_{n,r}\approx r^{-4}\gamma_n\sigma^2n^4$ steps.
In particular, each coordinate has variance $\approx r^{-4}\gamma_n\sigma^2$. Also note that given $p^{(n)}_{r,k}$, the coordinates of $q^{(n)}_{r,k+1}-p^{(n)}_{r,k+1}$ are independent. In particular, 
\begin{align}\label{eqn:conditionalcovariance}
    \mathrm{Cov}\round{\Delta q^{(n)}_{r,k}\,\Big\vert\, \Fcal_k} = r^{-4}\gamma_n\sigma^2 I + O(\gamma_n^2),
\end{align}
where $O(\gamma_n^2)$ means that each coordinate of $\mathrm{Cov}\round{\Delta q^{(n)}_{r,k}\,\Big\vert\, \Fcal_k}$ differs from $r^{-4}\gamma_n\sigma^2I$ at most by a constant factor of $\gamma_n^2$.

For a fix $t>0$, we will define $t_{n, r}\coloneqq \floor{tr^4/\gamma_n}$. Also define $q^{n}_r\colon\R_+ \to \Mcal_{r,+}$ to be a piecewise constant interpolation of $\round{q^{(n)}_{r,k}}_{k\in\Integer_+}$ given by
\begin{align}
    q^{(n)}_{r}(t) \coloneqq q^{(n)}_{r,t_{n,r}},\qquad t\in\R_+,\label{eq:q^n(t)}
\end{align}
In particular, we obtain 
\begin{align*}
    q^{(n)}_{r}(t) - q^{(n)}_{r}(0)=\sum_{k=0}^{t_{n,r}-1}\E{\Delta q^{(n)}_{r,k}\given \Fcal_k} + \sum_{k=0}^{t_{n, r}-1}\round{\Delta q^{(n)}_{r,k}-\E{\Delta q^{(n)}_{r,k}\given \Fcal_k}},\qquad t\in\R_+.
\end{align*}
Using the heuristic derived in~\eqref{eqn:DriftHeuristicapprox} and~\eqref{eqn:conditionalcovariance}, one expects that 
\begin{align}\label{eqn:Heur}
    q^{(n)}_{r}(t) - q^{(n)}_{r}(0) \approx \sum_{k=0}^{t_{n,r}-1}\gamma_nr^{-4}b_{r}\round{q^{(n)}_{r,k}} + \sum_{k=0}^{t_{n, r}-1}\Delta M^{(n)}_{r,k},\qquad t\in\R_+,
\end{align}
where $\round{\Delta M^{(n)}_{r,k}}_{k\in\Integer_+}$ is a $\Mcal_r$-valued martingale difference sequence with uniform coordinatewise variance $\gamma_nr^{-4}\sigma^2$. We must caution that the approximation in~\eqref{eqn:Heur} is not valid if $q^{(n)}_{r,k}$ is close to boundary $\set{0,1}$. The heuristic calculations have been derived under the assumption that all coordinates of $q^{(n)}_{r,k}$ are away from $\{0, 1\}$. Ignoring this boundary contribution, it is reasonable to conclude that 
\[
    q^{(n)}_{r}(t) - q^{(n)}_{r}(0) \approx \int_{0}^{t} b_{r}\round{q^{(n)}_{r}(s)}\diff s+ \sigma B_{r}(t),\qquad t\in\R_+,
\]
where $B_r$ is an $r\times r$ matrix with i.i.d. Brownian motions (up to matrix symmetry), in the interior of the state space. In view of this, it is reasonable to expect to that if the process $\round{q^{(n)}_{r,k}}_{k\in\Integer_+}$ spends negligible proportion of time at the boundary, then 
\[
    q^{(n)}_{r}(t) - q^{(n)}_{r}(0) \approx \int_{0}^{t} b_r\round{q^{(n)}_{r}(s)}\diff s + \sigma B_r(t)+ L^{(0)}_{r}(t)- L^{(1)}_{r}(t),\qquad t\in\R_+,
\]
where $\round{q^{(n)}_{r}, L^{(0)}_r, L^{(1)}_r}$ solves the Skorokhod problem on the cube $\Mcal_{r,+}$. That is, each coordinate process of $q^{(n)}_{r}$ satisfies the above SDE with reflection at the boundary $\{0, 1\}$. This heuristic argument can be made precise (see Proposition~\ref{thm:SBM_to_rDiffusion}) and it is one of the main takeaways of this section. Before we state the main theorem, we make a brief digression to the Skorokhod problem and the Skorokhod map which will play a crucial role in Proposition~\ref{thm:SBM_to_rDiffusion} and its proof.

\begin{proposition}\label{thm:SBM_to_rDiffusion}
    Let $\hamil$ satisfy Assumption~\ref{asmp:hamil} and let $(\gamma_n)_{n\in \N}$ satisfy condition~\eqref{eqn:gammagrowth}. 
    Let $D([0, \infty), \Mcal_{r,+})$ be the space of right continuous functions with left limits equipped with the topology of uniform convergence over compact subsets. Let $q^{(n)}_{r}\colon \R_+ \to \Mcal_{r,+}$ be a piecewise interpolation of $\round{q^{(n)}_{r,k}}_{k\in\Integer_+}$ (see equation~\eqref{eq:q^n(t)}). Then,
    $q^{(n)}_{r}$ converges weakly in $D([0, \infty), \Mcal_{r,+})$ to a process $X_{r}$ over compact time intervals, with continuous path that satisfies the SDE 
    \begin{equation}\label{eqn:rDiffusion}
        \diff X_{r}(t) = b_{r}(X_{r}(t)))\diff t+\sigma \diff B_{r}(t)+ \diff L_r^{(0)}(t) - \diff L_r^{(1)}(t), \qquad t\in\R_+,
    \end{equation}
    with initial condition $X_r(0)=q^{(n)}_{r,0}$, where $B_r$ is a symmetric $r\times r$ matrix with whose coordinates are i.i.d. Brownian motions (up to matrix symmetry) and $\round{X_r,L_r^{(0)},L_r^{(1)}}$ solves the Skorokhod problem w.r.t. the finite dimensional cube $\Mcal_{r,+}$ (see Section~\ref{sec:Skorokhod}).
\end{proposition}

\subsubsection{Proof of Proposition~\ref{thm:SBM_to_rDiffusion}}\label{sec:outline}
The proof is long and requires several lemmas. Therefore, we first give an outline of the proof before presenting the details. Fix $r\in \Natural$. For every $k\in\Integer_+$, let $\Fcal^{n}_k$ be the sigma algebra generated by $\set{q^{(n)}_{r,\ell}\given \ell\in\set{0}\cup[k]}$. Let $t_{n,r} = \floor{tr^4/\gamma_n}$ as defined earlier. For $i, j\in [r]$, notice that 
\begin{align*}
    q^{(n)}_{r,(i,j)}(t) - q^{(n)}_{r,(i,j)}(0) &= \sum_{\ell=0}^{t_{n, r}-1}\E{\Delta q^{(n)}_{r,\ell,(i,j)}\given \Fcal^n_\ell}\indicator{(0,1)}{q^{(n)}_{r,\ell,(i,j)}}\\
    &\qquad\qquad + \sum_{\ell=0}^{t_{n, r}-1}\Delta M^{(n)}_{r,\ell,(i,j)}+ L^{(n,0)}_{r,(i,j)}(t)-L^{(n,1)}_{r,(i,j)}(t),
\end{align*}
for every $t\in\R_+$, where $\Delta M^{(n)}_{r,\ell} = \Delta q^{(n)}_{r,\ell}-\E{\Delta q^{(n)}_{r,\ell}\given \Fcal^n_\ell}$ for all $\ell\in\Integer_+$ and 
\begin{align*}
    L^{(n,0)}_{r,(i,j)}(t) &= \sum_{\ell=0}^{t_{n, r}-1}\E{\Delta q^{(n)}_{r,\ell,(i,j)} \given \Fcal^n_\ell}\indicator{\set{0}}{q^{(n)}_{r,\ell,(i,j)}},\\
    L^{(n,1)}_{r,(i,j)}(t) &= \sum_{\ell=0}^{t_{n, r}-1}\E{\Delta q^{(n)}_{r,\ell,(i,j)} \given \Fcal^n_\ell}\indicator{\set{1}}{q^{(n)}_{r,\ell,(i,j)}},
\end{align*}
for $t\in\R_+$, where $L^{(n,0)}_{r,(i,j)}(t)$ is $\frac{1}{n^2}$ times the number of times the process $q^{(n)}_{r,(i,j)}$ visits $\set{0}$ before time $t$ and similarly for $L^{(n,1)}_{r,(i,j)}(t)$. Note that $\round{M^{(n)}_{r,k}\coloneqq \sum_{\ell=0}^{k-1}\Delta M^{(n)}_{r,\ell}}_{k\in\Integer_+}$ is a $\Mcal_r$-valued martingale and we define a piecewise constant interpolation of this martingale process $M^{(n)}_{r}$ defined as $M^{(n)}_{r}(t)=M^{(n)}_{r,t_{n, r}}$ for $t\in\R_+$. Let $\mathrm{Sko}$ be the Skorokhod map (see Section~\ref{sec:Skorokhod}), then for any $t\in\R_+$, and any $(i,j)\in[r]^{(2)}$,
\begin{align}\label{eqn:qnDefinition}
    q^{(n)}_{r,(i,j)}(t) = \Sko{q^{(n)}_{r,(i,j)}(0)+\sum_{\ell=0}^{t_{n, r}-1} \E{\Delta q^{(n)}_{r,\ell,(i,j)}\given \Fcal^n_\ell}\indicator{(0,1)}{q^{(n)}_{r,\ell,(i,j)}}+M^{(n)}_{r,(i,j)}(t)}.
\end{align}

\begin{enumerate}
    \item\label{outlineStep0} Since the Skorokhod map $\mathrm{Sko}$ is a $4$-Lipschitz map~\cite{kruk2007explicit}, to show that $q^{(n)}_r(t)$ converges uniformly to $X_{r}(t)$ as $n\to\infty$, it is sufficient to show that 
    \begin{align*}
        \sum_{\ell=0}^{t_{n, r}-1}& \E{\Delta q^{(n)}_{r,\ell,(i,j)}\given \Fcal^n_\ell}\indicator{(0,1)}{q^{(n)}_{r,\ell,(i,j)}}\\
        &\to \int_{0}^{t} b_{r,(i,j)}(X_r(s))\indicator{(0, 1)}{X_{r,(i,j)}(s)}\diff s,\\
        \text{and}\quad M^{(n)}_r(t)&\to \sigma B_r(t),
    \end{align*}
    uniformly over compact time intervals as $n\to\infty$.
    \item\label{outlineStep1} In Lemma~\ref{lem:Martingale_QV}, we show that the quadratic variation of the martingale $M^{(n)}_r$ in the time interval $[0,t]$ converges to $t\sigma^2$ for every $t\in\R_+$ as $n\to \infty$. The key ingredient is the fact that a simple symmetric reflected random walk spends negligible amount of time at the boundary.
    \item \label{outlineStep2} Using Lemma~\ref{lem:Martingale_QV} and~\cite[Theorem 1.4, Chapter 7]{ethier2009markov} we conclude that the process $M^{(n)}_r$ converges to the process (weakly) $\sigma B_r$ where $B_r$ is an $r\times r$ symmetric matrix with i.i.d. Brownian motions. Using Skorokhod representation theorem both $M^{(n)}_r$ and $B_r$ can be defined on some common probability space $(\Omega, \Fcal, \mathbb{P})$, on which we get almost sure convergence.
    \item \label{oulineStep3} On the probability space $(\Omega, \Fcal, \mathbb{P})$ obtained in Step~\ref{outlineStep2}, we define versions of the processes $X_{r}$ and $q^{(n)}_r$ using~\eqref{eqn:rDiffusion} and~\eqref{eqn:qnDefinition} respectively.
    \item \label{outlineStep4} It remains to show that the first condition in Step~\ref{outlineStep0} holds. To this end, we first show that for every fixed $\epsilon>0$, $\E{\Delta q^{(n)}_{r,\ell,(i,j)}\given \Fcal^n_\ell}\indicator{(\epsilon,1-\epsilon)}{q^{(n)}_{r,\ell,(i,j)}} - b_{r,(i,j)}\round{q^{(n)}_{r, \ell}}\indicator{(\epsilon,1-\epsilon)}{q^{(n)}_{r,\ell,(i,j)}} \to 0$ as $n\to\infty$. This is achieved by a sequence of reductions in Lemma~\ref{lem:Drift_AwayfromBoundary}. 
    \item\label{outlineStep5} Finally, we show that $\sum_{\ell=0}^{t_{n, r}-1}\indicator{(\delta, 1-\delta)^{c}}{q^{(n)}_{r, \ell, (i, j)}}\to 0$ as $n\to \infty$. Since $\frac{1}{\gamma_n}\E{\Delta q^{(n)}_{r, \ell}\given \Fcal_{\ell}}$ is uniformly bounded. We conclude that the first condition in Step~\ref{outlineStep0} holds. This completes the proof.
\end{enumerate}

\begin{proof}[Proof of Proposition~\ref{thm:SBM_to_rDiffusion}]
Let $q^{(n)}_r$ be defined by~\eqref{eqn:qnDefinition}. In the following we will keep $r$ fixed and therefore drop it from the subscript wherever necessary. Throughout, we will keep $t\in\R_+$ fixed. The map $\mathrm{Sko}$ in the discussion will refer to the Skorokhod map defined on the space $D([0, t], \mathbb{R}^{[r]^{(2)}})$, the space of right continuous paths with left limits from $[0, t]$ to $\mathbb{R}^{[r]^{(2)}}$.
Define
\[
    b^{(n)}_{r,(i, j)}\round{q^{(n)}_{r, \ell}} = \frac{1}{\gamma_n r^{-4}}\E{\Delta q^{(n)}_{r,\ell,(i,j)}\given \Fcal^n_\ell}\indicator{(0,1)}{q^{(n)}_{r,\ell,(i,j)}},
\]
and let $b_{r}$ be in Definition in~\ref{def:bFunction}. Recall that both $b^{(n)}_r$ and $b_r$ are uniformly bounded by some constant $C$. Also, recall that $B_r$ is an $r\times r$ symmetric matrix with i.i.d. Brownian coordinates. Define the stochastic processes $Y^{(n)}_r, \widetilde{Y}^{(n)}_{r}, \widetilde{Z}^{(n)}_{r}, Z^{(n)}_{r}\colon \R_+ \to \Rd{[r]^{(2)}}$ as:
\begin{align*}
    \begin{split}
        Y^{(n)}_{r,(i, j)}(t) &= q^{(n)}_{r,(i,j)}(0)+\sum_{\ell=0}^{t_{n, r}-1}\gamma_n r^{-4} b^{(n)}_{r,(i, j)}\round{q^{(n)}_{r, \ell}}+M^{(n)}_{r,(i,j)}(t),\\
         \widetilde{Y}^{(n)}_{r,(i, j)}(t) &= q^{(n)}_{r,(i,j)}(0)+\sum_{\ell=0}^{t_{n, r}-1} \gamma_n r^{-4}b^{(n)}_{r,(i, j)}\round{q^{(n)}_{r, \ell}}+\sigma B_{r,(i, j)}(t),\\
         \widetilde{Z}^{(n)}_{r,(i, j)}(t) &= q^{(n)}_{r,(i,j)}(0)+\int_{0}^{t} b^{(n)}_{r,(i, j)}\round{q^{(n)}_r(s)}\diff s + \sigma B_{r,(i, j)}(t),\\
         Z^{(n)}_{r,(i, j)}(t) &= q^{(n)}_{r,(i,j)}(0)+ \int_{0}^{t} b_{r, (i, j)}\round{q^{(n)}_r(s)}\diff s+ \sigma B_{r,(i, j)}(t),
    \end{split}
    \qquad t\in\R_+, (i,j)\in[r]^{(2)}.
\end{align*}
Notice that $Y^{(n)}_r$ is the ``unconstrained version" of the process $q^{(n)}_r$ in the sense that $\Sko{Y^{(n)}_r}(s) = q^{(n)}_r(s)$ for every $s\in\R_+$. Finally, let $\round{X_r, L^{(0)}_r, L^{(1)}_r}$ be the process that satisfies the Skorokhod SDE
\[
    \diff X_r(t) = b_r(X(t))\diff t + \sigma \diff B_r(t) + \diff L^{(0)}_r(t) - \diff L^{(1)}_r(t),\qquad X_{r,(i, j)}(0) = q^{(n)}_{r, (i, j)}(0).
\]
Denote the corresponding unconstrained process 
\[
    \widetilde{X}_r(t) = q^{(n)}_{r}(0)+ \int_{0}^{t} b_r(X_r(s))\diff s+ \sigma \diff B_r(t).
\]
Note that $\Sko{\widetilde{X}_r}=X_r$. Recall that the goal is to show that $q^{(n)}_r$ converges to the process $X_r$ as $n\to \infty$. Using the fact that the Skorokhod map is Lipschitz (see Section~\ref{sec:Skorokhod}), it is sufficient to show that $Y^{(n)}_r$ converges to $\widetilde{X}_r$. To this end, set 
\begin{align*}
    \Delta^{(n)}(t) \coloneqq \E{\sup_{s\in[0,t]}\normF{Y^{(n)}_r(s)-\widetilde{X}_r(s)}^2},&\qquad 
    \Delta^{(n)}_1(t) \coloneqq \E{\sup_{s\in[0,t]}\normF{Y^{(n)}_r(s)-\widetilde{Y}^{(n)}_r(s)}^2},\\
    \Delta^{(n)}_2(t) \coloneqq \E{\sup_{s\in[0,t]}\normF{\widetilde{Z}^{(n)}_r(s)-\widetilde{Y}^{(n)}_r(s)}^2},&\qquad
    \Delta^{(n)}_3(t) \coloneqq \E{\sup_{s\in[0,t]}\normF{\widetilde{Z}^{(n)}_r(s)-Z^{(n)}_r(s)}^2},\\
    \Delta^{(n)}_4(t) \coloneqq \E{\sup_{s\in[0,t]}\normF{\widetilde{X}^{(n)}_r(s)-Z^{(n)}_r(s)}^2}, &\qquad \text{for all}\; t\in\R_+.
\end{align*}

Since $(a+b+c+d)^2 \leq 4(a^2+b^2+c^2+d^2)$ for all $(a,b,c,d)\in\Rd{4}$,
$\Delta^{(n)}(t)\leq 4\sum_{i=1}^{4}\Delta^{(n)}_i(t) \leq 4\sum_{i=1}^{3}\Delta^{(n)}_i(t)+ 64t\int_{0}^{t}\Delta^{(n)}(s)\diff s$,
where the final inequality is due to the $4$-Lipschtizness of the Skorokhod map.
In particular, for $t\in [0, T]$, we have 
$\Delta^{(n)}(t)\leq 4\sum_{i=1}^{3}\Delta^{(n)}_i(t)+ 64 T\int_{0}^{t}\Delta^{(n)}(s)\diff s$.
Note that $t\mapsto \Delta^{(n)}_i(t)$ is increasing. Therefore, using Gr\"onwall's inequality~\cite{gronwall1919note}, we obtain
\begin{equation}\label{eqn:Gronwall0}
    \Delta^{(n)}(T) \leq 4\round{\sum_{i=1}^{3}\Delta^{(n)}_i(T)}\exp\round{64T^2}.
\end{equation}


It is therefore sufficient to show that $\Delta^{(n)}_i(t)\to 0$ as $n\to \infty$ for $i\in[3]$. This is done in following steps. Using Lemma~\ref{lem:Martingale_QV} below and Theorem~\cite[Theorem 1.4, Chapter 7]{ethier2009markov}, we know that the process $M^{(n)}_r$ converges to $\sigma B_r$ uniformly on compact subsets of time.
In particular, for fixed $t>0$ we have that $\Delta^{(n)}_1(t)\to 0$ as $n\to \infty$. For $\Delta^{(n)}_2$, we notice that the error is actually the error from the Riemann sum approximation. Hence, $\Delta^{(n)}_2(t)\leq C_{r}\gamma_n^2\to 0$ as $n\to \infty$. To see this, first recall that $q^{(n)}(s)$ is piecewise constant on the interval of length $\gamma_nr^{-4}$, that is, $q^{(n)}_r(s)=q^{(n)}_{r, \lfloor s/(\gamma_nr^{-4})\rfloor}$. Now observe that
\begin{align*}
    \abs{\widetilde{Y}^{(n)}_{r, (i, j)}(t) - \widetilde{Z}^{(n)}_{r, (i, j)}(t)}&=\abs{\sum_{\ell=0}^{t_{n, r}-1}\gamma_n r^{-4}b^{(n)}_{(i, j)}\round{q^{(n)}_{r, \ell}}-\int_{0}^{t}b^{(n)}_{(i, j)}\round{q^{(n)}_{r}(s)}}\\
    &=\abs{\round{t-\gamma_nr^{-4}(t_{n, r}-1)} b^{(n)}_{(i, j)}\round{q^{(n)}_{r, t_{n, r}-1}}}\leq C \gamma_n r^{-4},
\end{align*}
where the inequality in the last line follows from the fact that $b^{(n)}$ is uniformly bounded. Squaring both sides and summing over all $(i, j)\in[r]^{(2)}$ we conclude that $\Delta^{(n)}_2(t)\leq C_r\gamma_n^2$. 

We now show that $\Delta^{(n)}_3(t)\to 0$ as $n\to \infty$. To do this, we fix $\epsilon>0$, $\delta>0$ (we assume that $\delta\ll \epsilon$ and $\delta+\epsilon\ll 1$). Define 
\begin{align}
    A_{\epsilon}\coloneqq \set{M\in \Mcal_{r}\given \epsilon \leq  M_{(i, j)}\leq 1-\epsilon\;\;\;\forall\;\; (i, j)\in [r]^{(2)}},\quad B_{\epsilon}\coloneqq \Mcal_r\setminus A_{\epsilon}.\label{eq:A_eps}
\end{align}

Start observing that $\sup_{s\in[0,t]}\normF{\widetilde{Z}^{(n)}_r(s)-Z^{(n)}_r(s)}^2$ is at most
\begin{align}\label{eq:twoterms}
    t\int_{0}^{t}&\normF{b^{(n)}_r\round{q^{(n)}_r(s)}-b_r\round{q^{(n)}_r(s)}}^2\diff s\nonumber\\
    &\leq t\int_{0}^{t}\normF{b^{(n)}_r\round{q^{(n)}_r(s)}-b_r\round{q^{(n)}_r(s)}}^2\indicator{A_{\epsilon}}{q^{(n)}_r(s)}\diff s\nonumber\\
    &\qquad+ t\int_{0}^{t}\normF{b^{(n)}_r\round{q^{(n)}_r(s)}-b_r\round{q^{(n)}_r(s)}}^2\indicator{B_{\epsilon}}{q^{(n)}_r(s)}\diff s.
\end{align}
From Lemma~\ref{lem:Drift_AwayfromBoundary} below $b^{(n)}_r\round{q^{(n)}_r(s)}$ and $b_{r}\round{q^{(n)}_r(s)}$ are close in $\normF{}^2$ by $\frac{Cr^2}{n^4}+4r^2\beta_{n,r}^2e_n^3\max\set{\abs{\lambda},L}$ with probability at least $1-p_{n, \epsilon}$, when $q^{(n)}_r(s)\in A_{\epsilon}$ where $p_{n, \epsilon} = \frac{4r^2}{n^4}+2r^2\mathrm{erf}\round{\frac{r^2\epsilon}{4\sigma\sqrt{2\gamma_n}}}$ and $e_n=O(\gamma_n^2\log n)$. On the other hand, we notice that $b^{(n)}_r$ and $b_r$ are both uniformly bounded by in $\norm{\infty}{}$. Using these two facts we conclude that 
\begin{align}\label{eqn:Delta3_awayfrombd}
    &\E{\int_{0}^{t}\normF{b^{(n)}_r\round{q^{(n)}_r(s)}-b_r\round{q^{(n)}_r(s)}}^2\indicator{A_{\epsilon}}{q^{(n)}_r(s)}\diff s}\nonumber\\
    &\qquad\leq C\frac{t}{\gamma_n}p_{n, \epsilon}+\frac{Cr^2}{n^4}+4r^2\beta_{n,r}^2e_n^3\max\set{\abs{\lambda},L}.
\end{align}
Since $b^{(n)}_r$ and $b_r$ are uniformly bounded, the second term in~\eqref{eq:twoterms} is bounded as
\begin{align*}
    \int_{0}^{t}\normF{b^{(n)}_r\round{q^{(n)}_r(s)}-b_r\round{q^{(n)}_r(s)}}^2\indicator{B_{\epsilon}}{q^{(n)}_r(s)}\diff s\leq CD^{(n)}(t),
\end{align*}
for some constant $C>0$, where $D^{(n)}(t)\coloneqq \int_{0}^{t}\indicator{B_{\epsilon}}{q^{(n)}_r(s)}\diff s $.


We now approximate the indicator function, $\indicator{B_{\epsilon}}{\cdot}$ by a smooth function $\psi_{\epsilon, \delta}\in C^{\infty}([0, 1])$. That is, Let $\psi_{\epsilon, \delta}$ be a smooth function such that $\psi_{\epsilon, \delta}\equiv 1$ on the set $I_{\epsilon}\coloneqq [0, \epsilon)\cup (1-\epsilon, 1]$ and $0\leq \psi_{\epsilon, \delta}\leq 1$ and $\mathrm{supp}(\psi) \subset [0, \epsilon+\delta)\cup (1-\epsilon-\delta, 1]$. Recall that by our assumption $\epsilon+\delta\ll 1$.
Then, $D^{(n)}(t) \leq  \int_{0}^{t} \psi_{\epsilon, \delta}\round{q^{(n)}_r(s)}\diff s$.

Recall that $q^{(n)}_r(s)=\Sko{Y^{(n)}}(s)$ for every $s\in\R_+$. Also recall that both $\psi_{\epsilon, \delta}$ and $\mathrm{Sko}$ are Lipschitz functions. Therefore, the composition $\psi_{\epsilon, \delta}\circ \mathrm{Sko}$ is is also a Lipschitz function, say, with Lipschitz constant $L_{\epsilon, \delta}$. Define $\Psi_{\epsilon, \delta}(s)\coloneqq \psi_{\epsilon, \delta}\round{\Sko{Y^{(n)}_r}(s)}$ and $\widetilde{\Psi}_{\epsilon, \delta}(s)\coloneqq \psi_{\epsilon, \delta}\round{\Sko{\widetilde{Z}^{(n)}_r}(s)}$. Now observe that 
\begin{align*}
    \Psi_{\epsilon, \delta}(s)&\leq \widetilde{\Psi}_{\epsilon, \delta}(s) + L_{\epsilon, \delta} \normF{\widetilde{Z}^{(n)}_r(s)-Y^{(n)}_r(s)}^2\\
    &\leq \widetilde{\Psi}_{\epsilon, \delta}(s)+2L_{\epsilon, \delta}\round{\normF{\widetilde{Z}^{(n)}_r(s)-\widetilde{Y}^{(n)}_r(s)}^2+\normF{Y^{(n)}_r(s)-\widetilde{Y}^{(n)}_r(s)}^2}.
\end{align*}
Therefore, we obtain 
\begin{equation}\label{eqn:ent}
    \E{D^{(n)}(t)}\leq \E{\int_{0}^{t}\widetilde{\Psi}_{\epsilon, \delta}(s)\diff s}+ 2L_{\epsilon, \delta}t\round{\Delta^{n}_1(t)+\Delta^{(n)}_2(t)}.  
\end{equation}

\sloppy Note that $\E{\int_{0}^{t}\widetilde{\Psi}_{\epsilon, \delta}(s)\diff s}=\E{F_{\epsilon, \delta}\round{\widetilde{Z}^{(n)}}}$ for some bounded continuous function $F_{\epsilon, \delta}\colon C([0, t], \Mcal_{r,+})\to \R$. Also, recall that $\widetilde{Z}^{(n)}$ satisfies the SDE $\widetilde{Z}^{(n)}_r(t) = q_{r}(0)+ \int_{0}^{t} f(s)\diff s+ \sigma B_r(t)$, where $f(s)= b^{(n)}_r\round{q^{(n)}_r(s)}$ is a bounded function. Set
\[
    \mathcal{E}= \exp\round{\frac{1}{\sigma}\int_{0}^{t}b^{(n)}_r\round{q^{(n)}_r(s)}\diff B_r(s)-\frac{1}{2\sigma^2}\int_{0}^{t}b^{(n)}_r\round{q^{(n)}_r(s)}\diff s}.
\]
Using Girsanov's theorem and the Cauchy--Schwarz inequality we obtain
\begin{align*}
    \E{F_{\epsilon, \delta}\round{\widetilde{Z}^{(n)}_r}}^2 = \E{F_{\epsilon, \delta}(B)\mathcal{E}}^2\leq \E{F_{\epsilon, \delta}^2(B)}\E{\mathcal{E}^2}.
\end{align*}
Finally using the fact that $b^{(n)}_r$ is uniformly bounded, we obtain that $\E{\Ecal^2}\leq C_{r, t, \sigma}$. On the other hand, we notice that by definition 
\begin{align}\label{eqn:occupationNearbd}
    \E{F_{\epsilon, \delta}^2(B)} &= \E{\round{\int_{0}^{t}\psi_{\epsilon, \delta}(\mathrm{RBM}(s))\diff s}^2}\leq t\int_{0}^{t} \Prob{\mathrm{RBM}(s)\in I_{\epsilon}}\diff s \eqqcolon C(\epsilon, \delta, r,t)^2,
\end{align}
where the equality follows from the Cauchy-Schwarz and the fact that $\psi_{\epsilon, \delta}^2\leq \indicator{B_{\epsilon+\delta}}{\cdot}$.
Combining equations~\eqref{eqn:Delta3_awayfrombd},~\eqref{eqn:ent} and~\eqref{eqn:occupationNearbd} we obtain that 
$\Delta^{(n)}_3(t)\leq \frac{t}{\gamma_n}p_n + C(\epsilon,\delta, r, t)+2L_{\epsilon, \delta}t\round{\Delta^{n}_1(t)+\Delta^{(n)}_2(t)}$.


Putting this back in equation~\eqref{eqn:Gronwall0} we conclude that 
\[
    \Delta^{(n)}(t)\leq \round{\frac{t}{\gamma_n}p_n + C(\epsilon,\delta, r, t)+(2L_{\epsilon, \delta}t+C)\round{\Delta^{n}_1(t)+\Delta^{(n)}_2(t)}}\eu^{Ct^2}.
\]
Recall that $\frac{p_n}{\gamma_n}\to 0$ and $\Delta^{(n)}_1(t)\to 0$ and $\Delta^{(n)}_2(t)\leq \gamma_n^2\to 0$ as $n\to \infty$. Therefore, we conclude that 
$\limsup_{n\to \infty}\Delta^{(n)}(t)\leq C(\epsilon, \delta, r, t)\eu^{Ct^2}$.
Since $C(\epsilon, \delta, r, t)\to 0$ as $\epsilon, \delta\to 0$, this concludes the proof.
\end{proof}

Let $b_0$ be as in Definition~\ref{def:bFunction}. Recall from Remark~\ref{rem:DriftConvergence} that $\norm{\infty}{b_0+\beta D\Hcal}\to 0$ as $r\to \infty$. 
As an immediate consequence of Theorem~\ref{thm:mv} (see Remark~\ref{rem:Perturbation}) we obtain that the process $X_{r}$ converges to the McKean-Vlasov SDE defined in~\eqref{eq:McKeanV1} with drift given by $-\beta D\Hcal(w)$ as $r\to \infty$. That is, given a pair of $\mathrm{Uni}([0,1])$ i.i.d. random variables $\left(U,V\right)$ and a standard Brownian motion $B$ on some probability space $(\Omega, \Gcal, \mathbb{P})$, consider the following SDE conditioned on $\{(U,V)=(u,v)\}$,
\begin{align}\label{eqn:McKean_Metropolis}
    \begin{split}
        \diff X(t)&= -\beta D\Hcal(\E{W^{\sigma}(t)})(u,v)\diff t + \sigma \diff B(t) + \diff L^{(0)}(t) - \diff L^{(1)}(t),\\
        W^{\sigma}(t)(x,y)&= \mathrm{Law}(X(t)\,\vert\, (U,V)=(x,y)),\qquad (x,y)\in [0,1]^{(2)},
    \end{split}
\end{align}
for $t\in\R_+$, where $\round{X,L^{(0)},L^{(1)}}$ solves the Skorokhod problem with respect to $[0,1]$. 

\begin{proposition}\label{prop:rDiffusion_to_McKeanVlasov}
    \sloppy Let $X_{r}$ be a solution of~\eqref{eqn:rDiffusion} with initial condition $X_{r}(0)\in \Mcal_{r,+}$. If $\lim_{r\to \infty}\norm{\blacksquare}{X_{r}(0)-W_0}=0$, then $X_{r}$ converges in MVG sense, in probability, uniformly over compact time intervals, to a deterministic curve $W^{\sigma}$ in the space of MVGs as $r\to\infty$. Moreover, $W^{\sigma}$ is described by~\eqref{eqn:McKean_Metropolis} with initial condition $W_0$. 
\end{proposition}

\begin{remark}\label{rem:grapon_curve_MKV}
    \sloppy Consider the curve $w^{\sigma}$ in the space of graphons defined as $w^{\sigma}(t)\coloneqq \E{W(t)}$
    for all $t\in\R_+$.  It follows from Remark~\ref{rem:MVG_Graphon_MKV} that the random curves $\round{X_{r}}_{r\in\Natural}$ converge in cut-metric, uniformly on compact intervals of time, to $w^{\sigma}$ in probability. And, $w^{\sigma}$ can be recovered as a solution of a MKV SDE satisfying on $\set{(U, V)=(u, v)}$
    \begin{align}\label{eqn:McKean_Graphon}
        \begin{split}
            \diff X(t)&= -\beta D\Hcal(w^{\sigma})(u,v)\diff t + \sigma \diff B(t) + \diff L^{(0)}(t) - \diff L^{(1)}(t),\\
            w^{\sigma}(t)(x,y)&= \E{X(t)\given (U,V)=(x,y)},\qquad (x,y)\in [0,1]^{(2)},\quad t\in\R_+,
        \end{split}
    \end{align}
    where $\round{X,L^{(0)},L^{(1)}}$ solves the Skorokhod problem with respect to $[0,1]$.
\end{remark}
When $\sigma=0$, the graphon McKean-Vlasov~\eqref{eqn:McKean_Graphon} reduces to a deterministic evolution $w$ of kernels given by 
\begin{equation}\label{eqn:Sigma0_evolution}
    w(t)(x, y)=w(0)(x, y)-\beta \int_0^t D\Hcal(w(s))(x, y)\indicator{G_{w(s)}}{}\diff s,
\end{equation}
for $(x,y)\in[0,1]^{(2)}$ and $t\in\R_+$, where $G_{u}$ is defined analogous to~\eqref{eqn:G_function}.
This can be seen by defining $L^{(0)}$ and $L^{(1)}$ as
\begin{align*}
    L^{(1)}(t) &\coloneqq +\int_0^t b(w(s))(u,v)\indicator{}{w(s)(u,v)=1,b(w(s))>0}\diff s,\\
    L^{(0)}(t) &\coloneqq -\int_0^t b(w(s))(u,v)\indicator{}{w(s)(u,v)=0,b(w(s))<0}\diff s,
\end{align*}
on $\set{(U,V)=(u,v)}$, and observing that the process $(X,L^{(0)},L^{(1)})$ solves the Skorokhod problem w.r.t. $[0,1]^{(2)}$ (see Section~\ref{sec:Skorokhod}). 
It is clear that that $w$ is a constant factor reparametrization of gradient flow of $\hamil$ on the space of graphons. We now show that this is indeed the case, that is, as $\sigma\to 0$ the curve $w^{\sigma}$ converges to $w$ on the space of graphons under the cut metric, uniformly over compact intervals of time. To this end, let $(\Omega, \Fcal, \mathbb{P})$ be a probability space equipped with a family of i.i.d. uniform random variables $\set{U_i}_{i\in \N}$ and a collection of independent linear BM $\set{B_{(i, j)}}_{(i, j)\in \N^{(2)}}$. We can therefore define an IEA $X^{\sigma}$ on this probability space such that
\begin{align*}
    \begin{split}
        \diff X^{\sigma}_{(i, j)}(t)&= -\beta D\Hcal(\E{W^{\sigma}(t)})(U_i,U_j)\diff t + \sigma \diff B_{(i, j)}(t) + \diff L^{(0)}_{(i, j)}(t) - \diff L^{(1)}_{(i, j)}(t),\\
        W^{\sigma}(t)(x,y)&= \mathrm{Law}(X^{\sigma}_{(i, j)}(t)\,\vert\, (U_i,U_j)=(x,y)),\qquad (x,y)\in [0,1]^{(2)},
    \end{split}
\end{align*}
for $t\in\R_+$, where $\round{X^{\sigma},L^{(0)},L^{(1)}}$ solves the Skorokhod problem with respect to $[0,1]$.

Let $w$ be as defined in~\eqref{eqn:Sigma0_evolution}. Recall that $w(t)$ can be naturally identified with an MVG $W(t)$ defined as $W(t)(x, y)\coloneqq \delta_{w(t)(x, y)}$ for $(x,y)\in[0,1]^{(2)}$. On the probability space $(\Omega, \Fcal, \mathbb{P})$ we define another IEA given by $X_{(i, j)}(t)=w(t)(U_i, U_j)$ for $(i,j)\in\Natural^{(2)}$. Notice that the IEA $X$ satisfies the McKean-Vlasov SDE given by 
\begin{align}\label{eqn:McKean_IEA_sigma}
    \begin{split}
        \diff X_{(i, j)}(t)&= -\beta D\Hcal(\E{W(t)})(U_i,U_j)\diff t + \diff L^{(0)}_{(i, j)}(t) - \diff L^{(1)}_{i, j}(t),\\
        W(t)(x,y)&= \mathrm{Law}(X_{(i, j)}(t)\,\vert\, (U_i,U_j)=(x,y)),\qquad (x,y)\in [0,1]^{(2)},
    \end{split}
    \qquad t\in\R_+,
\end{align}
where $\round{X,L^{(0)},L^{(1)}}$ solves the Skorokhod problem with respect to $[0,1]$. Note that given $\set{U_i}_{i\in \N}$ the IEA $X$ is deterministic. In particular, $W(t)(x, y)=\delta_{w(t)(x, y)}$. 


\begin{proposition}\label{lem:Sigma0convergence}
    Let $w_0\in \Wcal_{[0,1]}$ be a kernel. Let $w^{\sigma}$ and $w$ be defined in equation~\eqref{eqn:McKean_Graphon} and~\eqref{eqn:Sigma0_evolution}. Then, for every finite $t>0$,  $\sup_{s\in[0,t]}\norm{\cut}{w^{\sigma}(s)-w(s)}\leq 2C\sigma^2t\exp(Ct^2)$ for some universal constant $C>0$.
\end{proposition}
\begin{proof}[Proof of Proposition~\ref{lem:Sigma0convergence}]
We first prove a slightly stronger result, that is, we show that $W^{\sigma}$ converges to $W$ in the MVG sense. The desired result therefore follows immediately. The proof closely resembles the proof of Theorem~\ref{thm:mv}. Let $(\Omega, \Fcal, \mathbb{P})$ be as above and $X^{\sigma}, X, W^{\sigma}, W$ be as above with the initial condition $W^{\sigma}(0)(x, y)=W(0)(x, y)=\delta_{w_0(x, y)}$. Using the Lipschitzness of Skorokhod map as in the proof of Theorem~\ref{thm:mv}, we observe that for any $(i, j)$ we have
\begin{align*}
    \abs{X^{\sigma}_{(i, j)}(t)-X_{(i, j)}(t)}^2 &\leq Ct\int_{0}^{t}\abs{b(w^{\sigma}(s))(U_i, U_j)-b(w(s))(U_i, U_j)}^2+ C\sigma^2 \abs{B_{i, j}(t)}^2.
\end{align*}
Summing over $(i, j)\in [k]^{2}$ and diving by $\frac{1}{k^2}$ we obtain that for each $k\in \N$ we have 
\begin{align*}
    &\norm{2}{\Kcal\round{X^\sigma[k](t)}-\Kcal\round{X[k](t)}}^2 \leq I_k(t) + J_k(t),\\
    \begin{split}
        \text{where}\quad I_{k}(t)&= Ct\int_{0}^{t}\frac{1}{k^2}\sum_{(i, j)\in [k]^2}\abs{b(w^{\sigma}(s))(U_i, U_j)-b(w(s))(U_i, U_j)}^2\diff s,\\
        \text{and}\quad J_k(t)&\coloneqq C\sigma^2\frac{1}{k^2}\sum_{(i, j)\in [k]^{2}}\abs{B_{(i, j)}(t)}^2.
    \end{split}
\end{align*}
\sloppy By Doob's maximal inequality~\cite[page 14, Theorem 3.8.iv]{KS91} and Markov's inequality we get $\Prob{\sup_{s\in[0,t]}J_{k}(s)\geq 2C\sigma^2t}\leq \frac{C}{4k^2}$. 
Using our assumption on $b$, we conclude that (compare with~\eqref{eq:triangle_lipschitz}) 
\begin{align}
    \sup_{s\in \interval{0,t}}\norm{\blacksquare}{\Kcal\round{X^\sigma[k](s)} - \Kcal\round{X[k](s)}}^2&\le C\beta^2\kappa^2_\blacksquare t \int_0^t \norm{\blacksquare}{W(s) - W^{\sigma}(s)}^2\diff s+ 2C\sigma^2t,
\end{align}
with probability at least $1-\frac{C}{4k^2}$.
Note that compared to equation~\eqref{eq:triangle_lipschitz} in the proof of Theorem~\ref{thm:mv}, the above inequality is much simpler. The reason being that the drift function $b$ depends only on MVG and not on $X_{(i, j)}(t)$. Secondly, the our initial condition ensures that $X^{\sigma}_{(i, j)}(0)=X_{(i, j)}(0)$. At this step, we use the same argument as in the proof of Theorem~\ref{thm:mv} to replace $\norm{\blacksquare}{\Kcal\round{X^\sigma [k](s)}-\Kcal\round{X[k](s)}}^2$ with $\norm{\blacksquare}{W(s)-W^{\sigma}(s)}^2$ up to a small error $C_k=64k^{-1/4}$ with probability at least $1-\eu^{-c_k}$ where $c_{k}=\sqrt{k}/20$. Combining all this and using  Gr\"onwall's inequality~\cite{gronwall1919note} as in the proof of Theorem~\ref{thm:mv} we conclude that
    \begin{align*}
        \begin{split}
            \sup_{s\in \interval{0,t}} &D_2^2\round{\Kcal\round{X^\sigma[k](s)},\Kcal\round{X[k](s)}} + \sup_{s\in \interval{0,t}}\norm{\blacksquare}{ W^{\sigma}(s) - W(s) }^2 \\
            &\le (C_k+ 2C\sigma^2t)\eu^{C\beta^2\kappa^2_\blacksquare t^2}, \; \text{with probability at least $1-\eu^{-c_k}-\frac{C}{4k^2}$.}
        \end{split}
    \end{align*}
     Letting $k\to \infty$, we conclude that $\sup_{s\in \interval{0,t}}\norm{\blacksquare}{ W^{\sigma}(s) - W(s) }^2\leq 2C\sigma^2t\eu^{C\beta^2\kappa^2_\blacksquare t^2}$. The desired claim now follows from the fact that $W\to \E{W}$ is a contraction.  
\end{proof}


\begin{proposition}[Non-asymptotic high probability error bound]\label{prop:rate}
    Let $X_n^\sigma$ be a solution to equation~\eqref{eq:SDE_n} for $\Sigma_n \equiv \sigma$. Let the assumptions of Theorem~\ref{thm:mv} and Theorem~\ref{lem:Sigma0convergence} hold. If the initial condition is i.i.d., then
    \begin{align}
        \sup_{s\in \interval{0,t}}\norm{\blacksquare}{ \Kcal(X^\sigma_n(s)) - W(s) }^2 &\leq C_tn^{-1/56}\log^{3/2}n + (64n^{-1/14} + 2C\sigma^2 t)\eu^{C\beta^2\kappa^2_\blacksquare t^2},
    \end{align}
    with probability at least $1-5n^{-3/7} - t n^{-\frac{2}{7\kappa^2 t}}$. Here $W(s)(x,y) = \delta_{w(s)(x,y)}$ for a.e. $(x,y)\in[0,1]^{(2)}$, and $s\in\R_+$ following equation~\eqref{eqn:McKean_Graphon}; and $\kappa = 32\sqrt{6}\round{L^2+2\kappa_\blacksquare^2}^{1/2}$.
\end{proposition}
\begin{proof}
    To get a non-asymptotic error rate, we need to control on $A_k$ and $B_k(n)$ in equation~\eqref{eq:semifinal_rate}. Observe that $B_k(n)$ depends on the initial condition and in general it can be arbitrarily slow. However, assuming that the initial condition is i.i.d., one can use Chebyshev's inequality to obtain $\Prob{B_k(n)\geq 66k^{-1/4}}\leq k^{-3/2}$. 

    On the other hand, combining the arguments in~\cite[Proposition 4.5]{HOPST22} and~\cite[Proposition 8.12]{lovasz2012large}, it can be shown that there exists a constant $M_{t}$ (depending only on $t$) such that for any $\delta>0$ we have $\Prob{A_k\geq M_t(\delta \log(1/\delta))^{1/4}}\leq k^{-2}+ t\delta^{-1}\eu^{\frac{128}{\delta\log(1/\delta)}} \eu^{-k\delta \log(1/\delta)/2}$. 
    

    To obtain above bound for $A_k$, we argue as follows. For each fixed $\psi\in \Lcal$, moment computation yields $t(C_4, \Gamma(\psi, \widetilde{A}_k(s)))$ is sub-gaussian with norm at most $\frac{1}{\sqrt{k}}$. In particular, for any $\delta>0$ we have $\Prob{t(C_4,\Gamma(\psi, \widetilde{A}_k(s))\geq \sqrt{\delta\log(1/\delta)}}\leq \eu^{-\frac{k\delta\log(1/\delta)}{2}}$. Note that the right side is independent of $\psi$. For any subset $F\subseteq \Lcal$ define $\Delta_{k, F}(s)\coloneqq \sup_{\psi\in F} \abs{t(C_4, \Gamma(\psi, \widetilde{A}_k(s)))}$. Fix $\epsilon>0$ and note that by Lemma~\ref{lem:finitesufficiency} there exists a finite set $F\subseteq \Lcal$ such that $\abs{F}\leq \eu^{\frac{32}{\epsilon^2}}$ and $\abs{\Delta_{k,\Lcal}(s)-\Delta_{k, F}(s)}\leq \epsilon$. Taking $\epsilon = \inv{2}\sqrt{\delta\log(1/\delta)}$, we get $\Prob{\Delta_{k, \Lcal}(s)\geq \sqrt{\delta \log(1/\delta)}}\leq \Prob{\Delta_{k, F}(s)\geq 2^{-1}\sqrt{\delta \log(1/\delta)}} \leq \eu^{\frac{128}{\delta\log(1/\delta)}} \eu^{-k\delta \log(1/\delta)/2}$. 

    Repeating the proof of~\cite[Proposition 4.5]{HOPST22}, we obtain that $(\Delta_{k, \Lcal})_{k\in\Natural}$ is equicontinuous with high probability. That is, for any fixed $\delta>0$ we have $\Prob{\sup_{\abs{s_1-s_2}\leq \delta, s_1, s_2\in [0, t]}\abs{\Delta_{k, \Lcal}(s_1)-\Delta_{k, \Lcal}(s_2)}\geq M_t \sqrt{\delta\log(1/\delta)}}\leq \frac{1}{k^2}$. It now follows from a $\delta$-net argument that $\Prob{\sup_{s\in [0, t]}\Delta_{k, \Lcal}(s)\geq M_t(\delta \log(1/\delta))^{1/2}}\leq k^{-2}+ t\delta^{-1}\eu^{\frac{128}{\delta\log(1/\delta)}} \eu^{-k\delta \log(1/\delta)/2}$. Following~\cite[Proposition 8.12]{lovasz2012large}, we have $\norm{\blacksquare}{\widetilde{A}_k(s)}^4\leq \Delta_{k, \Lcal}(s)$. This yields the desired conclusion.
    In particular, choosing $\delta=64\sqrt{k^{-1}\log k}$ and $\lambda_k=\log(k)/\round{16\cdot 384t(L^2+2\kappa_{\blacksquare}^2)}$, we have the left hand side of~\eqref{eq:semifinal_rate} bounded by $M_tk^{-1/16}\log^{3/2}k$ with probability at least $1-\frac{k^2}{n}-4k^{-\inv{\kappa^2t}} - 2t\eu^{-\sqrt{k}/20}- 2k^{-3/2}$, where $\kappa=32\sqrt{6}\round{L^2+2\kappa_{\blacksquare}^2}^{1/2}$.
    
    Since $t$ is fixed, we can choose $k$ to be a suitable function of $n$, say $k=n^{2/7}$. 
    The proof is now complete with the help of Proposition~\ref{lem:Sigma0convergence} and a triangle inequality.
\end{proof}
\begin{proposition}[Convergence McKean-Vlasov~\eqref{eqn:Sigma0_evolution} to equilibrium when $\sigma = 0$]\label{prop:convergence_rate_MH}
Let $\hamil$ be $\delta_\cut$-lower semicontinuous and satisfy Assumption~\ref{asmp:hamil} with $\lambda \geq 0$ and $L\in\lbrack\lambda,\infty)\cup\{\infty\} $. Let $w$ be the graphon valued curve as defined in~\eqref{eqn:Sigma0_evolution}. 
Let $w_{*}\in \Graphons_{[0,1]}$ be a minimizer of $\hamil$, then
\begin{align}
    \hamil(w(t))-\hamil(w_{*}) \leq \frac{\delta_2^2\round{w(0),w_{*}}}{2\beta t}, \qquad t\in\R_+.
\end{align}
Moreover, if $\lambda>0$ and $L<\infty$ and $w_{*}\in \Graphons_{[0,1]}$ is the unique minimizer of the strongly convex function $\hamil$, then for $t\in\R_+$,
\begin{align}
\begin{split}
    \delta_2\round{w(t),w_*} &\leq \eu^{-\beta\lambda t}\delta_2\round{w(0),w_*},\quad
    \hamil(w(t)) - \hamil(w_*) \leq \frac{L}{2}\eu^{-2\beta\lambda t}\delta_2^2\round{w(0),w_*}.
\end{split}
\end{align}
\end{proposition}
\begin{proof}
    Notice that $\beta D\hamil$ is the Fr\'echet-like derivative evaluation map of $\beta\hamil$. The proof immediately follows from~\cite[Remark 4.16]{Oh2023} following~\cite[Remark 4.0.5, part (d)]{ambrosio2005gradient},~\cite[Corollaray 4.0.6]{ambrosio2005gradient} and Assumption~\ref{asmp:hamil}.
\end{proof}


The notion of (geodesic) convexity of functions on graphons can be found in~\cite[Definition 2.15]{Oh2023}. When $\hamil$ is a linear combinations of homomorphism densities, it is only semiconvex. However, one may regularize $\hamil$ by adding a large enough multiple of scalar entropy to make it strictly convex~\cite[Section 5.1.1, Section 5.1.3]{Oh2023} and satisfy the conditions for exponential convergence in Proposition~\ref{prop:convergence_rate_MH}.

\section{Remaining Proofs}\label{sec:Proofs}
\subsection{Topology and metric on MVG}\label{sec:Proof_metric_mvg}

\begin{proof}[Proof of Proposition~\ref{prop:d_D_are_metric}]
    We first show that $\W_\blacksquare$ and $\Delta_\blacksquare$ are equal. Let $U, V$ be measure valued graphons and let $\varphi$ be some bounded Lipschitz function. Using the definition of the cut norm and using Fubini's theorem,
    \begin{align*}
          \norm{\cut}{\Gamma(\psi, U)-\Gamma(\psi, V)}&=\sup_{S, T}\abs{\int_{S\times T}(\Gamma(\psi, U)-\Gamma(\psi, V))(x, y)\,\diff x\diff y}\\
        &= \sup_{S, T}\abs{\int \psi(\zeta) (\Lcal_{S\times T}U)(\diff \zeta)-\int \psi(\zeta)(\Lcal_{S\times T}V)(\diff \zeta)},
    \end{align*}
    where $(\Lcal_{S\times T}W)\coloneqq \int_{S\times T}W(x, y)\diff x\diff y$ for any $W\in\Wfrak$, and Borel measurable sets $S,T\subseteq [0,1]$. Taking supremum over all Lipschitz functions $\psi$ on $[-1,1]$ with $\norm{\rm Lip}{\psi}\leq 1$ on both side and interchanging the order of two suprema in the right, we obtain 
    $\sup_{\psi} \norm{\cut}{\Gamma(\psi, U)-\Gamma(\psi, V)} = \sup_{S, T}\mathbb{W}_1(\Lcal_{S\times T}U, \Lcal_{S\times T}V)$.
    Since $U, V$ were arbitrary, the desired result now follows by replacing $V$ with $V^{\varphi}$ and taking infimum over all $\varphi\in \Tcal$. It follows that $\W_\blacksquare$ and $\Delta_\blacksquare$ are equal. The fact that $\W_\blacksquare$ is a metric on $\mvGraphons$ follows by mimicking the standard proof of cut-metric being a metric on graphons (see~\cite{lovasz2006limits}). We briefly outline the idea of the proof. Note that $\W_\blacksquare$ and $\Delta_\blacksquare$ do not satisfy positivity on $\Wfrak$, that is, $\W_\blacksquare(U, V)$ can be $0$ even though $U\neq V$. It suffices to show that $\W_{\blacksquare}(U, V)=0$ if and only if $U\cong V$ in $\mvGraphons$, that is, $t\ped{d}(F, U)=t\ped{d}(F, V)$ for every decorated graph $F$. This follows from Theorem~\ref{thm:equivalence_of_cut}. 
\end{proof}

\begin{lemma}\label{lem:Hom_density_is_point_separating}
    Let $\Dcal\subseteq \Ccal$ be a subset that is closed under finite products. Suppose that the linear span $A(\Dcal)$ generated by $\Dcal$ is dense in $\Ccal$ in the sup norm.  Let $\round{W_n}_{n\in\Natural}\in \Wfrak$ and let $W\in\Wfrak$.
    Then, the following are equivalent.
    \begin{enumerate}
        \item  $\lim_{n\to \infty}t\ped{d}(F, W_n)=t\ped{d}(F, W)$ for every decorated graph $F$.
        \item $\lim_{n\to \infty}t\ped{d}(F, W_n)=t\ped{d}(F, W)$ for every $\Dcal$-decorated graph $F$.
    \end{enumerate}
\end{lemma}
\begin{proof}
    Obviously $(1)$ implies $(2)$. To see the converse, first note that if $\lim_{n\to \infty}t\ped{d}(F, W_n)= t(F, W)$ for every $\Dcal$-decorated graph $F$ then $\lim_{n\to \infty}t\ped{d}(F, W_n)= t\ped{d}(F, W)$ for every $A(\Dcal)$-decorated graph $F$. Now let $(F, f)$ be a $\Ccal$-decorated graph and let $\epsilon>0$ be fixed. Then, there exists an $A(\Dcal)$-decoration $(F, g)$ of the skeleton $F$ such that $\max_{i,j\in E(F)}\norm{\infty}{f_{i,j}-g_{i,j}}\leq \epsilon$. Let $C>0$ be a finite constant such that $\max_{i,j}\norm{\infty}{f_{i,j}}\leq C$. It follows that $\max_{i,j}\norm{\infty}{g_{i,j}}\leq C'=(1+C)$. Using the Counting Lemma for decorated graphs~\cite[Lemma 10.26]{lovasz2006limits}, for any $U\in \mvGraphons$ we have $\abs{t\ped{d}((F, g), U) - t\ped{d}((F, f), U)}\leq 4|E(F)|C'\epsilon$. Thus 
    \begin{align*}
        &\abs{t\ped{d}((F, f), W_n) - t\ped{d}((F, f), W)}\\
        &\leq \abs{t\ped{d}((F, f), W_n) - t\ped{d}((F, g), W_n)}+\abs{t\ped{d}((F, g), W) - t\ped{d}((F, f), W)}\\
        &\qquad+ \abs{t\ped{d}((F, g), W_n) - t\ped{d}((F, g), W)}\\
        &\leq 2C'\epsilon + \abs{t\ped{d}((F, g), W_n) - t\ped{d}((F, g), W)}.
    \end{align*}
    Since $g$ is an $A(\Dcal)$-decoration and $\abs{t\ped{d}((F, g), W_n) - t\ped{d}((F, g), W)}\to 0$ as $n\to \infty$. It follows that $\lim_{n\to \infty}\abs{t\ped{d}((F, f), W_n) - t\ped{d}((F, f), W)}\leq 2C'\epsilon$. Taking $\epsilon\to 0$ completes the proof. 
\end{proof}

\begin{lemma}\label{lem:OurConvg_implies_LovaszConvg}
    \sloppy Let $\round{W_n}_{n\in\Natural}\in \Wfrak$ and let $W\in\Wfrak$. Then, $\round{W_n}_{n\in\Natural}\to W$ in $\mvGraphons$ if and only if $\round{t(F, W_n)}_{n\in\Natural}\to t(F, W)$ for every finite simple graph $F$.
\end{lemma}
\begin{proof}
    \sloppy Let $(F, f)$ be a decorated graph. Define $\varphi_F\coloneqq \tensor_{\set{i,j}\in E(F)}f(\set{i,j})$. Hence, $t\ped{d}((F, f), W) = \inner{\varphi_F, t(F, W)}$,
    where $t(F, \cdot)$ is as in Definition~\ref{eqn:Def_homdensity_W}. It follows that if $t(F, W_n)\to t(F, W)$ weakly for a skeleton $F$, then $t\ped{d}(F, W_n)\to t\ped{d}(F, W)$ for any decoration $(F, f)$. Conversely, the linear span of $\set{\varphi_F\colon f \text{ is a decoration of }F}$ is dense in $C(I_{F})$ by the Stone-Weierstrass theorem. Thus, $t\ped{d}((F, f), W_n)\to t\ped{d}((F, f), W)$ implies that $\inner{\varphi, t(F, W_n)}\to \inner{\varphi, t(F, W)}$ for any $\varphi\in C(I_F)$.
\end{proof}
\begin{lemma}\label{lem:D_cut_implies_mvG}
    If $\lim\limits_{n\to \infty} \Delta_\blacksquare(W_n, W)= 0$ then $\lim\limits_{n\to \infty}W_n = W$ in $\mvGraphons$ (see Definition~\ref{def:mvGraphons}). 
\end{lemma}
\begin{proof}
    Assume that $\lim\limits_{n\to \infty} \Delta_\blacksquare(W_n, W)= 0$. We want to show that $\lim_{n\to \infty}t\ped{d}(F, W_n)=t\ped{d}(F, W)$ for every decorated graph $F$. Since the set of Lipschitz continuous functions is dense in $\Ccal$, by Lemma~\ref{lem:Hom_density_is_point_separating} it is enough to show that $\lim_{n\to \infty}t\ped{d}(F, W_n)=t\ped{d}(F, W)$ for every Lipschitz decorated graph $F$. 
    
    To this end, fix a Lipschitz decorated graph $F$. Let $L>0$ be such that $\max_{\set{i,j}\in E(F)}\norm{\rm BL}{f_{i,j}}\leq L$. Now observe that for any $W\in \mvGraphons$ we have
       $t\ped{d}(F, W) = \int_{[0, 1]^{V(F)}}\prod_{\{i,j\}\in E(F)} \Gamma(F_{i,j}, W)(x_i, x_j)\prod_{v\in V(F)}\diff x_{v}$.
    It follows from above and the Counting Lemma for decorated graphs~\cite[Lemma 10.24]{lovasz2006limits} that 
    \begin{align*}
       \abs{t\ped{d}(F, W_n)-t\ped{d}(F, W)}&\leq 4\sum_{\set{i,j}\in E(F)}\norm{\cut}{\Gamma(F_{i,j}, W_n)-\Gamma(F_{i,j}, W)}\\
       &\leq 4L\sum_{\set{i,j}\in E(F)} \norm{\blacksquare}{W_n - W}= 4L\abs{E(F)}\norm{\blacksquare}{W_n - W}.
    \end{align*}
\sloppy Replacing $W_n$ by $W_n^{\varphi_n}$ and $W$ by $W^{\varphi}$ for any $\varphi_n,\varphi\in \Tcal$ and taking infimum we obtain $\abs{t\ped{d}(F, W_n)-t\ped{d}(F, W)}\leq L|E(F)|\Delta_\blacksquare(W_n, W)$. Since $\Delta_\blacksquare(W_n, W)\to 0$ as $n\to \infty$, it follows that $t\ped{d}(F, W_n)\to t\ped{d}(F, W)$ as $n\to \infty$. 
\end{proof}

\begin{lemma}\label{lem:finitesufficiency}
    Let $\Lcal$ be the space of all Lipschitz function on $[-1, 1]$ with bounded Lipschitz norm at most $1$. For every $\epsilon>0$, there exists a finite set $\Fcal_{\epsilon}\subseteq \Lcal$ such that $\abs{\Delta_\blacksquare(U, V) - \Delta^{\Fcal_{\epsilon}}_{\blacksquare}(U, V)} \leq \epsilon$,
    for all $U, V\in \mvGraphons$,
    where $\Delta^{\Fcal}_{\blacksquare}(U, V)\coloneqq \inf_{\varphi_1, \varphi_2\in \Tcal}\sup_{\psi\in \Fcal}\norm{\cut}{\Gamma(\psi, U^{\varphi_1})-\Gamma(\psi, V^{\varphi_2})}$,
    for any subset $\Fcal\subseteq \Ccal$. Moreover, the set $F_{\epsilon}$ can be chosen so that $\abs{F_{\epsilon}} \leq 3\cdot 2^{16/\epsilon^2}$.
\end{lemma}
\begin{proof}[Proof of Lemma~\ref{lem:finitesufficiency}]
    It is an immediate consequence of Arz\'ela-Ascoli theorem that $\Lcal$ is compact in $\Ccal$. Let $\epsilon>0$ be given. By the compactness of $\Lcal$, there exists a finite subset $\Fcal_{\epsilon}\subseteq \Lcal$ such that union of $\epsilon/2$ balls centered at $\psi\in\Fcal_{\epsilon}$ cover $\Lcal$. In other words, for every $\psi\in \Lcal$ there exists $\psi_0\in \Fcal_{\epsilon}$ such that $\norm{\infty}{\psi-\psi_0}< \epsilon/2$. For any $U, V\in\Wfrak$, by triangle inequality, we obtain
    \begin{align*}
         \abs{\norm{\cut}{\Gamma(\psi, U)-\Gamma(\psi, V)}-\norm{\cut}{\Gamma(\psi_0, U)-\Gamma(\psi_0, V)}}&\leq \norm{\cut}{\Gamma(\psi-\psi_0, U)}+ \norm{\cut}{\Gamma(\psi-\psi_0, V)},
    \end{align*}
    that is strictly bounded by $\epsilon$. It follows that
    \begin{equation}\label{eqn:InftoFinite}
        \abs{\sup_{\psi\in \Lcal}\norm{\cut}{\Gamma(\psi, U)-\Gamma(\psi, V)}-\sup_{\psi\in \Fcal_{\epsilon}}\norm{\cut}{\Gamma(\psi, U)-\Gamma(\psi, V)}}<\epsilon.
    \end{equation}
    Since the above inequality holds for every $U, V\in \Wfrak$, the desired conclusion follows by replacing $U$ and $V$ by $U^{\varphi_1}$ and $V^{\varphi_2}$ respectively and taking infimum over $\varphi_1, \varphi_2\in \Tcal$.
    
    For the second part of the claim, we will construct the finite set of bounded Lipschitz functions, denoted as $F_\epsilon$, as follows: We divide the domain $[-1,1]$ into $4/\eps$ contiguous intervals, each of length $\eps/4$. Given our interest in functions with a Lipschitz constant bounded by $1$, we also partition the range $[-1,1]$ into $4/\epsilon$ contiguous intervals of length $\eps/4$. Observe that any continuous function with bounded Lipschitz norm bounded by $1$, can be approximated to within $\eps$ in the supremum norm using piecewise linear and continuous functions whose local slopes are taken from the set $\set{-1,0,1}$ over the divided domain. Therefore, to define $F_\eps$, it suffices to consider the set of piecewise linear and continuous functions that have local slopes in the set $\set{-1,0,1}$ over the aforementioned partition. By our construction, the size of this set is at most $3\cdot 2^{16/\epsilon^2}$.
\end{proof}

\begin{proof}[Proof of Theorem~\ref{thm:equivalence_of_cut}]
    Equivalence of $(1)$ and $(2)$ follows from Lemma~\ref{lem:OurConvg_implies_LovaszConvg}. Lemma~\ref{lem:D_cut_implies_mvG} shows that $(3)$ (or equivalently $(4)$) implies $(1)$. It remains to show that $(1)$ implies $(3)$. Suppose $(W_n)_{n\in\Natural}\to W$ in $\mvGraphons$ as $n\to\infty$. We want to show that $\lim_{n\to \infty}\Delta_\blacksquare(W_n, W)=0$. 
    
    We will argue by contradiction. Suppose, for contradiction, that there exists some $\epsilon>0$ and some subsequence $(n_k)_{k=1}^{\infty}$ such that $\Delta_\blacksquare(W_{n_k}, W)\geq \epsilon$. By Lemma~\ref{lem:finitesufficiency} there exists a finite family of functions $\mathcal{F}\subseteq \Lcal$ such that $\Delta_\blacksquare(U, V)\leq \Delta^{\Fcal}_{\blacksquare}(U, V)+ \frac{\epsilon}{2}$, for all $U, V\in \mvGraphons$.
  Since $\Fcal$ is finite and $\round{W_{n_k}}_{k\in\Natural}\to W$ as $k\to \infty$ in $\mvGraphons$, it follows from~\cite[Lemma 3.2, Lemma 3.7]{lovasz2010decorated} that $\lim_{k\to \infty}\Delta^{\Fcal}_{\blacksquare}(W_{n_k}, W)=0$. This implies that $\limsup_{k\to \infty} \Delta_\blacksquare(W_{n_k}, W)\leq \epsilon/2$ which is a contradiction. 
\end{proof}

\begin{proof}[Proof of Lemma~\ref{lem:convergence_of_sample_weightedgraph}]
    Let $(F, f)$ be a decorated graph and let $\mathbb{G}(n, W)$ be as defined in Lemma~\ref{lem:convergence_of_sample_weightedgraph}. Recall that $t\ped{d}(F, \mathbb{G}(n, W)) = \frac{1}{n^k}\sum_{i_1, \ldots, i_k} \prod_{\{j,l\}\in E(F)} f_{j,l}(\zeta_{i_j,i_l})$, where $\zeta_{u,v}$s are independent and distributed as $W(U_{u}, U_{v})$ for all $(u,v)$. In particular, $\E{t\ped{d}(F, \mathbb{G}(n, W))}=t\ped{d}(F, W)$ for each $n\in \N$. It suffices to show that $t\ped{d}(F, \mathbb{G}(n, W))$ concentrates around its mean for all decorated graphs $(F, f)$. To this end, fix a decorated graph $(F, f)$ and set $d_n(F)\coloneqq \abs{t\ped{d}(F, \mathbb{G}(n, W))-\E{t\ped{d}(F, \mathbb{G}(n, W))}}$. Using a $4$-th moment bound, following the same argument as in~\cite[Equation 11.5]{lovasz2012large}, we obtain $\Prob{d_n(F)\geq \epsilon}\leq \frac{C}{\epsilon^2 n^2}$. Using Borel-Cantelli Lemma we conclude that $t\ped{d}(F, \mathbb{G}(n, W))\to t\ped{d}(F, W)$ almost surely. To conclude the proof, we observe that set of all finite simple graphs is countable and $\Ccal=C[-1, 1]$ is a separable space. We, therefore, can find a countable dense subset of decorated graphs for which almost sure convergence of homomorphism densities holds. The proof is complete using a standard approximation argument similar to~\cite[Theorem 3.4]{kovacs2014multigraph}.
\end{proof}

\begin{lemma}\label{lem:ComparisonMetric}
    $\Delta_\blacksquare \leq \Delta_2$.
\end{lemma}
\begin{proof} 
    For any $\psi\in\Lcal$ and $W_1, W_2\in \Wfrak$ define
    \[
        V_\psi(x,y) = \int_{-1}^1\psi(\xi) W_1(x,y)(\diff \xi) - \int_{-1}^1\psi(\xi) W_2(x,y)(\diff \xi),
    \]
    for $(x,y)\in[0,1]^{(2)}$. For any $S, T\subseteq [0, 1]$, by the Kantorovich duality and Proposition~\ref{prop:d_D_are_metric}, we observe that
    \begin{align*}
        &\mathbb{W}_1\round{\int_{S\times T}W_1(x,y)\diff x\diff y, \int_{S\times T}W_2(x,y)\diff x\diff y} \nonumber\\
        &= \sup_{\psi\in \Lcal}\abs{\int_{S\times T}V_\psi(x,y)\diff x\diff y} \leq \sup_{\psi\in \Lcal}\int_{[0,1]^2}\abs{V_\psi(x,y)}\diff x\diff y \leq \int_{[0,1]^2}\sup_{\psi\in \Lcal}\abs{V_\psi(x,y)}\diff x\diff y\nonumber\\
        &= \int_{[0,1]^2}\mathbb{W}_1\round{W_1(x,y),W_2(x,y)}\diff x \diff y \leq \int_{[0, 1]^2}\mathbb{W}_2(W_1(x, y), W_2(x, y))\diff x\diff y\nonumber\\
        & \leq \round{\int_{[0,1]^2}\mathbb{W}_2^2\round{W_1(x,y), W_1(x,y)}\diff x \diff y}^{1/2},
    \end{align*}
    where the last inequality follows from the Cauchy-Schwarz inequality. 
The conclusion follows by replacing $W_1, W_2$ by $W_1^{\varphi_1}$ and $W_2^{\varphi_2}$ respectively, and taking infimum over $\varphi_1, \varphi_2\in \Tcal$.
\end{proof}

\subsection{Scaling limit of Metropolis}\label{sec:proof_MH}
\subsubsection{Properties of drift function}
\begin{lemma}\label{lemma:ExplicitDrift}
    Let $Y$ be a standard normal random variable on $\Rd{[r]^{(2)}}$.
    For any $v\in\Rd{[r]^{(2)}}$ and $t>0$ we have
    \[
        \Exp{Y}{Y\exp\round{-t\inner{v,Y}^{+}_{\rm F}}} = -2tv\exp\round{t^2 \normF{v}^2}\overline{\Phi}(\sqrt{2}t\normF{v}).
    \]
\end{lemma}
\begin{proof}
Let $Y$ be as above. Let $\pi \colon y\mapsto \inner{v,y}_{\rm F}$ and let $X=\pi(Y)$.
Note that $X=\inner{v, Y}\sim \Ncal(0, 2\normF{v}^2)$. The factor of $2$ is due to symmetry. Observe that
\begin{align*}
    &\E{Y\exp\round{-t\inner{v,Y}^{+}_{\rm F}}}= \Exp{X}{\exp\round{-t X^{+}}\Exp{Y}{Y\given \inner{v, Y}=X }}\\
    &= \frac{v}{\normF{v}^2}\E{X\exp\round{-t X^{+}}}=\frac{ \sqrt{2} v}{\normF{v}}\E{Z\exp\round{-\sqrt{2}t\normF{v} Z^{+}}},
\end{align*}
where $Z\sim N(0, 1)$ is standard normal random variable. he proof follows by observing that $\E{Z\exp\round{-\alpha Z^{+}}}= -\alpha\exp\round{\frac{1}{2}\alpha^2}\overline{\Phi}(\alpha)$ and taking $\alpha=\sqrt{2}t\normF{v}$.
\end{proof}

\subsubsection{Martingale quadratic variation}
The proof of the following lemma follows from a standard argument using Donsker's invariance theorem and the Lipschitzness of Skorokhod map and is skipped.

\begin{lemma}[Time at boundary of reflected RW]\label{lem:BoundaryTime}
Let $\ell_n=\lceil r^{-4}\sigma^2\gamma_n n^4\rceil$. Fix  $x\in \set{i/n^2\given i=0, \ldots, n^2}$. Let $S$ denote the symmetric random walk with step size $\frac{1}{n^2}$ reflected at $\{0, 1\}$ starting at $x$.  Then, 
\[
    \lim_{n\to \infty} \frac{1}{\gamma_n n^4}\sum_{k=1}^{\ell_n}\indicator{\{0, 1\}}{S_k}=0,\quad \text{in probability.}
\]
\end{lemma}
We now compute the quadratic variation of the martingale $M^{(n)}_r(t)$ defined in Section~\ref{sec:outline}.
\begin{lemma}[Martingale Quadratic Variation]\label{lem:Martingale_QV}
    \sloppy For $r\in\Natural$, $n\in\Natural$ and $t\in\R_+$, let $M^{(n)}_{r}(t)\coloneqq \sum_{\ell=0}^{t_{n, r}-1}\Delta M^{(n)}_{r, \ell}$ where $t_{n,r} = \floor{tr^4/\gamma_n}$ . Then, the quadratic variation of $M_n$ in the time interval $[0,t]$ converges to $t\sigma^2 I$ for all $t\in\R_+$. That is, the following convergence holds in probability:
    \[
        \lim_{n\to \infty}\sum_{\ell=0}^{t_{n, r}-1}\E{\round{\Delta M^{(n)}_{r, \ell, (i,j)}}\round{\Delta M^{(n)}_{r, \ell, (i',j')}}\given \Fcal_\ell} = t\sigma^2\indicator{}{i=i', j=j'},
    \]
    for all $(i,j), (i', j')\in[r]^{(2)}$.
\end{lemma}

\begin{proof}
    We first notice that for each $k\in \N$ we have $\E{\norm{2}{\widetilde{q}^{(n)}_{r,k+1}-q^{(n)}_{r,k}}^2\given \Fcal_k} \leq r^2\gamma_n^2$.
    Let $\Gcal_k$ be the sigma algebra generated by $\Fcal_k \vee\set{ p^{(n)}_{r,k+1}}$. Recall that given $p^{(n)}_{r,k+1}$, the iterate $q^{(n)}_{r,k+1}$ is obtained by running $\ell_n$ steps of independent symmetric random walk with step size $\frac{1}{n^2}$ (with reflection at $\{0, 1\}$) starting at $p^{(n)}_{r,k+1}$. Fix $(i,j)\in[r]^{(2)}$. Let $S_k$ denote the symmetric random walk with step size $1/n^2$ run for $m$ steps starting at $p^{(n)}_{r,k+1,(i,j)}$. Now observe that
    \begin{align*}
       \E{\round{q^{(n)}_{r,k+1,(i,j)}-p^{(n)}_{r,k+1,(i,j)}}^2\given \Gcal_k} &= \frac{1}{n^4}\sum_{m=1}^{\ell_n} \indicator{(0, 1)}{S_{k, m}}+\frac{1}{2n^4}\sum_{m=1}^{\ell_n} \indicator{\{0, 1\}}{S_{k,m}}\\
       &=\frac{\ell_n}{n^4} - \frac{1}{2n^4}\sum_{m=1}^{\ell_n} \indicator{\{0, 1\}}{S_{k,m}}.
    \end{align*}
    Set $h^{(n)}_{k} = \frac{1}{2n^4}\sum_{m=1}^{\ell_n} \indicator{\{0, 1\}}{S_m}$. Note that $\lim_{n\to \infty}\sum_{m=0}^{t_{n, r}-1}\frac{\ell_n}{n^4} = t\sigma^2$. It follows that 
    \begin{align*}
        \lim_{n\to \infty} \abs{  \sum_{\ell=0}^{t_{n, r}-1}\E{\round{\Delta M^{(n)}_{k, \ell, (i,j)}}^2} - t\sigma^2} &\leq \lim_{n\to \infty}r^2\gamma_n^2t_{n, r}+\lim_{n\to \infty} \sum_{\ell=0}^{t_{n, r}-1} h^{(n)}_{k}. 
    \end{align*}
    It is clear that $r^2\gamma_n^2t_{n, r}\to 0$ as $n\to \infty$, and $\lim_{n\to \infty} \sum_{k=0}^{t_{n, r}-1} h^{(n)}_{k} = 0$ by Lemma~\ref{lem:BoundaryTime}. 

    For simplicity define $\widehat\Delta q^{(n)}_{r, k, (i, j)}\coloneqq q^{(n)}_{r, k+1, (i, j)}-p^{(n)}_{r, k+1, (i, j)}$. If $\{i, j\}\neq \{i', j'\}$ then  $\widehat\Delta q^{(n)}_{r, k, (i, j)}$ and $\widehat\Delta q^{(n)}_{r, k, (i', j')}$ are independent given $\Gcal_{k}$. In particular, 
    \begin{align*}
        \E{\widehat\Delta q^{(n)}_{r, k, (i, j)}\widehat\Delta q^{(n)}_{r, k, (i', j')}\given \Gcal_{k}}&= \E{\widehat\Delta q^{(n)}_{r, k, (i, j)}\given \Gcal_{k}}\E{\widehat\Delta q^{(n)}_{r, k, (i', j')}\given \Gcal_{k}}\\
        &\leq \frac{1}{n^4}\sum_{m=1}^{\ell_n} \indicator{\{0, 1\}}{S_{k,m}},
    \end{align*}
    where $S_{k, m}$ is as above. Using Lemma~\ref{lem:BoundaryTime} we conclude that 
     \begin{align*}
        \lim_{n\to \infty} \abs{  \sum_{\ell=0}^{t_{n, r}-1}\E{\Delta M^{(n)}_{k, \ell, (i,j)}\Delta M^{(n)}_{k, \ell, (i',j')}}} &\leq \lim_{n\to \infty} \sum_{\ell=0}^{t_{n, r}-1} h^{(n)}_{k}=0. 
    \end{align*}
    This completes the proof.
\end{proof}

\subsubsection{Away from boundary}
In the following, we denote by $S = \round{S_k}_{k\in\Integer_+}$ a standard simple symmetric random walk. Recall the KMT embedding theorem~\cite{komlos1975approximation} which states that one can couple $S$ with some Brownian motion $B$ such that 
\[
    \Prob{\max_{0\leq k\leq T}\frac{\abs{S_k-B(k)}}{n^2}\geq C\frac{\log T+x}{n^2}}\leq \eu^{-x},
\]
for any $T\in\Natural$. 
Taking $T=s_n$ (and $\ell_n$ respectively), we obtain that for $n$ sufficiently large we have 
\begin{align*}
    \Prob{\max_{0\leq k\leq s_n}\frac{\abs{S_k-B(k)}}{n^2}\geq \frac{C\log n}{n^2}}&\leq
    \Prob{\max_{0\leq k\leq \ell_n}\frac{\abs{S_k-B(k)}}{n^2}\geq \frac{C\log n}{n^2}}\leq \frac{1}{n^{4}}.
\end{align*}
Further observe that for a fixed $\delta>0$, we have that 
\begin{align*}
    \Prob{\max_{0\leq t\leq s_n/n^4} \abs{B(t)}\geq \delta} &\leq 
    \Prob{\max_{0\leq t\leq \ell_n/n^4} \abs{B(t)}\geq \delta} \leq 2\overline{\Phi}\round{\frac{\delta}{r^{-2}\sigma\sqrt{\gamma_n}}}.
\end{align*}
We combine these observations to obtain the following lemma.
\begin{lemma}\label{lem:Hitting}
Let $\widetilde{S}_k = \frac{1}{n^2}S_{k}$ for every $k\in\Integer_+$. Let $\epsilon>0$ be fixed. Then, for all $n\in\Natural$ sufficiently large, we have 
\begin{align*}
    \Prob{\max_{k\leq s_n}\abs{\widetilde{S}_k}\geq \epsilon/2} &\leq 
    \Prob{\max_{k\leq \ell_n}\abs{\widetilde{S}_k}\geq \epsilon/2} \leq \frac{1}{n^{4}} + 2\overline{\Phi}\round{\frac{\epsilon}{4r^{-2}\sigma\sqrt{\gamma_n}}}.
\end{align*}
\end{lemma}

\begin{lemma}\label{lem:DeltaTildereduction}
    Let $\epsilon>0$ fixed. Let $\ell\in\Integer_+$ be such that $q^{(n)}_{r, \ell}\in A_{\epsilon}$. Then, for $n$ sufficiently large, we have 
    \[
       \normF{ \E{\Delta q^{(n)}_{r, \ell}\given \Fcal_{\ell}}-\E{\widetilde{\Delta} q^{(n)}_{r, \ell}\given \Fcal_{\ell}}}^2\leq 2\round{\frac{r^2}{n^4}+2r^2\overline{\Phi}\round{\frac{\epsilon}{4r^{-2}\sigma\sqrt{\gamma_n}}}},
    \]
    with probability at least $1-\frac{r^2}{n^4}-2r^2\overline{\Phi}\round{\frac{\epsilon}{4r^{-2}\sigma\sqrt{\gamma_n}}}$.
\end{lemma}
\begin{proof}
    Let $\epsilon>0$, $r, \ell$ be fixed. Let $\widehat{\Delta}q^{(n)}_{r, \ell} \coloneqq q^{(n)}_{r, \ell+1}-\widetilde{q}^{(n)}_{r, \ell+1}$. Begin by observing that
    \[
        \normF{ \E{\Delta q^{(n)}_{r, \ell}\given \Fcal_{\ell}}-\E{\widetilde{\Delta} q^{(n)}_{r, \ell}\given \Fcal_{\ell}}}^2=\norm{2}{\E{\widehat{\Delta}q^{(n)}_{r, \ell}\given \Fcal_{\ell}}}^2.
    \]
    Let ${E}_{n, \ell}$ be the event that $\widetilde{q}^{(n)}_{r, \ell+1}\in A_{\epsilon/2}$. Using Lemma~\ref{lem:Hitting} and union bound we conclude that 
    \[
        \Prob{\widetilde{E}_{n, \ell}}\geq 1-\frac{r^2}{n^{4}}-2r^2\overline{\Phi}\round{\frac{\epsilon}{4r^{-2}\sigma\sqrt{\gamma_n}}}.
    \]
    Given $p^{(n)}_{r, \ell+1}$, we observe that $\widehat{\Delta}q^{(n)}_{r, \ell}$ has the same distribution as symmetric random walk with step-size $\frac{1}{n^2}$ (reflected at boundary $\set{0,1}$) run for $\ell_{n,r}$ steps. Let us denote this $j$-th step of this walk by $S_{k, j}$. Also define a simple random walk with step-size $\frac{1}{n^2}$ (without reflection) $\widetilde{S}_{k}$ starting with the same initial condition as $S_{k}$.
    Given $\widetilde{q}^{(n)}_{r, \ell+1}\in A_{\epsilon/2}$, we can couple the walk $S_{k}$ and $\widetilde{S}_k$ so that $S_{k, j}=\widetilde{S}_{k, j}$ for all $j\leq T$ where $T=\min\set{i \in\Integer_+ \given \abs{\widetilde{S}_{k, i}-\widetilde{S}_{k, 0}}\geq \epsilon/2}$.
    That is, we couple the two walks so that they are equal till they move at least $\epsilon/2$ distance from the starting position. Now notice that 
    \begin{align*}
        \E{\widehat{\Delta}q^{(n)}_{r, \ell}\given \Gcal_{\ell}}= \E{\widehat{\Delta}q^{(n)}_{r, \ell}\indicator{T\leq \ell_n}{}\given \Gcal_{\ell}}+\E{\widehat{\Delta}q^{(n)}_{r, \ell}\indicator{T>\ell_n}{}\given \Gcal_{\ell}}.
    \end{align*}
    using the bound $\widehat{\Delta}q^{(n)}_{r, \ell}\leq 1$ and Lemma~\ref{lem:Hitting} we have that
    \[
        \normF{\E{\widehat{\Delta}q^{(n)}_{r, \ell}\indicator{T\leq \ell_n}{}\given \Gcal_{\ell}}}^2\leq \frac{r^2}{n^4}+2r^2\overline{\Phi}\round{\frac{\epsilon}{4r^{-2}\sigma\sqrt{\gamma_n}}}.
    \]
    On the other hand, using the fact that $\E{\widetilde{S}_{k, \ell_n}\given \Gcal_\ell}=0$, we obtain
    \begin{align*}
        \normF{\E{\widehat{\Delta}q^{(n)}_{r, \ell}\indicator{T>\ell_n}{}\given \Gcal_{\ell}}}^2&=\normF{\E{\widetilde{S}_{k,\ell_n}\indicator{T>\ell_n}{}\given \Gcal_{\ell}}}^2\\
        &=\normF{\E{\widetilde{S}_{k, \ell_n}\indicator{T>\ell_n}{}\given \Gcal_{\ell}}-\E{\widetilde{S}_{k, \ell_n}\given \Gcal_{\ell}}}^2\\
        &=\normF{\E{\widetilde{S}_{k, \ell_n}\indicator{T\leq \ell_n}{}\given \Gcal_{\ell}}}^2\leq r^{-2}\sigma^2\gamma_n \round{\frac{r^2}{n^4}+2r^2\overline{\Phi}\round{\frac{\epsilon}{4r^{-2}\sigma\sqrt{\gamma_n}}}}.
    \end{align*}
    Thus, we conclude that for $n$ sufficiently large we have  
    \[
        \normF{\E{\widehat{\Delta}q^{(n)}_{r, \ell}\given \Fcal_{\ell}}}^2\leq 2\round{\frac{r^2}{n^4}+2r^2\overline{\Phi}\round{\frac{\epsilon}{4r^{-2}\sigma\sqrt{2\gamma_n}}}}.
    \]
    completing the proof.
\end{proof}

\begin{lemma}\label{lem:upper_bound_quadratic_error}
    Let $r\in\Natural$, for any  $k\in\Integer_+$, let $q^{(n)}_{r,k}$ and $\widetilde{q}^{(n)}_{r,k+1}$ be as defined in Step~\ref{step1} of our algorithm in Section~\ref{sec:relax_Metro}. Then, there exists a universal constant $c>0$ such that for all $n\in\Natural\setminus\squarebrack{\ceil{\eu^{cr/4}}}$, we have
    \[
        \enorm{K\round{\widetilde{q}^{(n)}_{r,k+1}} - K\round{q^{(n)}_{r,k}}}^2 \leq 4\gamma_n^2 + (16/c)\gamma_n^2\log n\eqqcolon e_n \leq (32/c)\gamma_n^2\log n,
    \]
    with probability at least $1-2n^{-4}$.
\end{lemma}
\begin{proof}
    For every $i,j\in[r]$, let $\round{S_{i,j,k}}_{k\in\Integer_+}$ denote the $1$-dimensional symmetric random walk with step-size $\frac{1}{n^2}$ starting at $0$. Let these random walks be independent up to the double index symmetry for indices $(i,j)\in[r]^{(2)}$. Recall that $s_n = \ceil{\gamma_n^2 n^4}$, and note that  for any $i, j\in [r]$, given $q^{(n)}_{r,k,(i,j)}$ we have
    \[
        \widetilde{q}^{(n)}_{r,k+1,(i, j)}-q^{(n)}_{r,k,(i, j)}\stackrel{\rm d}{=} \Sko{S_{s_n}}.
    \]
    Since the Skorokhod map is $4$-Lipschitz, we conclude that
    \begin{align}
       \Prob{\enorm{K\round{\widetilde{q}^{(n)}_{r,k+1}} - K\round{q^{(n)}_{r,k}}}^2 \geq e_n}&\leq \Prob{\inv{r^2}\sum_{(i,j)\in[r]^{(2)}}\round{S_{i,j,s_n}}^2\geq e_n/4}.\label{eq:skorokhod_probability_upper_bound}
    \end{align}
    We will now show that the quantity $\inv{r^2}\sum_{i,j\in[r]^{(2)}}\round{S_{i,j,s_n}}^2$ is concentrated near its expectation, that is $s_n\cdot \round{\inv{n^2}}^2$. From the Hanson-Wright concentration inequality~\cite{rudelson2013hanson},
    \begin{align}
        &\Prob{ \abs{ \inv{r^2}\sum_{i,j\in[r]^{(2)}}\round{S_{i,j,s_n}}^2 - \gamma_n^2} > t_n} \leq 
        \Prob{ \abs{ \inv{r^2}\sum_{i,j\in[r]^{(2)}}\round{S_{i,j,s_n}}^2 - s_n n^{-4}} > t_n}\nonumber\\
        &\leq 2\exp\round{-c\min\set{\frac{t_n^2}{\round{\inv{n^2}}^4rs_n^2},\frac{t_n}{\round{\inv{n^2}}^2s_n}}}
        \leq 2\exp\round{-c\min\set{\frac{t_n^2}{r\gamma_n^4},\frac{t_n}{\gamma_n^2}}},
    \end{align}
    for every $t_n\geq 0$, for some universal constant $c>0$. Let us consider $t_n\geq r\gamma_n^2$. Then, the above probability becomes $2\exp\round{-ct_n/\gamma_n^2}$. Moreover, for $n\geq \eu^{cr/4}$ if we choose $t_n = (4/c)\gamma_n^2\log n$, we have that for all $(i,j)\in[r]^{(2)}$,
    \[
        \inv{r^2}\sum_{i,j\in[r]^{(2)}}\round{S_{i,j,s_n}}^2 \leq \gamma_n^2 + (4/c)\gamma_n^2\log n,
    \]
    with probability at least $1-2n^{-4}$,
    for $e_n\coloneqq 4\gamma_n^2 + (16/c)\gamma_n^2\log n \leq (32/c)\gamma_n^2\log n$.
\end{proof}

\begin{lemma}\label{lem:Drift_AwayfromBoundary}
Let $\epsilon>0$ be fixed, $e_n$ be as defined in Lemma~\ref{lem:upper_bound_quadratic_error}, and let $q^{(n)}_{r, \ell}\in A_{\epsilon}$ where $A_\epsilon$ is defined in~\eqref{eq:A_eps}. Then, 
\[
    \normF{\E{\Delta q^{(n)}_{r, \ell}\given \Fcal_{\ell}} - \gamma_n r^{-4}b_{r}\round{q^{(n)}_{r, \ell}}}^2\leq \frac{Cr^2}{n^4}+  4r^2\beta_{n, r}^2e_n^3\max\set{\lambda,L}+2r^2\overline{\Phi}\round{\frac{\epsilon}{4r^{-2}\sigma\sqrt{\gamma_n}}},
\]
with probability at least $1-\frac{2}{n^4}-\frac{2r^2}{n^{4}}-4r^2\overline{\Phi}\round{\frac{\epsilon}{4r^{-2}\sigma\sqrt{\gamma_n}}}$.
\end{lemma}

    \begin{proof}

Let $I\coloneqq \enorm{K\left(\widetilde{q}^{(n)}_{r, \ell+1}\right) - K\left(q^{(n)}_{r, \ell}\right)}^2$ and let $A_{n, \ell}$ be the event that $\{I\leq e_n\}$. In the following we will work on this event. Set 
\[
    J\coloneqq \frac{\exp\round{-\beta_{n, r}\squarebrack{\hamil\round{K\round{\widetilde{q}^{(n)}_{r, \ell}}} - \hamil\round{K\round{q^{(n)}_{r, \ell}}}}^+}}{\exp\round{-\beta_{n, r}\inner{D \hamil\round{K\round{q^{(n)}_{r, \ell}}}, K\left(\widetilde{q}^{(n)}_{r, \ell+1}\right) - K\left(q^{(n)}_{r, \ell}\right)}^+}},
\]
From our assumption on $\round{\gamma_n}_{n\in\Natural}$, we have that for sufficiently large $n$, $\beta \gamma_n\log^2 n\leq 1$. Notice that by Assumption~\ref{asmp:hamil}, we have
\[
    1-2\beta_{n,r}\lambda I \leq \exp\round{-\beta_{n,r}\lambda I} \leq J \leq \exp\round{\beta_{n,r}LI} \leq 1+2\beta_{n,r}LI,
\]
if $\beta_{n,r}\lambda I, \beta_{n,r}L I\leq 1$, i.e., when $n$ is sufficiently large.
Define $b^{(n)}_r$ at $q^{(n)}_{r, \ell}$ as
\begin{align*}
    \E{\round{\widetilde{q}^{(n)}_{r, \ell+1}-q^{(n)}_{r, \ell}}\exp\round{-\beta_{n, r}\inner{D \hamil\round{K\round{q^{(n)}_{r, \ell}}}, K\left(\widetilde{q}^{(n)}_{r, \ell+1}\right) - K\left(q^{(n)}_{r, \ell}\right)}^+}\given \Fcal_k}.
\end{align*}
Then, on the event $A_{n, \ell}$ we have 
$\normF{\E{\widetilde{\Delta} q^{(n)}_{r, \ell} \given \mcal{F}_k} -{b}^{(n)}_r\round{q^{(n)}_{r, \ell}}}^2 \leq 4r^2\beta_{n, r}^2e_n^3\max\set{\lambda,L}$.


Let $E_{n,\ell}$ be the event as in the  proof of Lemma~\ref{lem:Hitting}. On this event, we have
\begin{equation}\label{eqn:Delta_Delta}
   \normF{ \E{\Delta q^{(n)}_{r, \ell}\given \Fcal_{\ell}}-\E{\widetilde{\Delta} q^{(n)}_{r, \ell}\given \Fcal_{\ell}}}^2\leq C\round{\frac{r^2}{n^4}+2r^2\overline{\Phi}\round{\frac{\epsilon}{4r^{-2}\sigma\sqrt{\gamma_n}}}}.
\end{equation}
Moreover, on the event $E_{n,\ell}$ we also have that given $\Fcal_{k}$, the coordinates of the $r\times r$ symmetric matrix $\round{\widetilde{q}^{(n)}_{r, \ell+1}-q^{(n)}_{r, \ell}}$ are i.i.d. and have the same distribution as $\widetilde{S}_{s_n}$.

Let $\widetilde{Y}_n$ be $r\times r$ matrix with independent entries such that $\widetilde{Y}_{n,(i, j)}$ be increment of the symmetric random walk (without reflection) of step-size $n^{-2}$ starting from $q^{(n)}_{r,\ell,(i,j)}$ run for $s_n=\ceil{\gamma_n^2n^{4}}$ steps. Let $B_r$ be an $r\times r$ symmetric matrix of standard Brownian motions. On the event $E_{n, \ell}\cap A_{n, \ell}$, we use the Berry-Esseen lemma (see~\cite[Theorem 16]{petrov2012sums}) with a union bound to obtain
   $\W_2^{2}\round{\widetilde{Y}_n, B_r(\gamma_n^2)}\leq \frac{Cr^2}{n^4}$,
for some universal constant $C>0$.

Let $\nabla H_r\round{q^{(n)}_{r, \ell}}=V$. Define a function $G(Y)\coloneqq Y\exp\round{-\beta_{n, r}\inner{V, Y}_{\rm F}^{+}}$. Note that $G$ is a bounded Lipschitz function of $Y$. Observe that 
$b^{(n)}_{r}\round{q^{(n)}_{r, \ell}} = \E{G\round{\widetilde{Y}_n}}$.
On the other hand, we know that $\gamma_n r^{-4}b_{r}\round{q^{(n)}_{r, \ell}} = \E{G(B_r(\gamma_n))}$.
We conclude that on the event $E_{n, \ell}\cap A_{n, \ell}$ we have 
\begin{align*}
    \normF{\E{\widetilde{\Delta} q^{(n)}_{r, \ell}\given \Fcal_{\ell}} - \gamma_n r^{-4}b_{r}\round{q^{(n)}_{r, \ell}}}^2 \leq \frac{Cr^2}{n^4}+  4r^2\beta_{n, r}^2e_n^3\max\set{\lambda,L}.
\end{align*}
The conclusion follows by using~\eqref{eqn:Delta_Delta} and noticing that the
\[
    \Prob{E_{n, \ell}\cap A_{n, \ell}}\geq 1-\frac{2}{n^4}-\frac{2r^2}{n^{4}}-4r^2\overline{\Phi}\round{\frac{\epsilon}{4r^{-2}\sigma\sqrt{\gamma_n}}}.
\]
\end{proof}

\bibliographystyle{alpha} 
\bibliography{references}


\end{document}